\documentclass[10pt]{article}

\usepackage{fullpage}

\usepackage{graphicx} 
\usepackage{amsmath}
\usepackage{amssymb}
\usepackage{mathtools}
\usepackage{todonotes}
\usepackage{hyperref}
\usepackage{enumerate}

\usepackage{amsthm,amsfonts,mathrsfs,latexsym}

\newcommand{\bfa}{\mathbf{a}}
\newcommand{\bfb}{\mathbf{b}}

\newcommand{\bfz}{\mathbf{z}}

\newcommand{\dDelta}{\Delta} 
\newcommand{\hite}{\beta}
\renewcommand{\Im}{\mathrm{Im}}
\newcommand{\intZ}{\mathbb{Z}}
\newcommand{\Li}{\mathrm{Li}}
\newcommand{\prob}{\mathbb{P}}
\renewcommand{\Re}{\mathrm{Re}}
\newcommand{\realR}{\mathbb{R}}
\newcommand{\rmd}{\mathrm{d}}
\newcommand{\rmi}{\mathrm{i}}
\newcommand{\rmL}{\mathrm{L}}
\newcommand{\rmR}{\mathrm{R}}
\newcommand{\ut}{{\hat u}}

\newtheorem{thm}{Theorem}[section]
\newtheorem{prop}[thm]{Proposition}
\newtheorem{lemma}[thm]{Lemma}
\newtheorem{cor}[thm]{Corollary}

\newtheorem{defn}[thm]{Definition}

\newcommand{\ratio}{\mathsf{r}}  

\newcommand{\beq}{ \begin{equation} }
\newcommand{\eeq}{ \end{equation} }
\newcommand{\beqq}{ \begin{equation*} }
\newcommand{\eeqq}{ \end{equation*} }
\newcommand{\ben}{ \begin{enumerate}[(a)] }
\newcommand{\een}{ \end{enumerate} }
\newcommand{\beni}{ \begin{enumerate}[(i)] }
\newcommand{\eeni}{ \end{enumerate} }
\newcommand{\bep}{ \begin{proof} }
\newcommand{\eep}{ \end{proof} }

\newcommand{\ii}{\mathrm{i}}
\newcommand{\C}{\mathbb{C}}
\newcommand{\dd}{\mathrm{d}}

\newcommand{\bfw}{\mathbf{w}}

\newcommand{\bfk}{\mathbf{k}}

\newcommand{\bfy}{\mathbf{y}}
\newcommand{\bft}{\mathbf{t}}
\newcommand{\bfh}{\mathbf{h}}
\newcommand{\bfx}{\mathbf{x}}
\newcommand{\bfn}{\mathbf{n}}

\newcommand{\kpz}{\mathsf{H}}
\newcommand{\pkpz}{\mathcal{H}^{(per)}}
\newcommand{\ppkpz}{\mathcal{H}}
\newcommand{\equid}{\stackrel{d}{=}}
\newcommand{\tasep}{\mathbf{h}}
\newcommand{\tasepF}{\mathbf{F}}
\newcommand{\bm}{\mathsf{B}}
\newcommand{\bb}{\mathbb{B}^{\mathrm{br}}}

\newcommand{\N}{\mathbb{N}}
\newcommand{\bfone}{\mathbf{1}}
\newcommand{\bfc}{\mathbf{c}}
\newcommand{\ff}{\mathsf{G}_{\ell}}
\newcommand{\fff}{\mathsf{G}}

\newcommand{\Cd}{\mathsf{K}}  		

\newcommand{\uu}{\mathsf{u}}        
\newcommand{\UU}{\mathsf{U}}        
\newcommand{\PP}{P}        
\newcommand{\cstar}{c_{*}}        
\newcommand{\bfzero}{\mathbf{0}}
\newcommand{\bT}{T^{\bullet}}     
\newcommand{\bC}{C^{\bullet}}     
\newcommand{\bD}{D^{\bullet}}     
\newcommand{\hbD}{\hat D^{\bullet}}     
\newcommand{\hbs}{\hat s^{\bullet}}     
\newcommand{\bs}{s^{\bullet}}     

\newcommand{\pp}{q}
\newcommand{\qq}{p}

\newcommand{\cnst}{c}

\newcommand{\dist}{\mathrm{dist}}

\newcommand{\Sr}{I_\ratio}
\newcommand{\bmcir}{\mathsf{B}^{\Sr}}
\newcommand{\bbcir}{\mathbb{B}^{\mathrm{br};\Sr}}
\def\fddto{\xrightarrow{\textit{f.d.d.}}}
\def\law{{\rm Law}}

\title{Pinched-up periodic KPZ fixed point}

\author{Jinho Baik\footnote{Department of Mathematics, University of Michigan,
Ann Arbor, MI, 48109, USA, \texttt{baik@umich.edu}} \and Zhipeng Liu\footnote{Department of Mathematics, University of Kansas,
Lawrence, KS 66045, USA, \texttt{zhipeng@ku.edu}}}

\date{\today}

\begin{document}

\maketitle

\begin{abstract}
The periodic KPZ fixed point is the conjectural universal limit of the KPZ universality class models on a ring when both the period and time critically tend to infinity. For the case of the periodic narrow wedge initial condition, we consider the conditional distribution when the periodic KPZ fixed point is unusually large at a particular position and time. 
We prove a conditional limit theorem up to the ``pinch-up" time. 
When the period is large enough, the result is the same as that for the KPZ fixed point on the line obtained by Liu and Wang in 2022. 
We identify the regimes in which the result changes and find probabilistic descriptions of the limits. 

\end{abstract}



\section{Introduction and main results}

The KPZ fixed point is a universal two-dimensional random field \cite{Matetski-Quastel-Remenik21} to which the height functions of many random growth models on the line are expected to converge in the large-time limit. 
Among various properties found for the KPZ fixed point (see, for example, \cite{Ferrari-Occelli18, Johansson20, Basu-Ganguly21, Liu22c, Liu22b, Quastel-Remenik22, Baik-Prokhorov-Silva23} and references therein)  is the recent study on conditional distributions when the field is uncharacteristically large at a specific position and time \cite{Liu-Wang22}. 
This paper aims to study similar conditional distributions for the periodic KPZ fixed point, which arises as the universal limit for random growth models on a ring. The size of the ring affects the field, and the interest is to determine the effect of domain size on the conditional distribution. 
We first review a result for the ``pinched-up" KPZ fixed point and then introduce the periodic KPZ fixed point.

\subsection{KPZ fixed point when it is pinched-up} \label{sec:KPZresult}

Let $\kpz(x,t)$ denote the KPZ fixed point with the narrow wedge initial condition. Consider the situation when $\kpz(0,1)=L$ is large.
It was shown  in \cite{Nissim-Zhang22} that conditional on $\kpz(0,1)=L$, the one point distribution of $\kpz(x, t) - L$ converges to a properly scaled Tracy-Widom distribution 
for every fixed  $(x, t)\in \realR\times (1, \infty)$ as $L\to \infty$.
This is consistent with the intuition that the conditioning makes the shifted height $\kpz(x, 1) - L$ close to the narrow wedge, and thus, from the Markovian property, the pinched-up process after time $t=1$ should look again like the KPZ fixed point with the narrow wedge initial condition, starting at $t=1$. 
On the other hand, for $t\in (0,1)$, the following result was proven.

\begin{thm}[\cite{Liu-Wang22}] \label{resultLiuWang}
Let $\kpz(x,t)$ be the KPZ fixed point with the narrow wedge initial condition. Let $\mathbb{B}_1$ and $\mathbb{B}_2$ be independent Brownian bridges. Then 
\beq \label{eq:kpzLiuWang}
	\law \left(\left\{\frac{\kpz(\frac{x}{2L^{1/4}}, t)- t L}{ L^{1/4}} \right\}_{(x,t)\in\realR\times (0,1)}\, \Big|\,\kpz(0, 1)=L \right) \fddto \law \left(\left\{ \bb_2(t) - \left| \bb_1(t) -x\right| \right\}_{(x,t)\in\realR\times (0,1)}\right)
\eeq
as $L\to \infty$, where $\fddto$ denotes the convergence of finite dimensional distributions, and the conditional law should be understood as $\prob(\cdot \mid \kpz(0, 1)=L) = \lim_{\epsilon\to 0}\prob(\cdot \mid \kpz(0,1)\in (L-\epsilon,L+\epsilon))$.
\end{thm}

Here, the term ``Brownian bridges'' means standard Brownian bridges, that is, the duration of time is $1$, and they return to $0$ at time $1$. We use the same convention throughout this paper. We also use the term ``Brownian motion'' to refer to a standard Brownian motion which starts at $0$. We emphasize that our Brownian motions and Brownian bridges could be defined on two different metric spaces in this paper, the real line $\realR$ or the quotient space $\Sr=\realR/\ratio\intZ$, where $\ratio$ is a positive number. We will use $\bm$, $\bb$ to denote the Brownian motions and Brownian bridges on the real line, and $\bmcir$, $\bbcir$ the  Brownian motions and Brownian bridges on $\Sr$ respectively.

The original formula in \cite{Liu-Wang22} was given in terms of $(\bb_1(t)-x)\wedge (\bb_2(t)+x)$. Here we used the identity $a\wedge b =\frac{1}{2}(a+b) -\frac12|a-b|$, and the invariance of the law of two independent Brownian bridges under orthonormal transformations to rewrite it as the form in \eqref{eq:kpzLiuWang}.

The KPZ fixed point with the narrow wedge initial condition satisfies the invariance properties
\beq \label{eq:KPZinvariance}
	\alpha \kpz(\alpha^{-2} x, \alpha^{-3} t) \equid \kpz(x,t)
	\quad \text{and} \quad 
	 \kpz(x, t) \equid \kpz(x+ \beta t, t) + \frac1{t} \left( (x+\beta t)^2-x^2 \right)
\eeq
for every $\alpha>0$ and $\beta\in \realR$. 
Thus, \eqref{eq:kpzLiuWang} also implies a result when the conditioning is given at a general point $(X, T)$ instead of $(0,1)$ (see \cite[Remark 1.5]{Liu-Wang22}): 
\beq \label{eq:kpzXaway}
	\law \left(\left\{\frac{\kpz( tX + \frac{x T^{3/4}}{2L^{1/4}}, tT )- t L}{ T^{1/4} L^{1/4}}\right\}_{(x,t)\in\realR\times (0,1)}\, \Big|\,\kpz(X, T)=L \right) \fddto \law \left(\left\{ \bb_2(t) - \left| \bb_1(t) -x\right| \right\}_{(x,t)\in\realR\times (0,1)}\right)
\eeq 
as $L\to \infty$.

The papers \cite[Theorem 1.9]{Ganguly-Hedge22} and \cite{Lamarre-Lin-Tsai23} also considered conditional limit theorems and obtained the first-order term and concentration results. 
The result \eqref{eq:kpzLiuWang} has an implication on the geodesics in the directed landscape as well. Based on the above result, the authors of \cite{Liu-Wang22} conjectured that conditional on $\kpz(0,1)=L$ goes to infinity, the geodesic converges to the Brownian bridge. 
This conjecture was recently proved by \cite{Ganguly-Hedge-Zhang23} using geometric and probabilistic methods. 
These results show that the geodesic between $(0,0)$ to $(0,1)$ typically stays within the distance of order $L^{-1/4}$ from the straight line.

\subsection{Periodic KPZ fixed point} \label{sec:pKPZfp}

Let $\tasep(n,t)$ be the height function of the totally asymmetric simple exclusion process (TASEP) on the discrete ring of size $2a$. 
We identify the ring as the set $\{-a+1, \cdots, a\}$ and extend the TASEP periodically on the integers $\intZ$  by setting $\tasep(n\pm 2a, t)=\tasep(n, t)$. 
We may call the extended TASEP a periodic TASEP of period $2a$. 
Suppose that initially, $\tasep(n,0)= |n|$ for $-a+1\le n\le a$ and is extended periodically. 
This initial condition is called the periodic step initial condition. 

An interesting large time limit arises when the period $2a$ is proportional to $t^{2/3}$, which is called a relaxation time scale. 
In this limit, the ring size  affects the fluctuations of the height function nontrivially.
It was shown\footnote{In \cite{Baik-Liu19}, the case when $\gamma_i\neq \gamma_{i'}$, $\tau_i=\tau_{i'}$ and $\hite_i=\hite_{i'}$ for some $i\neq i'$ was not analyzed. See Appendix  \ref{sec:appendix} how we can obtain the result in this case.} in \cite{Baik-Liu19} that for every positive integer $m$,  
for every $m$ distinct points $(\gamma_i, \tau_i)\in \realR \times \realR_+$ and every $m$ real numbers $\hite_i$, $i=1, \cdots, m$, 
\beq \label{eq:taseplimit}
	\lim_{T= (2a)^{3/2} \to \infty} \prob\left(    \bigcap_{i=1}^m   \left\{ 
            \frac{\tasep(\gamma_i T^{2/3}, 2\tau_i T) - \tau_i T}{-T^{1/3}}\le \hite_i   \right\}
           \right)
	= \tasepF_m(\hite; \gamma, \tau) 
\eeq
converges, where $\hite=(\hite_1, \cdots, \hite_m)$, $\tau= (\tau_1, \cdots, \tau_m)$, and $\gamma=(\gamma_1, \cdots, \gamma_m)$. 
The function $\tasepF_m(\hite; \gamma, \tau)$ is periodic with the shift $\gamma_i\mapsto \gamma_i+1$ for any $i$, and can be extended continuously for  $(\gamma, \tau, \hite)\in \realR^m \times \realR_+^m \times \realR^m$. 
The functions $\tasepF_m$, $m=1, 2, \cdots$, form a consistent collection of multivariate cumulative distribution functions.
See Section \ref{sec:formdf} and Appendix  \ref{sec:appendix} for the formula and properties of these functions.

Let $\pkpz(\gamma, \tau)$, $(\gamma, \tau)\in \realR\times \realR_+$, be a process whose law is defined by the collection $\tasepF_m$. 
It satisfies the spatial periodicity 
\beqq
	\pkpz(\gamma+1, \tau)= \pkpz(\gamma, \tau) .
\eeqq 
Since we will only discuss finite-dimensional properties in this paper, we will simply call it the periodic KPZ fixed point with the periodic narrow wedge initial condition.  
The limit \eqref{eq:taseplimit} was also proved for the discrete-time TASEP \cite{Liao22} and the PushASEP \cite{Li-Saenz23} on a ring. 
The distribution functions $\tasepF_m$ are expected to be the universal limits for the multi-time, multi-position distributions of the KPZ universality class in the periodic domain at the relaxation time scale. 
The convergence \eqref{eq:taseplimit} was also extended to other initial conditions satisfying some technical assumptions \cite{Baik-Liu21}. 
Since we will only consider the periodic narrow wedge initial condition case in this paper and leave other initial conditions for future consideration, we will simply call $\pkpz(\gamma, \tau)$ the periodic KPZ fixed point without mentioning the initial condition.

Unlike the KPZ fixed point, the periodic KPZ fixed point does not satisfy the invariance properties \eqref{eq:KPZinvariance}.  
Instead, it was conjectured in \cite{Baik-Liu19} that 
\beq \label{eq:pkpzsmallt}
	\epsilon^{-1/3} \pkpz(2\epsilon^{2/3} \gamma, \epsilon \tau) \to \kpz(\gamma, \tau) \qquad \text{as $\epsilon\to 0$}
\eeq
and 
\beq \label{eq:pkpzlarget}
	\frac{\sqrt{2}}{\pi^{1/4}T^{1/2}}  \left( \pkpz(\gamma,  \tau T) + \tau T \right)  \to \bm(\tau) \qquad \text{as $T\to\infty$}
\eeq 
where $\bm$ is a Brownian motion. 
The limit in \eqref{eq:pkpzlarget} does not depend on $\gamma$. 
The one-point distribution case of \eqref{eq:pkpzsmallt} with $\gamma=0$ was verified in \cite[Theorem 1.6]{Baik-Liu-Silva22} and the one-point distribution case of \eqref{eq:pkpzlarget} was proved in \cite[Theorem 1.5]{Baik-Liu-Silva22}. 

\bigskip

The process $\pkpz$ has period $1$. It is illuminating to consider the general periods. 
For $p>0$, let
\beq \label{eq:pkpwithpdef}
 \ppkpz_p(\gamma, \tau) := p^{1/2} \pkpz(p^{-1} \gamma, p^{-3/2} \tau).
\eeq
Then, it satisfies 
\beq
	\ppkpz_p(\gamma+p, \tau)= \ppkpz_p(\gamma, \tau)
\eeq
We call it the 
$p$-periodic KPZ fixed point (with the periodic narrow wedge initial condition).
The processes with different parameters are related by the formula
\beq \label{eq:scalingprop}
	p^{-1/2} \ppkpz_p( p \gamma, p^{3/2} \tau) \stackrel{d}{=}   \ppkpz_1(  \gamma, \tau) 
\eeq
for all $p>0$. 
The conjectures \eqref{eq:pkpzsmallt} and \eqref{eq:pkpzlarget} are translated to the conjecture on the large period limit 
\beq \label{eq:pargeperiod}
  \ppkpz_p(2\gamma, \tau)  \to \kpz(\gamma, \tau)\qquad \text{as $p \to \infty$}
\eeq
and the conjecture on the small period limit
\beq \label{eq:smallperiod}
 \frac{\sqrt{2} p^{1/4} }{\pi^{1/4}} \left(\ppkpz_p(\gamma, \tau) + p^{-1} \tau \right)  \to \bm(\tau)\qquad \text{as $p \to 0$,}
\eeq 
both in the sense of convergence in finite-dimensional distributions. 

\subsection{Results}

The goal of this paper is to study the periodic KPZ fixed point when $\ppkpz_p(0,1)=\ell$ is unusually large.
We obtain the following results up to the ``pinch-up" time. We allow the period $p$ vary while $\ell\to\infty$. 
There are three theorems depending on the value $p\ell^{1/4}$. The critical case is when $p\ell^{1/4}=O(1)$. We call the other two cases $p\gg O(\ell^{-1/4})$ and $p\ll O(\ell^{-1/4})$ the case of large period and the case of small period, respectively.

\begin{thm}[Large period case]
	\label{thm1}
	We have
\beqq\begin{split}
	\law \left(\left\{\frac{\ppkpz_p(\frac{ x}{ \ell^{1/4}}, t)- t\ell}{ \ell^{1/4}}\right\}_{(x, t)\in \realR\times (0,1)}\,\Big|\, \ppkpz_p(0,1)=\ell \right)\fddto  \law\left( \left\{ \bb_2(t)  - \left|  \bb_1(t)  -x\right| \right\}_{(x, t)\in \realR\times (0,1)} \right)
\end{split}\eeqq
as $\ell \to \infty$ if \footnote{The notations mean that $p\ell^{1/4} \to \infty$ and $\log p/\ell^{3/2}\to 0$
as $\ell\to \infty$.} 
\beqq
	\ell^{-1/4} \ll p \qquad \text{and} \qquad \log p\ll \ell^{3/2},
\eeqq
where $\bb_1$ and $\bb_2$ are independent Brownian bridges. 
\end{thm}

In order to introduce the limit theorem in the critical period case, we need to define a Brownian bridge on a periodic domain. 
For $\ratio>0$, let $\Sr=\realR/ \ratio \intZ$ be the quotient space of $\realR$ by $\ratio \intZ$. 
We denote by $\{x\}_\ratio$ the equivalence class of $x\in\realR$. Naturally, it satisfies the periodicity $\{x\}_\ratio =\{x+\ratio\}_\ratio$. The distance of two points $\{x\}_\ratio$ and $\{y\}_\ratio$ of $\Sr$ is defined as
\begin{equation} \label{eq:distratiodef}
  \dist_\ratio  (\{x\}_\ratio, \{y\}_\ratio) = \min_{k\in\intZ} |x-y+k\ratio|.
\end{equation}
A Brownian motion $\bmcir(t)$ on $\Sr$ can be defined from a Brownian motion $\bm(t)$  on $\realR$ by the formula 
\begin{equation}
	\bmcir(t) = \{\bm(t)\}_\ratio. 
\end{equation}
It is straightforward to see that $\bmcir(t)$ is a Markov process with the transition density 
\begin{equation*}
\lim_{\epsilon\to 0}\epsilon^{-1}\prob\left(\dist_\ratio \left(\bmcir(t) - \{x\}_\ratio\right) \le \epsilon  \mid \bmcir(s)=\{y\}_\ratio\right) = \phi^{(\ratio)}_{t-s}(\{x-y\}_\ratio)
\end{equation*}
for all $\{x\}_\ratio,\{y\}_\ratio \in \Sr$ and ordered times $s<t$, where the function $\phi^{(\ratio)}$ is defined by
\beq 
\label{eq:def_phi_ratio}
    \phi^{(\ratio)}_t(\{x\}_\ratio) = \sum_{k\in\intZ} \phi_t(x+k\ratio), \quad \text{for } \{x\}_\ratio\in\Sr,
\eeq
and $\phi_t(x):= \frac{1}{\sqrt{2\pi t}}e^{-\frac{x^2}{2t}}$ is the density function of the centered Gaussian distribution with variance $t>0$.

A Brownian bridge $\bbcir(t)$ on $\Sr$ is a Brownian motion on $\Sr$ conditional on $\bbcir(1)=\{0\}_\ratio$. Its finite-dimensional distributions are given by
\begin{equation*}
   \prob\left( \bigcap_{i=1}^m \left\{\bbcir(t_i)\in A_i \right\} \right) = \prob\left( \bigcap_{i=1}^m \left\{\bmcir(t_i)\in A_i\right\} \Big|\, \bmcir(1) = \{0\}_\ratio\right)
\end{equation*}
for any $m\ge 1$, where $t_1,\cdots, t_m$ are $m$ distinct times on $(0,1)$, and $A_1,\cdots,A_m$ are $m$ open sets in $\Sr$.

\begin{thm}[Critical period case]
	\label{thm2}
		For $\ratio>0$, we have
	\beqq\begin{split}
		\law \left(\left\{\frac{\ppkpz_p(\frac{ x}{ \ell^{1/4}}, t)- t\ell}{ \ell^{1/4}}\right\}_{(x, t)\in \realR\times (0,1)}\,\Big|\, \ppkpz_p(0,1)=\ell \right)\fddto  \law\left( \left\{ \bb_2(t) - \dist_\ratio \left(\bbcir_1(t),\{x\}_\ratio\right)\right\}_{(x, t)\in \realR\times (0,1)} \right)
	\end{split}\eeqq
	as $\ell \to \infty$ if
\beqq
        p = \ratio   \ell^{-1/4}, 
\eeqq
where $\bbcir_1(t)$ is a Brownian bridge on $\Sr$ and $\bb_2(t)$ is a Brownian bridge on $\realR$ which is independent of $\bbcir_1(t)$.
\end{thm}

\begin{thm}[Small period case]
	\label{thm3}
	We have
\beqq\begin{split}
	\law \left(\left\{\frac{\ppkpz_p(\frac{  x}{ \ell^{1/4}}, t)- t\ell}{  \ell^{1/4}}\right\}_{(x, t)\in \realR\times (0,1)}\,\Big|\, \ppkpz_p(0,1)=\ell \right)\fddto  \law\left( \left\{  \bb(t)\right\}_{(x, t)\in \realR\times (0,1)} \right)
\end{split}\eeqq
as $\ell \to \infty$ if 
\beqq
	\ell^{-1} \log \ell\ll p\ll \ell^{-1/4}, 
\eeqq
where $\mathbb{B}$ is a Brownian bridge. 
\end{thm}

\bigskip

We have several remarks. 
 \begin{itemize}

\item 
Let $\widetilde \ppkpz_p(x, t):= \frac{\ppkpz_p(\frac{ x}{\ell^{1/4}}, t)- t\ell}{ \ell^{1/4}} $. 
The above results should be understood as, for every positive integer $m$ and real numbers $h_1, \cdots, h_{m-1}$, the limit 
\beq \label{eq:limitinell} \begin{split}
	\lim_{\ell \to \infty} \prob\left( \widetilde \ppkpz_p(x_1, t_1) \ge h_1  ,\cdots,  \widetilde \ppkpz_p(x_{m-1}, t_{m-1}) \ge h_{m-1} \,\bigg|\, \ppkpz_p(0,1)= \ell \right)
\end{split}\eeq
exists in each case and is given by the corresponding joint probabilities of the limiting fields  in the theorems. 
The conditional probability in \eqref{eq:limitinell}  should be understood as 
\beqq\begin{split}
	\lim_{\epsilon\to 0} 
	\frac{\prob\left(  \widetilde \ppkpz_p(x_1, t_1) \ge h_1  ,\cdots,  \widetilde \ppkpz_p(x_{m-1}, t_{m-1}) \ge h_{m-1}  ,  \ppkpz_p(0,1)\in (\ell-\epsilon, \ell+ \epsilon) \right)}{\prob (\ppkpz_p(0,1)\in (\ell-\epsilon,\ell+ \epsilon) ) }.
\end{split}\eeqq

\item In all three cases, the position and the height are scaled the same way as in the KPZ fixed point case \eqref{eq:kpzLiuWang}, except for multiplicative $2$ in the spatial variable (see \eqref{eq:pargeperiod}). The factor $2$ is due to the fact that the periodic KPZ fixed point uses a different convention for the spatial variable compared to the KPZ fixed point. Similar discrepancy also appears in the equation \eqref{eq:pargeperiod}.

\item
   Consider the limiting field of Theorem \ref{thm2}. Note that 
	\begin{equation}
		\lim_{\ratio\to\infty} \dist_\ratio \left(\{x\}_\ratio, \{y\}_\ratio\right) =|x-y|,\qquad  \lim_{\ratio\to 0} \dist_\ratio \left(\{x\}_\ratio, \{y\}_\ratio\right)=0
        \qquad \text{for $x,y\in \realR$. }
	\end{equation}
    Also note that a Brownian bridge on $\Sr$ becomes a Brownian bridge on $\realR$ as $\ratio\to\infty$ and tends to $0$ as $\ratio\to0$. Thus, the limiting field in the critical period case interpolates the limiting fields in the other two cases. 
    
\item 
For fixed $t$, the limiting field of Theorems \ref{thm1} has maximum value $\bb_2(t)$ obtained at $x=\bb_1(t)$. 
Similarly, for fixed $t$, the limiting field of Theorems \ref{thm2} has maximum value $\bb_2(t)$ obtained at $x=\bbcir_1(t)$. 
On the other hand, the limiting field of Theorem \ref{thm3} does not depend on $x$.

\item  
As mentioned at the end of Section \ref{sec:KPZresult}, conditional on $\kpz(0,1)=\ell$, the geodesic from $(0,0)$ to $(0,1)$ in the directed landscape stays within a distance of order $\ell^{-1/4}$  from a straight line as $\ell\to\infty$. 
Since $\ppkpz_p$ has the period $p$, it is natural to conjecture that limit theorems for the periodic KPZ fixed point take  different forms depending on $p\gg \ell^{-1/4}$ or $p\ll \ell^{-1/4}$. 
The results above show that the critical regime is indeed when $p$ is same order as $\ell^{-1/4}$. 

\item 
In Theorem \ref{thm1}, $p$ is allowed to tend to zero, stay $O(1)$, or tend to infinity as long as it satisfies 
$p\gg \ell^{-1/4}$ and $\log p\ll \ell^{3/2}$. 
In this case, 
the limit is exactly same that of the KPZ fixed point \eqref{eq:kpzLiuWang} as the KPZ fixed point (except for the factor $2$ in the spatial scale). 
The condition $\log p\ll \ell^{3/2}$ is a technical one. 
We expect that Theorem \ref{thm1} holds true as long as $p\gg \ell^{-1/4}$, but it is not clear how to remove this condition from our proof.

\item In Theorems \ref{thm2} and \ref{thm3}, the period $p$ necessarily tends to zero. 
Theorem \ref{thm2} corresponds to the case when the period and the geodesic interact non-trivially. 
The condition $p\gg \ell^{-1} \log\ell$ in theorem \ref{thm3} is also a technical one, which may be weakened, but we do not expect that it can be completely removed.

\item From the scaling property \eqref{eq:scalingprop}, the theorems imply similar results when we condition at time $\tau$ instead of time $1$: conditional on that $\ppkpz_p(0,\tau)=\ell$, 
\beqq
     \frac{\ppkpz_p( \frac{  x \tau^{3/4}}{ \ell^{1/4}}, t \tau )- t \ell}{ \tau^{1/4} \ell^{1/4}}  
\eeqq
converges to the limit from one of the above three theorems as $\ell\to \infty$.

\item 
For the case when $p=O(1)$ or $p\to \infty$ in Theorem \ref{thm1}, it is also interesting to consider the analogous result conditional on $\ppkpz_p(x,1)=\ell$. We expect a result similar to \eqref{eq:kpzXaway}. However, unlike the KPZ fixed point situation, 
the periodic KPZ fixed point does not satisfy the invariance properties \eqref{eq:KPZinvariance}, and the result does not directly follow from Theorem \ref{thm1}. This situation will be studied in a separate paper. 

\end{itemize}


\subsection{Asymptotics of the one-point density in the right tail regime}
The analysis of the paper also yields the following  asymptotics. 
\begin{thm} \label{res:onepointtail}
Let
\begin{equation}
	f_p(\hite;\gamma,\tau)= \frac{\dd}{\dd\hite} \prob\left( \ppkpz_p (\gamma, \tau)\le \hite \right) 
\end{equation}
be the one-point density function of the periodic KPZ fixed point. 
Then, as $\ell\to\infty$, 
\begin{equation}
	f_p(\ell;0,1)
	=\begin{dcases}
	\frac{1}{8\pi \ell}e^{-\frac{4}{3}\ell^{3/2}} (1+o(1)) & \text{if } \ell^{-1/4}\ll p \text{ and } \log p\ll \ell^{3/2},\\
	\frac{\cnst(\ratio)}{8\pi \ell}e^{-\frac{4}{3}\ell^{3/2}} (1+o(1)) &\text{if } 
        p=\ratio  \ell^{-1/4},\\
	\frac{1}{4\sqrt{2\pi} \ell^{5/4}p}e^{-\frac{4}{3}\ell^{3/2}} (1+o(1)) \quad & \text{if } \ell^{-1}\log\ell \ll p\ll \ell^{-1/4},
\end{dcases}
\end{equation}
where
\begin{equation}
    \cnst(\ratio)= \sum_{k\in\intZ}e^{-\frac12\ratio^2 k^2} = \frac{\sqrt{2\pi}}{\ratio}  \sum_{k\in\intZ} e^{-\frac{2\pi^2}{\ratio^2} k^2} .
\end{equation}
\end{thm}

Thus, if $\ell^{-1/4}\ll p \text{ and } \log p\ll \ell^{3/2}$, the asymptotic matches the right tail behavior of the density function of the GUE Tracy-Widom distribution. 
For the KPZ fixed point on the line, which is the $p=\infty$ case of the periodic KPZ fixed point, the one-point distribution is given by the GUE Tracy-Widom distribution. 
Hence, we expect that the asympotics does not depend on $p$ as long as $p\gg \ell^{-1/4}$.

The right tail when $p=1$ was previously obtained for the  one-point distribution function. The result \cite[Theorem 1.7]{Baik-Liu-Silva22} shows that\footnote{Asymptotic result was also obtained for $\prob\left( \ppkpz_1 (\gamma, \tau)> \ell \right)$ for all $\gamma, \tau$.} 
\beqq
	\prob\left( \ppkpz_1 (0, 1)> \ell \right) 
	= \frac{1}{16\pi \ell^{3/2}} e^{-\frac43 \ell^{3/2}} (1+ O(\ell^{-3/2})) \qquad \text{as $\ell\to \infty$.}
\eeqq 
The above theorem when $p=1$ is consistent with the formal derivative of this result.

\subsection{Structure of the paper}

The proofs of Theorems \ref{thm1}--\ref{thm3} are based on an analysis of an explicit formula of the multi-time, multi-position distributions of the periodic KPZ fixed point obtained in \cite{Baik-Liu19}. 
The method is similar to that of \cite{Liu-Wang22} for the KPZ fixed point, but since the formulas for the periodic KPZ fixed point are more complicated, the analysis is more involved. 
The other difficulty is to find probabilistic descriptions of the limits of the formulas, especially for Theorem \ref{thm2}, which we first obtain in terms of complicated contour integrals. We guess the probabilistic interpretations of the formulas and check that they are correct by direct computations.

The explicit formula of the multi-time, multi-position distributions of the periodic KPZ fixed point involves an integral of a Fredholm determinant. 
In Section \ref{sec:proof}, we introduce this formula and show that upon the conditioning, the integral of some terms of the series expansion of the Fredholm determinant vanishes. 
We then state four propositions, Propositions \ref{lm:main_contribution_largeL}--\ref{propSrpr}, and prove Theorems \ref{thm1}--\ref{thm3} assuming these propositions. 
We also prove Theorem \ref{res:onepointtail} in this section. 
Section \ref{sec:asapre} is a preparatory section where we consider a function appearing in the distribution formula to compute its limit and obtain several bounds. 
Section \ref{sec:asymp} is the main analytic part of the paper. We perform asymptotic analysis and prove 
Proposition \ref{lm:main_contribution_largeL}--\ref{lm:error_estimate_largeL}. 
Proposition \ref{propSrpr}, is proved in Section \ref{sec:pflastprop}. 
There are two sections in the Appendix. 
In Appendix \ref{sec:appendix}, we prove \eqref{eq:taseplimit} for the exceptional values of parameters that were not treated in \cite{Baik-Liu19} and also prove the continuity and consistency of the distribution functions $\tasepF_m$. 
Finally, we show in Section \ref{sec:Dnformula} that the series formula of the Fredholm determinant in Section \ref{sec:proof} is the same as that of \cite{Baik-Liu19}.

\subsection*{Acknowledgments}

The authors would like to thank B\'alint Vir\'ag and an anonymous referee for their comments, which helped improve the presentation of the paper. 

The work of Baik was supported in part by NSF grants DMS-1954790, DMS-2246790, and by the Simons Fellows program. 
The work of Liu was supported by  NSF grants DMS-1953687  and DMS-2246683. 


\section{Proof of theorems}\label{sec:proof}

\subsection{Set-up}
\label{sec:set-up}

The conditions on $p$ and $\ell$ in Theorem \ref{thm1}--\ref{thm3} are 
\begin{itemize}
\item (Case 1) 
$p\gg \ell^{-1/4}$, $\log p\ll \ell^{3/2}$, and $\ell\to\infty$,
\item (Case 2) $p=\ratio  \ell^{-1/4}$  and $\ell\to \infty$,
\item (Case 3) $\ell^{-1} \log\ell\ll p\ll \ell^{-1/4}$ and $\ell\to \infty$,
\end{itemize}
respectively. 
In the rest of the paper, we will refer these limits as ``for Case 1", and so on. 
In each case, we evaluate the limit of 
\beq \label{eq:ptoper} \begin{split}
	 \prob\left(  \bigcap_{i=1}^{m-1} \left\{  \widetilde \ppkpz_p(x_i, t_i) \ge h_i  \right\} 
	 \,\bigg|\, \ppkpz_p(0,1)= \ell \right). 
\end{split} \eeq
in \eqref{eq:limitinell}. 
We will also often state that a result holds ``eventually" to mean that it holds when the appropriate parameters are large enough. 
For example, for Case 3, it means that there are positive constants $c_1, c_2, c_3>0$ such that the result holds for all $\ell$ and $p$ satisfying 
$\ell\ge c_1$, $p^{-1}\ell^{-1/4}\ge c_2$, and $\frac{p\ell}{\log \ell}\ge c_3$. 

\bigskip

The following result from \cite{Liu-Wang22} shows that when we consider the limit of \eqref{eq:ptoper} it is enough to consider the case when $t_1, \cdots, t_{m-1}$ are all distinct.  

\begin{lemma}\cite[Lemma 3.6]{Liu-Wang22}
\label{lem:equal_time} 
Let $Y$ be a random field on $\realR\times(0,T)$ with the property that 
for every positive integer $d$ and $x_1,\cdots,x_d\in \realR$, the cumulative distribution function 
$\prob\left(\cap_{i=1}^d \{ Y(x_i,t_i)\le \hite_i \} \right)$ is continuous in the variables $\hite_i$ and $t_i$ for every $1\le i\le d$. 
If a sequence of random fields $Y_n$ on $\realR\times(0,T)$ satisfies 
$(Y_n(x_i,t_i))_{i=1,\cdots,d}\to (Y(x_i,t_i))_{i=1,\cdots, d}$ 
in distribution as $n\to \infty$ for every $d$ and for every $(x_i, t_i)\in \realR \times (0,T)$, $i=1,\cdots, d$, where $t_1,\cdots,t_d$ are distinct numbers, 
then 
$Y_n(x,t)\to Y(x,t)$ in the sense of convergence of finite-dimensional distributions as $n\to \infty$.  
\end{lemma}

Thus, the convergence for distinct times imply the convergence for arbitrary times if the limit distributions satisfy a continuity property. 
The limit fields in Theorems \ref{thm1}-\ref{thm2} clearly satisfy the continuity properties, and thus, it is enough to prove the convergence in distribution for distinct times only.

\subsection{Formula of the distribution functions} \label{sec:formdf}

Recall the relation \eqref{eq:pkpwithpdef} between $\ppkpz_p$ and $\pkpz$. 
When the times are distinct,\footnote{The result is also obtained for equal-time case when $\tau_i=\tau_{i+1}$ for some $i$ as long as $\beta_i<\beta_{i+1}$.
}  there is an exact formula for the multi-point distributions of $\pkpz$. We state the formula here. We analyze this formula to prove the theorems.  

Let $\hite_1, \cdots, \hite_m$ be real numbers. 
Set $I_i^+=[\hite_i, \infty)$ and $I_i^-=(-\infty, \hite_i]$. 
Consider $m$ points 
$(\gamma_1, \tau_1)$, $\cdots$, $(\gamma_m, \tau_m)\in \realR\times (0, \infty)$ satisfying $0<\tau_1<\cdots<\tau_m$. 
The paper \cite{Baik-Liu19} obtained formulas for the joint probabilities
\beq \label{eq:proplusminus}
	\prob\left( \pkpz(\gamma_1, \tau_1)\in I_1^{\pm}, \cdots, \pkpz(\gamma_{m-1}, \tau_{m-1})\in I_{m-1}^{\pm}, \pkpz(\gamma_m, \tau_m)\in I_m^{-}  \right)
\eeq
for arbitrary choices of $+$ and $-$ in each place.
When all but the last signs are positive, we have (see \cite[eq. (7.17)]{Baik-Liu19} and note the relation \eqref{eq:pkpwithpdef} between $\pkpz$ and $\ppkpz_p$) 
\beq \label{eq:jointprointe}
	\prob\left(\bigcap_{i=1}^{m-1} \left\{ \ppkpz_p (\gamma_i,\tau_i) \ge \hite_i \right\} \cap \left\{ \ppkpz_p(\gamma_m, \tau_m)\le \hite_m \right\}\right)
	= \frac{(-1)^{m-1}}{(2\pi \ii)^m } \oint \cdots \oint C(\bfz)D(\bfz) \prod_{i=1}^m \frac{\dd z_i}{z_i} 
\eeq
where the contours are circles centered at the origin with the radii satisfying $0<|z_1|<\cdots<|z_m|<1$. 
With $\bfz=(z_1, \cdots, z_m)$, the functions $C(\bfz)$ and $D(\bfz)$ are defined in \eqref{eq:Cdeff} and \eqref{eq:def_Dbfz} below. 
In Appendix \ref{sec:appendix}, we will use the case when all signs are negative.

To introduce the function $C(\bfz)$, let $\Li_s(z)$ denote the polylogarithm function of order $s$.  
Then,  
\beq \label{eq:Cdeff}
	C(\bfz)=\prod_{i=1}^{m-1} \frac{z_i}{z_{i}-z_{i+1}} 
	\prod_{i=1}^m \frac{e^{ \frac{\hite_i}{p^{1/2}} A_1(z_i) +\frac{\tau_i}{p^{3/2}} A_2(z_i)}}{e^{ \frac{\hite_i}{p^{1/2}}  A_1(z_{i+1}) + \frac{\tau_i}{p^{3/2}} A_2(z_{i+1})}}
	e^{2B(z_i, z_i)-2B(z_{i+1},z_i)}
\eeq
where 
\beq \label{eq:A12B} \begin{split}
	&A_1(z)=-\frac{1}{\sqrt{2\pi}} \Li_{3/2}(z), \qquad A_2(z) = -\frac{1}{\sqrt{2\pi}} \Li_{5/2}(z),
	\qquad B(z,z')=\frac{1}{4\pi}\sum_{k,k'= 1}^\infty \frac{z^k(z')^{k'}}{(k+k')\sqrt{kk'}} . 
\end{split} \eeq
Here, we set $z_{m+1}=0$ in the expressions.

The function $D(\bfz)$ is a Fredholm determinant.
The series formula of it is 
\beq \label{eq:def_Dbfz}
	D(\bfz) = \sum_{\bfn \in \{0, 1, \cdots\}^m } \frac{1}{(\bfn !)^2} D_{\bfn}(\bfz)
\eeq
where $\bfn!= n_1! n_2! \cdots n_m!$ for $\bfn=(n_1, \cdots, n_m)$ and $D_{\bfn}(\bfz)$ is given below. 
The formula of $D_{\bfn}(\bfz)$ below is slightly different from that of \cite[Lemma 2.10]{Baik-Liu19}  and we explain in Appendix \ref{sec:Dnformula} how to obtain the formula.\footnote{The paper \cite{Baik-Prokhorov-Silva23} also discusses another Fredholm determinant formula.} 

For $|z|<1$, define the discrete set 
\beq
	\rmL_z=\{w: e^{-w^2/2}=z, \, \Re(w)<0\}. 
\eeq
For $\bfn=(n_1, \cdots, n_m)$ and distinct complex numbers $z_1, \cdots, z_m$ in the punctured unit disk, let 
\beq \label{eq:D_bfz} \begin{split}
	D_{\bfn}(\bfz) 
	&=
	\prod_{i=2}^m \left(1-\frac{z_{i-1}}{z_i} \right)^{n_i}
	\left(1-\frac{z_{i}}{z_{i-1}} \right)^{n_{i-1}}
	\sum_{U, \, \hat U \in \rmL_{z_1}^{n_1}\times \cdots \times \rmL_{z_m}^{n_m}} 
	H_{\bfn}(U, \hat U) R_{\bfn}(U, \hat U) E_{\bfn} (U, \hat U) 
\end{split} \eeq
with the functions defined as follows. 
Let 
\beq \label{eq:hfuntion}
	h(w,z) = -\frac{1}{\sqrt{2\pi}} \int_{-\infty}^w \Li_{1/2}(ze^{(w^2-y^2)/2})\rmd y \qquad \text{for $\Re(w)<0$.} 
\eeq 
For $U= (U^{(1)}, \cdots, U^{(m)})$ and $\hat U= ( \hat U^{(1)}, \cdots, \hat U^{(m)})$ with $U^{(i)}, \hat U^{(i)}\in \rmL_{z_i}^{n_i}$, 
we write the components $U^{(i)}=(u_i^{(i)}, \cdots, u^{(i)}_{n_i})$ and $\hat U^{(i)}=(\hat u_i^{(i)}, \cdots, \hat u^{(i)}_{n_i})$. 
Then,  
\beq \label{eq:Hdefff}
	H_{\bfn}(U, \hat U)
	= \prod_{i=1}^{m} 
	\prod_{j=1}^{n_{i}} e^{2h_i(u_{j}^{(i)}) - h_{i+1}(u_{j}^{(i)}) - h_{i-1}(u_{j}^{(i)})  
	+ 2h_i(\hat u_{j}^{(i)}) - h_{i+1}( \hat u_{j}^{(i)}) - h_{i-1}( \hat u_{j}^{(i)})} 
\eeq
where $h_i(w):= h(w; z_i)$ and $h_0(w)=h_{m+1}(w)=0$.  
Next, for $X=(x_1, \cdots, x_a)$ and $Y=(y_1, \cdots, y_a)$, let
\beq \label{eq:Caudedef}
	\Cd(X; Y)= \det \left( \frac{1}{x_i + y_j} \right)_{i, j=1}^{a} 
	= \frac{\prod_{1\le i<j\le a} (x_j-x_i) (y_j-y_i)}{\prod_{i, j=1}^a (x_i + y_j)}
\eeq 
denote the Cauchy determinant. 
We have 
\beq \label{eq:Rdetfom}
	R_{\bfn}(U, \hat U)
	=  \left[ \prod_{i=1}^m \prod_{j_i=1}^{n_i} \frac{1}{u_{j_i}^{(i)} \hat u_{j_i}^{(i)} } \right]
	\prod_{i=0}^{m} \Cd( U^{(i)}, -\hat U^{(i+1)} ; \hat U^{(i)}, - U^{(i+1)}) 
\eeq
with the convention that $U^{(0)}=\hat U^{(0)}=U^{(m+1)}=\hat U^{(m+1)}=\emptyset$. 
Finally, 
\beq \label{eq:Edeffff}
	E_{\bfn}(U, \hat U)
	=  \prod_{i=1}^m\prod_{j_i=1}^{n_i} E^{i,+}(u_{j_i}^{(i)})E^{i,-}(\hat u_{j_i}^{(i)})  , 
	\qquad
	E^{i, \pm} (s):= e^{-\frac{\tau_i-\tau_{i-1}}{3p^{3/2}} s^3  \pm \frac{\gamma_i-\gamma_{i-1}}{2p}  s^2 + \frac{\hite_i-\hite_{i-1}}{p^{1/2}} s}.
\eeq

\subsection{Derivative of the distribution function}

The conditional probability is interpreted as (see \eqref{eq:ptoper}) 
\beq \label{eq:conditional_prob_origin} \begin{split}
        &\prob \bigg(\bigcap_{k=1}^{m-1} \big\{ \ppkpz_p (\gamma_k,\tau_k) \ge \hite_k \big\} \, \bigg|\,    \ppkpz_p(\gamma_m, \tau_m) = \hite_m \bigg) \\
      	&\quad =
       \lim_{\epsilon\to 0} 
       \frac{ \prob\left(\cap_{k=1}^{m-1} \left\{ \ppkpz_p (\gamma_k,\tau_k) \ge \hite_k\right\} 
       \cap \left\{ \ppkpz_p (\gamma_m, \tau_m)\in (\hite_m -\epsilon, \hite_m +\epsilon) \right\}\right)}
             {\prob\left( \ppkpz_p (\gamma_m, \tau_m)\in  (\hite_m-\epsilon, \hite_m+\epsilon)  \right)} \\   
       &\quad = \frac{\frac{\partial}{\partial \hite_m} \prob\left(\cap_{k=1}^{m-1} \left\{ \ppkpz_p (\gamma_k,\tau_k) \ge \hite_k\right\} \cap \left\{ \ppkpz_p (\gamma_m, \tau_m)\le \hite_m \right\}\right)}
             {\frac{\partial}{\partial \hite_m} \prob\left( \ppkpz_p (\gamma_m, \tau_m)\le \hite_m \right)}.
\end{split} \eeq
We now take a derivative of \eqref{eq:jointprointe} to find a formula for \eqref{eq:conditional_prob_origin}. 
We have the following result for the numerator. The denominator is given by the same formula with $m=1$.
In the result below, compared with \eqref{eq:def_Dbfz},  the sums are only over $\bfn\in \{1, 2, \cdots\}^m$, instead of being over $\bfn\in \{0, 1, 2, \cdots\}^m$. 
Also, $\hat D_{\bfn}(\bfz)$ is the same as $D_{\bfn}(\bfz)$, except for the extra factor $\sum_{j=1}^{n_m} (u_j^{(m)} +\hat u_j^{(m)})$ in the summand.
This proof is modeled on a computation from \cite{Liu-Wang22}.

\begin{prop} \label{lm:change_low_bound_n_k}
Let $\N=\{1, 2, \cdots\}$, the set of positive integers. Then,  
\beq \label{eq:derj_joint_CDF} \begin{split}
	&\frac{\partial}{\partial \hite_m} \prob\left(\cap_{k=1}^{m-1} \big\{ \ppkpz_p (\gamma_k,\tau_k) \ge \hite_k\big\} \cap \big\{ \ppkpz_p (\gamma_m, \tau_m)\le \hite_m \big\}\right) \\
	&= \frac{(-1)^{m-1} }{(2\pi \ii)^m p^{1/2}} \oint \cdots \oint \left( A_1(z_m)C(\bfz) \sum_{\bfn \in \N^m } \frac{D_{\bfn}(\bfz)}{(\bfn !)^2}  
	+ C(\bfz) \sum_{\bfn \in \N^m} \frac{\hat D_{\bfn}(\bfz) }{(\bfn !)^2} \right)  \prod_{i=1}^m \frac{\rmd z_i}{z_i} 
\end{split} \eeq
where the contours are the circles centered at the origin with radii satisfying $0<|z_1|<\cdots < |z_m|<1$. The terms $A_1(z)$, $C(\bfz)$, and $D_{\bfn}(\bfz)$ are defined in \eqref{eq:A12B}, \eqref{eq:Cdeff}, and \eqref{eq:def_Dbfz}, and 
for $\bfn=(n_1, \cdots, n_m)$, 
\beq\label{eq:hat_D_bfz}
	\hat D_{\bfn}(\bfz) 
	=
	\prod_{i=2}^m \left(1-\frac{z_{i-1}}{z_i} \right)^{n_i}
	\left(1-\frac{z_{i}}{z_{i-1}} \right)^{n_{i-1}}
	\sum_{U, \, \hat U \in \rmL_{z_1}^{n_1}\times \cdots \times \rmL_{z_m}^{n_m}} 
	H_{\bfn}(U, \hat U) \hat R_{\bfn}(U, \hat U) E_{\bfn} (U, \hat U) 
\eeq
where
\beq
	\hat R_{\bfn} (U, \hat U) = R_{\bfn}(U, \hat U)  \sum_{j=1}^{n_m} (u_j^{(m)} + \hat u_j^{(m)}). 
\eeq
\end{prop}

\begin{proof}
In the formula \eqref{eq:jointprointe}, $\beta_m$ appears in two places. 
Since 
\beqq
	\frac{\dd C(\bfz)}{\dd\beta_m} = \frac1{p^{1/2}} A_1(z_m)C(\bfz)
	\quad \text{and} \quad
	\frac{\dd E_{\bfn}(U, \hat U)}{\dd\beta_m} 
	=  \frac1{p^{1/2}}  E_{\bfn}(U, \hat U)   \sum_{j=1}^{n_m} (u_j^{(m)} + \hat u_j^{(m)})  , 
\eeqq
we find that
\beq \label{eq:derj_joint_CDF0} \begin{split}
	& \frac{\partial}{\partial \hite_m} \prob\left(\cap_{k=1}^{m-1} \big\{ \ppkpz_p (\gamma_k,\tau_k) \ge \hite_k\big\} \cap \big\{ \ppkpz_p (\gamma_m, \tau_m)\le \hite_m \big\}\right)  \\
	&=  \frac{(-1)^{m-1}}{(2\pi \ii)^m p^{1/2}} \oint \cdots \oint \bigg( A_1(z_m)C(\bfz) \sum_{\bfn \in \{0, 1, \cdots\} ^m } \frac{D_{\bfn}(\bfz)}{(\bfn !)^2}  
	+ C(\bfz) \sum_{\bfn \in \{0, 1, \cdots\}^m } \frac{\hat D_{\bfn}(\bfz) }{(\bfn !)^2}  \bigg) \prod_{i=1}^m \frac{\rmd z_i}{z_i}  
\end{split} \eeq
where the sums are over $\bfn\in \{0, 1, \cdots\}^m$. 
Note the fact that $E(U, \hat U)$ decays super-exponentially fast as a variable tends to $\infty$ in the sets $\rmL_z$  where the rate of decay depends only on $|z|\in(0,1)$. Hence the summation of $D_{\bfn}$ and $\hat D_{\bfn}$ are uniformly convergent. Thus, we can change the order of the integral with respect to $\bfz$ and summation over $\bfn$. 
Now Lemma \ref{nzeroiszeo} below shows that the integral is zero if one of the components of $\bfn$ is zero. 
Thus, we obtain the result.  
\end{proof}

Recall that the contours for the integral are circles satisfying $0<|z_1|<\cdots < |z_m|<1$. 

\begin{lemma}\label{nzeroiszeo}
If one of the components of $\bfn=(n_1, \cdots, n_m)$ is zero, then  
\beq \label{eq:vanished_integral_n_k=0}
	\oint \cdots \oint A_1(z_m)C(\bfz) D_{\bfn}(\bfz) \prod_{i=1}^m \frac{\rmd z_i}{z_i} 
	=0
\eeq
and
\begin{equation}\label{eq:vanished_integral2_n_k=0}
\oint \cdots \oint C(\bfz) \hat D_{\bfn}(\bfz) \prod_{i=1}^m \frac{\rmd z_i}{z_i}  =0.
\end{equation}
\end{lemma}

\begin{proof}
The case when $m=1$ can be checked directly. Note that in this case $D_\bfn(\bfz)=1$, $\hat D_{\bfn}(\bfz)=0$. And the functions $C(\bfz)$, $A_1(z_1)/z_1$ are both analytic at $z_1=0$. These imply the two identities \eqref{eq:vanished_integral_n_k=0} and \eqref{eq:vanished_integral2_n_k=0}. Below we assume that $m\ge 2$.

Let $\bfn=(n_1, \cdots, n_m)$ be given and one of the components is zero. 
Let $k$ be the smallest integer such that $n_k=0$. 
We will show the integrands of both integrals are analytic as a function of $z_k$ in the integration domain, and thus the integrals are zero.

We first focus on the integral in~\eqref{eq:vanished_integral_n_k=0}.
Since $n_k=0$, the set $\rmL_{z_k}^{n_k}$ is empty. Thus, the integrand does not contain any factor involving $U^{(k)}$ and $\hat U^{(k)}$, which depend on $z_k$. 
If $k=1$, the only term that depends on $z_1$ in $A_1(z_m)C(\bfz)D_{\bfn}(\bfz) \prod_{i=1}^m \frac{1}{z_i} $ 
is the factor 
\beqq
	\frac{1}{z_1-z_2} e^{-\sum_{j=1}^{n_2}(h_1(u_{j}^{(2)})+h_1(\hat u_{j}^{(2)}))} \left(1-\frac{z_1}{z_2}\right)^{n_2}e^{\hite_1 A_1(z_1) + \tau_1 A_2(z_1) +2 B(z_1) -2B(z_2,z_1)}. 
\eeqq
Since $|z_1|<|z_2|$, this function is analytic at $z_1=0$. This implies the integrand in~\eqref{eq:vanished_integral_n_k=0} is analytic in $z_1$ around the origin. Hence \eqref{eq:vanished_integral_n_k=0} holds when $k=1$.
  
If $1<k<m$, the only term that depends on $z_k$ in $A_1(z_m)C(\bfz)D_{\bfn}(\bfz) \prod_{i=1}^m \frac{1}{z_i} $ 
is  
\beqq \begin{split}
	&\frac{1}{(z_{k-1}-z_k)(z_k-z_{k+1})} e^{-\sum_{i=1}^{n_{k-1}}(h_k(u_{i}^{(k-1)})+h_k(\hat u_i^{(k-1)}))-\sum_{j=1}^{n_{k+1}}(h_k(u_{j}^{(k+1)})+h_k(\hat u_j^{(k+1)}))} \\
	&\times \left(1-\frac{z_k}{z_{k-1}}\right)^{n_{k-1}} \left(1-\frac{z_k}{z_{k+1}}\right)^{n_{k+1}}  e^{(\hite_k-\hite_{k-1})A_1(z_k) + (\tau_k-\tau_{k-1})A_2(z_k) + 2B(z_k) -2B(z_{k},z_{k-1}) -2B(z_{k+1},z_k)}. 
\end{split} \eeqq
As a function of $z_k$, it is of the form
\beqq
	(z_k - z_{k-1} )^{n_{k-1}-1} (z_k-z_{k+1})^{n_{k+1}-1} \times \text{(a term analytic in  $|z_k|<1$)}
\eeqq
Since $n_{k-1}\ge 1$, the first factor is analytic in $z_k$. 
On the other hand, due to the contour conditions, the second factor is analytic in $|z_k|<|z_{k+1}|$. 
Thus, the whole term is analytic at $z_k=0$, and we obtain \eqref{eq:vanished_integral_n_k=0} when $1<k<m$.

If $k=m$, 
the only term that depends on $z_m$ in $A_1(z_m) C(\bfz)D_{\bfn}(\bfz) \prod_{i=1}^m \frac{1}{z_i} $ is 
\beqq \begin{split}
	&\frac{A_1(z_m)}{z_m ( z_{m-1}-z_m) } e^{-\sum_{j=1}^{n_{m-1}}(h_m(u_j^{(m-1)})+h_m(\hat u_j^{(m-1)}))} \\
 & \times \left(1-\frac{z_m}{z_{m-1}} \right)^{n_{m-1}} e^{(\hite_m-\hite_{m-1})A_1(z_m) + (\tau_m-\tau_{m-1})A_2(z_m) + 2B(z_m) -2B(z_{m},z_{m-1})}. 
\end{split} \eeqq
As a function of $z_m$, it is of the form 
\beqq
	 (z_m - z_{m-1} )^{n_{m-1}-1} \frac{A_1(z_m)}{z_m}  \times \text{(a term analytic in  $|z_m|<1$)}
\eeqq
Since $n_{m-1}\ge 1$, the above is analytic at $z_m=z_{m-1}$. On the other hand, since $A_1(0)=0$, the term $\frac{A_1(z_m)}{z_m}$ is analytic at $z_m=0$. 
Thus, the integrand in~\eqref{eq:vanished_integral_n_k=0} is analytic in $z_m$ within the integration contour. We obtain~\eqref{eq:vanished_integral_n_k=0}. 

The proof of~\eqref{eq:vanished_integral2_n_k=0} is exactly the same as that of~\eqref{eq:vanished_integral_n_k=0} when $k<m$ since $\hat D_{\bfn}(\bfz)$ is the same as $D_{\bfn}(\bfz)$ except an extra factor $\sum_{j=1}^{n_m} (u_j^{(m)}+\hat u_j^{(m)})$ which does not depend on $z_k$. When $k=m$, we have $n_m=0$. This factor $\sum_{j=1}^{n_m} (u_j^{(m)}+\hat u_j^{(m)})=0$ hence the integrand is zero. We still have~\eqref{eq:vanished_integral2_n_k=0}.
\end{proof}

\bigskip

From the above results, we can write the probability in \eqref{eq:ptoper} as \eqref{eq:PreadyPP} below. 

\begin{defn}
For $\bfz=(z_1, \cdots, z_m)$ with $0<|z_1|< \cdots<|z_m|<1$, 
define 
\beq \label{eq:bulleCdeff}
	C^{\bullet}(\bfz) =C(\bfz) \prod_{i=1}^{m-1} \frac{z_{i}-z_{i+1}} {z_i}
	= \prod_{i=1}^m \frac{e^{ \frac{\hite_i}{p^{1/2}} A_1(z_i) +\frac{\tau_i}{p^{3/2}} A_2(z_i)}}{e^{ \frac{\hite_i}{p^{1/2}}  A_1(z_{i+1}) + \frac{\tau_i}{p^{3/2}} A_2(z_{i+1})}} e^{2B(z_i, z_i)-2B(z_{i+1},z_i)}
\eeq
where $A_1, A_2, B$ are given in \eqref{eq:A12B}, and we set $z_{m+1}=0$. 
Define 
\beq \label{eq:bulleD_bfz} \begin{split}
	D^{\bullet}_{\bfn}(\bfz) 
	&=
	\sum_{U, \, \hat U \in \rmL_{z_1}^{n_1}\times \cdots \times \rmL_{z_m}^{n_m}} 
	H_{\bfn}(U, \hat U) R_{\bfn}(U, \hat U) E_{\bfn} (U, \hat U) , \\
	\hat D^{\bullet}_{\bfn}(\bfz) 
	&=
	\sum_{U, \, \hat U \in \rmL_{z_1}^{n_1}\times \cdots \times \rmL_{z_m}^{n_m}} 
	H_{\bfn}(U, \hat U) \hat R_{\bfn}(U, \hat U) E_{\bfn} (U, \hat U) 
\end{split} \eeq
where the functions $H_{\bfn}(U, \hat U)$, $R_{\bfn}(U, \hat U)$, and $E_{\bfn} (U, \hat U)$ are defined in \eqref{eq:Hdefff}, \eqref{eq:Rdetfom}, and 
\eqref{eq:Edeffff}, while the function $\hat R_{\bfn}(U, \hat U)$ is defined in \eqref{eq:hat_D_bfz}. 
Also define 
\beq \label{eq:bfTdefff}
	\bT_{\bfn}(\bfz)= 
	\prod_{i=2}^m \left(1-\frac{z_{i-1}}{z_i} \right)^{n_i}
	\left(1-\frac{z_{i}}{z_{i-1}} \right)^{n_{i-1}-1}.
\eeq
\end{defn}

\begin{cor} \label{cor:Pready}
Let $\N=\{1, 2, \cdots\}$ and $\bfone=(\underbrace{1, \cdots, 1}_{m})$. Let 
$C^{\bullet}(\bfz)$,  $D^{\bullet}_{\bfn}(\bfz)$, and  $\hat D^{\bullet}_{\bfn}(\bfz)$ be given above with the parameters
\beq \label{eq:tagahi00}
	\tau_i = t_i , \qquad 
	\gamma_i =  \frac{  x_i}{\ell^{1/4}}, 
	\qquad \hite_i = t_i \ell + h_i  \ell^{1/4}
	\qquad \text{for $i=1,\cdots,m$,}
\eeq
where 
$0<t_1<\cdots< t_{m-1}<1$, $x_1, \cdots, x_{m-1}\in \realR$,  $h_1, \cdots, h_{m-1}\in \realR$, and 
$t_m=1$, $x_m=0$, $h_m=0.$
Then, 
\beq \label{eq:PreadyPP} \begin{split}
	\prob\left( \bigcap_{i=1}^{m-1} \bigg\{  \frac{\ppkpz( \frac{  x_i}{\ell^{1/4}}, t_i)- t_i \ell}{  \ell^{1/4}} \ge h_i \bigg\}  \,\bigg|\, \ppkpz_p(0,1)= \ell \right) 
	= \frac{\PP_{m,1}+ \PP_{m,2}+ \hat \PP_{m,1}+ \hat \PP_{m,2}}{\PP_{1,1}+\PP_{1,2}+ \hat \PP_{1,1}+ \hat \PP_{1,2} }
\end{split} \eeq
where
\beq \label{eq:PPdefns} \begin{split}
	\PP_{m,1} &= \frac{(-1)^{m-1}}{(2\pi \ii)^m}  \oint \cdots \oint  A_1(z_m) \bC(\bfz) \bD_{\bfone}(\bfz) \bT_{\bfone}(\bfz) \prod_{i=1}^m \frac{\rmd z_i}{z_i}, \\
	\PP_{m,2} &= \frac{(-1)^{m-1}}{(2\pi \ii)^m} \sum_{\bfn \in \N^m\setminus\{\bfone\} }  \frac1{(\bfn !)^2} \oint \cdots \oint  A_1(z_m)\bC(\bfz) \bD_{\bfn}(\bfz)  \bT_{\bfn}(\bfz)  \prod_{i=1}^m \frac{\rmd z_i}{z_i}, \\
	\hat \PP_{m,1} &= \frac{(-1)^{m-1}}{(2\pi \ii)^m} \oint \cdots \oint  \bC(\bfz) \hat D^{\bullet}_{\bfone}(\bfz) \bT_{\bfone}(\bfz) \prod_{i=1}^m \frac{\rmd z_i}{z_i}, \\
	\hat \PP_{m,2} &= \frac{(-1)^{m-1}}{(2\pi \ii)^m} \sum_{\bfn \in \N^m\setminus\{\bfone\} }  \frac1{(\bfn !)^2} 
	\oint \cdots \oint  \bC(\bfz) \hat D^{\bullet}_{\bfn}(\bfz)  \bT_{\bfn}(\bfz)  \prod_{i=1}^m \frac{\rmd z_i}{z_i}. 
\end{split}\eeq
\end{cor}

\subsection{Four propositions}

We analyze the equation \eqref{eq:PreadyPP} to prove Theorems \ref{thm1}--\ref{thm3}. 
We will see that the main contributions to the limit comes from $\frac{\hat\PP_{m,1}}{\hat\PP_{1,1}}$ for all three Cases. 
There are four propositions in this subsection. 
Proposition \ref{lm:main_contribution_largeL} computes the limit of $\hat \PP_{m,1}$. 
Proposition \ref{resultD1error} shows that $\PP_{m,1}$ is of a smaller order. 
Similarly, Proposition \ref{lm:error_estimate_largeL} shows that $\PP_{m,2}$ and $\hat\PP_{m,2}$ are also of smaller orders. 
Probabilistic interpretations of the limits from Proposition \ref{lm:main_contribution_largeL} are obtained in Proposition \ref{propSrpr}. 
In the next subsection, we prove the main theorems assuming these four propositions. 
The proofs of these propositions are the main analysis of this paper and they are given in Section \ref{sec:asymp} and \ref{sec:pflastprop}.

All results in this subsection hold uniformly for 
the parameters in compact subsets of 
$0<t_1<\cdots< t_{m-1}<1$, $(x_1, \cdots, x_{m-1})\in \realR^{m-1} \in \realR$,  $(h_1, \cdots, h_{m-1})\in \realR^{m-1}$ although we do not state this fact explicitly.

\medskip 

We first need some definitions. 

\begin{defn}
For every vector $\bfa=(a_1,\cdots,a_m)$ of real numbers, we denote  
\beq
	\dDelta a_i = \begin{dcases}
                 a_1, \qquad & i=1,\\
                 a_i -a_{i-1},  \qquad & 2\le i \le m.
                \end{dcases}   
\eeq 
\end{defn} 

\begin{defn} \label{defnS}
For $\bfa=(a_1,\cdots,a_m)\in\realR^m$ satisfying $0<a_1<\cdots<a_m$ and $\bfb=(b_1,\cdots,b_m)\in\realR^m$, define
\beq \label{eq:Sinfde}
    S_\infty (\bfa ,\bfb) =  \frac{(-1)^{m-1}\sqrt{2}}{(2\pi \ii)^m} \int\cdots \int \prod_{i=2}^m \frac{1}{\xi_i-\xi_{i-1}} \prod_{i=1}^m e^{ \dDelta a_i \xi_i^2 -  \dDelta b_i \xi_i} \dd \xi_1 \cdots \dd \xi_m 
\eeq
where the contours are vertical lines, oriented upward, satisfying $\Re (\xi_1) >\cdots > \Re (\xi_m)$. 
For $\bfw=(w_1, \cdots,w_m)\in \C^m$ satisfying $0< |w_1|< \cdots <|w_{m}|$, define 
\beq \label{eq:Srdefnt}
	S_\ratio (\bfa,\bfb; \bfw) = \frac{(-1)^{m-1}\sqrt{2}}{\ratio^m } \sum_{\xi_1,\cdots,\xi_m} 
	\prod_{i=2}^m \frac{1}{\xi_i-\xi_{i-1}} \prod_{i=1}^m e^{ \dDelta a_i \xi_i^2 - \dDelta b_i \xi_i} \qquad \text{for $\ratio>0$}
\eeq
where the sum is over the roots $\xi_i$ of the equations 
\beq
    {e^{-\ratio \xi_i}}=w_i \qquad \text{for $i=1, \cdots, m$.} 
\eeq
\end{defn}

Let $\bft=(t_1,\cdots,t_m)=(t_1,\cdots,t_{m-1},1)$, $\bfx=(x_1,\cdots,x_{m})=(x_1,\cdots,x_{m-1},0)$, and $\bfh=(h_1,\cdots,h_m) =(h_1,\cdots,h_{m-1},0)$. 
The first proposition is about $\hat \PP_{m,1}$.

\begin{prop} \label{lm:main_contribution_largeL}
We have  
\begin{equation}
\begin{split}
	& \frac{4\ell}{p^{1/2}} e^{\frac{4}{3} \ell^{\frac32}} \hat \PP_{m,1} \to 
\begin{dcases}
	S_\infty(\bft, \bfh-\bfx ) S_\infty(\bft,\bfh+\bfx) \quad &\text{for Case 1,} \\
	\oint\cdots \oint S_{\ratio}(\bft, \bfh-\bfx; \bfw) S_{\ratio}(\bft,\bfh+\bfx;  \bfw) \prod_{i=2}^{m} (1-\frac{w_{i-1}}{w_i})
	\prod_{i=1}^m \frac{\rmd w_i}{2\pi\rmi w_i} 
	\quad & \text{for Case 2,} 
\end{dcases}
\end{split}
\end{equation}
and 
\beq
\begin{split}
	& 2^{3/2} \ell^{5/4} p^{1/2} e^{\frac{4}{3} \ell^{\frac32}}
	\hat \PP_{m,1}
	\to  S_\infty (2\bft,2\bfh)\qquad \text{for Case 3}. 
\end{split}
\eeq
The integral contours for Case 2 are counterclockwise circles satisfying $0<|w_1|<\cdots<|w_m|$. 
\end{prop}

The formula of  $\hat \PP_{m,1}$ in \eqref{eq:PPdefns} contains $\hat D_{\bfone}(\bfz)$, which, from \eqref{eq:hat_D_bfz}, is a series. 
The above result is obtained by showing that after scaling $\bfz$ appropriately, the series converges to an integral for Case 1 and to a series for Case 2. 
Note that $S_\infty$ is an integral while $S_\ratio$ is a series. 
For Case 3, only one term dominates the series $\hat D_{\bfone}(\bfz)$.

The second proposition shows that $\PP_{m,1}$ is smaller than $\hat \PP_{m,1}$. 
Note from our assumptions in section~\ref{sec:set-up}, $p\ell\to\infty$ for all three Cases.

\begin{prop} \label{resultD1error}
There is a constant $C>0$ 
such that 
\begin{equation}
\begin{split}
	& \left|  \frac{\ell}{p^{1/2}} e^{\frac{4}{3} \ell^{\frac32}} \PP_{m,1}  \right| 
	\le \frac{C}{\sqrt{p\ell}} e^{-\frac{{p\ell }}{2}} 
	\qquad \text{for Case 1 and 2}
\end{split}
\end{equation}
and
\begin{equation}
\begin{split}
	& \left| \ell^{5/4}p^{1/2} e^{\frac{4}{3} \ell^{\frac32}} \PP_{m,1} \right| 
	\le \frac{C}{\sqrt{p\ell}} e^{-\frac{{p\ell }}{2}} 
	\qquad \text{for Case 3}
\end{split}
\end{equation}
eventually. 
\end{prop}

The third proposition shows that $\PP_{m,2}$ and $\hat \PP_{m,2}$ are small. 

\begin{prop} \label{lm:error_estimate_largeL}
There are positive constants $\delta$ and  
$C$ such that 
\begin{equation}
	\left| e^{\frac{4}{3} \ell^{\frac32}} \PP_{m,2}  \right|  \le C e^{-\delta \ell^{3/2}} 
	\quad \text{and} \quad 
	\left| e^{\frac{4}{3} \ell^{\frac32}}   \hat \PP_{m,2}   \right| \le C e^{-\delta \ell^{3/2}}.
\end{equation}
eventually for all three Cases. 
\end{prop}

The fourth and final proposition is a probabilistic interpretation of the limits in Proposition \ref{lm:main_contribution_largeL}.  
The result \eqref{eq:Sinftypb} was obtained\footnote{We need to set $\xi_i=-u_i$ in \eqref{eq:Sinfde} to find the formula (3.6) of \cite{Liu-Wang22}.} in \cite{Liu-Wang22}. 
Recall the definition of the quotient space $\Sr= \realR/\ratio \intZ$ and Brownian motions on it, discussed before Theorem \ref{thm2}.

\begin{prop} \label{propSrpr}
Let $\bfa, \bfb, \bfc \in \realR^m$ satisfying $0<a_1<\cdots<a_{m-1}<a_m$. Recall that $\phi_t(x) =\frac{1}{\sqrt{2\pi t}}e^{-\frac{x^2}{2t}}$ is the density function of the centered Gaussian distribution with variance $t>0$ 
and $\phi^{(\ratio)}_t(\{x\}_\ratio) = \sum_{k\in\intZ} \phi_t(x+k\ratio)$ is
the transition density function of a Brownian motion on $\Sr$ at time $t$, as defined in \eqref{eq:def_phi_ratio}. 
\begin{enumerate}[(a)]
\item (\cite[Lemma 3.4]{Liu-Wang22}) We have
\beq \label{eq:Sinftypb}
    S_\infty (\bfa ,\bfb) 
	= 
	  \prob\left( \bm(a_1) \ge \frac{b_1}{\sqrt{2}}, \cdots, \bm(a_{m-1}) \ge \frac{b_{m-1}}{\sqrt{2}} \mid  \bm(a_m)=\frac{b_{m}}{\sqrt{2}} \right)  \phi_{a_m}(\frac{b_{m}}{\sqrt{2}})
\end{equation}
where $\bm$ is a Brownian motion.

\item For every $\ratio \in (0,\infty)$, 
\beq \label{eq:Srpb}
  \begin{split}
	&\oint\cdots \oint S_{\ratio}(\bfa,  \bfc-\bfb ; \bfw) S_{\ratio}(\bfa,  \bfc+\bfb ; \bfw) 
	\prod_{i=2}^{m} \left( 1-\frac{w_{i-1}}{w_i} \right)
	\prod_{i=1}^m \frac{\rmd w_i}{2\pi \ii w_i} \\
	& =  \prob\left( \bigcap_{i=1}^{m-1} \bigg\{ \bm_2(a_i) -\dist_\ratio \left(\bmcir_1(a_i), \left\{ b_i\right\}_\ratio \right) \ge c_i \bigg\} 
	\,\Big|\, \bm_2(a_m) =c_m, \bmcir_1(a_m) = \left\{b_m\right\}_\ratio  \right) \phi_{a_m}\left(c_m\right) \phi^{(\ratio)}_{a_m} (\{b_m\}_\ratio)
\end{split} \eeq 
where the contours are circles satisfying $0<|w_1|<\cdots <|w_m|<1$, and $\bmcir_1$ and $\bm_2$ are independent Brownian motions on $\Sr$ and $\realR$, respectively. 
\end{enumerate} 
\end{prop}

\subsection{Proof of Theorems \ref{thm1}, \ref{thm2}, \ref{thm3}, and \ref{res:onepointtail}}

We now prove the main theorems assuming Proposition \ref{lm:main_contribution_largeL}--\ref{propSrpr}. 
In \eqref{eq:PreadyPP}, denote $\PP_{m,1}+ \PP_{m,2}+ \hat \PP_{m,1}+ \hat \PP_{m,2}= \PP_{m}$ .

\begin{proof}[Proof of Theorems \ref{thm1}, \ref{thm2}, and \ref{thm3}]
For Case 1, 
Proposition \ref{lm:main_contribution_largeL}, \ref{resultD1error}, and \ref{lm:error_estimate_largeL} imply that 
\beqq \begin{split}
	&\frac{4\ell}{p^{1/2}} e^{\frac{4}{3} \ell^{\frac32}}  \PP_{m} \to 
	S_\infty(\bft,\bfh + \bfx) S_\infty(\bft,\bfh - \bfx) . 
\end{split} \eeqq
By Proposition \ref{propSrpr} (a), recalling that $t_m=1$ and $x_m=h_m=0$, we find that 
\beqq 
\begin{split}
	\frac{ \PP_{m}}{\PP_1}
	&\to \prob\left( \bigcap_{i=1}^{m-1} \left\{ \bm'_1(t_i)\ge \frac{h_i+ x_i}{\sqrt{2}}, \, \bm'_2(t_i) \ge \frac{h_i-x_i}{\sqrt{2}}  \right\}\, \bigg| \,   \bm'_1(1)=\bm'_2(1)=0 \right)	\\
    &\quad=\prob\left( \bigcap_{i=1}^{m-1} \left\{ \sqrt{2}\min\left\{\bm'_1(t_i) -\frac{x_i}{\sqrt{2}}, \bm'_2(t_i) +\frac{x_i}{\sqrt{2}} \right\} \ge h_i  \right\} \, \bigg| \,   \bm'_1(1)=\bm'_2(1)=0 \right)	
\end{split} \eeqq 
for independent Brownian motions $\bm'_1$ and $\bm'_2$. Using the simple identity $\min\{a,b\}=\frac{a+b}{2} -\frac{|a-b|}{2}$ and noting that $\bm_1 :=(\bm'_1-\bm'_2)/\sqrt{2}$ and $\bm_2 :=(\bm'_1+\bm'_2)/\sqrt{2}$ are two independent Brownian motions, the above limit is equal to 
\begin{equation}
\prob\left(\bigcap_{i=1}^{m-1} \left\{ \bm_2(t_i) -|\bm_1(t_i)-x_i| \ge h_i  \right\}\mid   \bm_1(1)=\bm_2(1)=0\right)=\prob\left(\bigcap_{i=1}^{m-1} \left\{ \bb_2(t_i) -|\bb_1(t_i)-x_i| \ge h_i  \right\} \right)
\end{equation}
for two independent Brownian bridges $\bb_1$ and $\bb_2$. Theorem \ref{thm1} then follows from \eqref{eq:ptoper}, \eqref{eq:PreadyPP}, and Lemma~\ref{lem:equal_time}.

Similarly, for Case 2, Proposition \ref{lm:main_contribution_largeL}, \ref{resultD1error}, \ref{lm:error_estimate_largeL} and Proposition \ref{propSrpr} (b) imply that $\frac{P_m}{P_1}$ converges to 
\beq  
   \begin{split}
&	\prob\left( \bigcap_{i=1}^{m-1} \bigg\{ \bm_2(t_i) -\dist_\ratio \left(\bmcir_1(t_i), \left\{ x_i\right\}_\ratio \right) \ge h_i \bigg\} 
	\,\Big|\, \bm_2(1) =0, \bmcir_1(1) = \left\{0\right\}_\ratio  \right) \\
    &=	\prob\left( \bigcap_{i=1}^{m-1} \bigg\{ \bb_2(t_i) -\dist_\ratio \left(\bbcir_1(t_i), \left\{ x_i\right\}_\ratio \right) \ge h_i \bigg\}   \right)
\end{split} 
\eeq 
where $\bbcir_1$ is a Brownian bridge on $\Sr$, and $\bb_2$ is a Brownian bridge on $\realR$ that is independent of $\bbcir_1$. Theorem \ref{thm2} follows.

Finally, for Case 3, Proposition \ref{lm:main_contribution_largeL}, \ref{resultD1error}, \ref{lm:error_estimate_largeL} and Proposition \ref{propSrpr} (a) again show that $\frac{P_m}{P_1}$ converges to 
\beqq
\begin{split}
	\prob\left( \bigcap_{i=1}^{m-1} \left\{ \bm(2t_i)  \ge \sqrt{2} h_i\right\}   \, \bigg|   \bm(2)=0  \right)  
	=
	\prob\left( \bigcap_{i=1}^{m-1} \left\{ \bb(t_i)  \ge   h_i\right\}   \right)  
\end{split}\eeqq
where $\bm$ is a Brownian motion and $\bb$ is a Brownian bridge on $[0,1]$. 
Thus, we obtain Theorem \ref{thm3}.
\end{proof}

\begin{proof}[Proof of Theorem \ref{res:onepointtail}]
Proposition~\ref{lm:change_low_bound_n_k}, Lemma~\ref{nzeroiszeo}, and Corollary~\ref{cor:Pready} show that 
\beqq
	f_p(\ell; 0, 1)= \frac1{p^{1/2}} ( \PP_{1,1}+\PP_{1,2}+ \hat \PP_{1,1}+ \hat \PP_{1,2}) 
\eeqq
with $t_1=1$, $x_1=0$, and $\ell_1=\ell$. 
Propositions~\ref{lm:main_contribution_largeL}--\ref{lm:error_estimate_largeL} thus imply the result. 
The equality of the two formula of $\cnst(\ratio)$ is due to the Poisson summation formula, $\sum_{k\in \intZ} g(k) = \sum_{k\in \intZ} \hat{g}(k)$ with $\hat{g}(t)= \int_{-\infty}^\infty  g(x) e^{-2\pi \ii t x} \dd x$, for suitable functions $g$. 
\end{proof}

\medskip

We prove Proposition \ref{lm:main_contribution_largeL}, \ref{resultD1error}, \ref{lm:error_estimate_largeL} and  \ref{propSrpr} in Section \ref{sec:asymp} and \ref{sec:pflastprop}. 
In the next section, we prove a limit and estimates for a key function that appear in the proofs.

\section{Preparations} \label{sec:asapre}

Let $a>0$, $b\in \realR$, $c\in \mathbb{R}$, and $d\ge 0$. 
For $\ell>0$, consider the function from $\ff: \mathbb{R}\to \mathbb{C}$ defined by 
\beq \label{eq:expfunt}
	\ff(x)= 3a  \xi(x)^2 +(c-2b) \xi(x) +  \frac{1}{\ell^{3/4}}  \left(  b \xi(x)^2 - a\xi(x)^3\right) 
	\quad \text{where} \quad 
	\xi(x)= -  \frac{2(d+ \ii x)}{1+\sqrt{1+ \frac{2(d+ \ii x)}{\ell^{3/4}}}} . 
\eeq
While proving Proposition  \ref{lm:main_contribution_largeL}, \ref{resultD1error}, and \ref{lm:error_estimate_largeL}, we need to analyze the functions $E^{i, \pm}(s)$ in \eqref{eq:Edeffff}. 
In the appropriate choice of the variable $s$, $E^{i, \pm}(s)$ are related to the function $\ff$ with particular values of $a, b, c$, and $d$: see \eqref{eq:EiinG}. 
We compute a pointwise limit and uniform bounds of  $\ff(x)$ in this section.

\subsection{Pointwise limit}

\begin{lemma}[Pointwise limit] \label{lem:ptwiselm}
For every $x\in\mathbb{R}$,   
\beq \label{eq:ptl1}  
	\ff ( x_{\ell} ) \to 3a (d+ \ii x)^2 -(c-2b) (d+\ii x) 
\eeq
if  $\ell \to \infty$ and $x_{\ell}\to x$. 
The convergence is uniform for $x$ in a compact subset of $\realR$ 
and for $(a, b, c, d)$ in a compact subset of $(0,\infty)\times \realR\times \realR\times [0, \infty)$. 
\end{lemma}

\begin{proof}
It is clear since $\xi ( x_{\ell} ) \to -(d+ \ii x)$. 
\end{proof}

\subsection{A lemma}

The following simple lemma will be used in the next subsection. 

\begin{lemma} \label{lem:quadeq}
Let $A\ge 0$. 
Let $r$ be a solution to the equation $r^2- \frac{A}{r^2} =1$. 
Then, $|r| \ge  1+ \frac{2A}{3(\sqrt{13}+1)}$ if $0\le A\le 3$ and $|r| \ge  1+   \frac{A^{1/4}}{3\sqrt{2}}$ if $A\ge 1$.
\end{lemma}

\begin{proof}
Solving a quadratic equation, all solutions satisfy
\beqq
	|r| = \left(\frac{1+ \sqrt{1+4A}}{2} \right)^{1/2}
	= \left(1+ \frac{2A}{\sqrt{1+4A}+1} \right)^{1/2}.
\eeqq
Note that 
\beqq
	\sqrt{1+y} \ge \begin{dcases}
	1+\frac{y}{3} \quad & \text{for $0\le y\le 3$,}\\
	1+\frac{\sqrt{y}}{3} \quad & \text{for $y\ge 9/16$.}
\end{dcases}
\eeqq
If $0\le A\le 3$, then $\frac{2A}{\sqrt{1+4A}+1} \le A\le 3$, while if $A\ge 1$, then $\frac{\sqrt{1+4A}-1}{2} \ge \frac{\sqrt{5}-1}{2} > \frac{9}{16}$. 
Hence,
\beqq
	|r|\ge 1+ \frac{2A}{3(\sqrt{1+4A}+1)} \qquad \text{if $0\le A\le 3$}
\eeqq
and
\beqq
	|r|\ge  1+ \frac13 \left( \frac{2A}{\sqrt{1+4A}+1} \right)^{1/2} \qquad \text{if $A\ge 1$}.
\eeqq
The result follows by noting that 
\beqq
	\sqrt{1+4A}+1 \le \begin{dcases}
	\sqrt{13}+1 \quad & \text{for $0\le A\le 3$},\\
	1+ \sqrt{4A}+1 \le 4\sqrt{A} \quad & \text{for $A\ge 1$.}
\end{dcases}
\eeqq
\end{proof}

\subsection{Uniform estimates}

We find a uniform upper bound of $|e^{\ff(x)}|= e^{\Re \ff(x)}$. 
From its definition, $\xi(x)$ satisfies the equation 
\beq \label{eq:xiquadeq}
	\frac{\xi(x)^2}{\ell^{3/4}} = 2 \xi(x)+ 2(d+\ii x).
\eeq 
Thus, 
\beq \label{eq:alpha2es}
	\ff(x)= \fff_1(x)+ 2b(d+\ii x) \qquad \text{where} \quad 
	\fff_1(x) = 3a  \xi(x)^2 + c \xi(x) -  \frac{a}{\ell^{3/4}}  \xi(x)^3 . 
\eeq

Consider $\fff_1(x)$. Note that 
$\fff_1(x)= - \frac{a}{\ell^{3/4}}( \xi(x)- \ell^{3/4})^3+ 3a \xi(x) \ell^{3/4} - a \ell^{3/2} + c \xi(x)$. 
Write 
\beqq
	\frac{\xi(x)}{\ell^{3/4}} -1 = - v_x+\ii w_x. 
\eeqq
where $v_x>0$ and $w_x\in \realR$. 
The quadratic equation \eqref{eq:xiquadeq} for $\xi(x)$ implies that $v_x$ and $w_x$ satisfy 
\beq \label{eq:pqeqs}
	v_x^2-w_x^2 = 1+2d \ell^{-3/4} \quad \text{and} \quad v_x w_x= - x \ell^{-3/4} . 
\eeq 
Using the first equation, we see that 
$\Re((\xi(x)-\ell^{3/4})^3)= (-v_x^3 +3 v_x w_x^2) \ell^{9/4} = (2v_x^3 - 3 v_x) \ell^{9/4} -6dv_x \ell^{3/2}$. 
Hence, 
\beq \label{eq:Reff1bd1}
	\Re \left( \fff_1(x) \right) 
	= - 2a (v_x^3-1) \ell^{3/2}  -  ( c   - 6 a d) v_x \ell^{3/4} +c \ell^{3/4}. 
\eeq
In order to obtain an upper bound of $\Re \left( \fff_1(x) \right)$, we need an estimate of $v_x$.

\begin{lemma}\label{lem:deltes}
Define
\beq \label{eq:deltap}
	\delta_x=  \frac{v_x}{(1+ 2d \ell^{-3/4})^{1/2}} -1.
\eeq
Then, there is a constant $c_0>0$ such that if $\ell\ge c_0$, 
\beqq
	\delta_x \ge  \frac{x^2}{10 \ell^{3/2}}  \qquad \text{for $|x|\le \sqrt{3} \ell^{3/4}$}
\eeqq
and
\beqq
	\delta_x \ge  \frac{|x|^{1/2}}{5 \ell^{3/8}} \qquad \text{for $|x|\ge \frac65 \ell^{3/4}$}. 
\eeqq
\end{lemma}

\begin{proof}
From \eqref{eq:pqeqs}, $v_x^2$ satisfies the equation 
\beqq
	v_x^2 - \frac{B}{v_x^2}= C 
	\qquad \text{where $B= x^2 \ell^{-3/2}$ and $C= 1+ 2d \ell^{-3/4}$.} 
\eeqq
Let $r= C^{-1/2} v_x$ and apply Lemma \ref{lem:quadeq} with $A= \frac{B}{C^2}$. 
Note that since $v_x>0$, we have $r>0$. Also note that $r=1+\delta_x$. 
Thus, Lemma \ref{lem:quadeq} implies that 
\beqq
	\delta_x \ge \frac{2 x^2 \ell^{-3/2}}{3(\sqrt{13}+1) (1+ 2d \ell^{-3/4})^2} 
	\qquad \text{for $|x|\le \sqrt{3} \ell^{3/4}(1+ 2d \ell^{-3/4})$}
\eeqq
and
\beqq
	\delta_x \ge \frac{|x|^{1/2} \ell^{-3/8}}{3\sqrt{2} (1+2d\ell^{-3/4} )^{1/2}} \qquad \text{for $|x|\ge  \ell^{3/4}(1+ 2d \ell^{-3/4}) $}.
\eeqq
We take $\ell$ large enough so that $2d \ell^{-3/4}\le \frac15$. 
The result follows by noting that $\frac{2}{3(\sqrt{13}+1) (6/5)^2} >\frac1{10}$ and $\frac{1}{3\sqrt{2}\sqrt{6/5}}> \frac15$.
\end{proof}

From the definition \eqref{eq:deltap}, 
\beq \label{eq:pbydel}
	v_x=   (1+ 2d \ell^{-3/4})^{1/2} (1+\delta_x) . 
\eeq
In \eqref{eq:Reff1bd1}, $\Re \left( \fff_1(x) \right)$ is a cubic function of $v_x$. 
We write the linear term of $v_x$ in terms of a linear term $\delta_x$ using \eqref{eq:pbydel}. 
For the cubic term of $v_x$, we note that since $(1+x)^c\ge 1+cx$ for all $x>0$ and $c\ge 1$, 
\beqq \begin{split}
	v_x^3 =&   (1+ 2d \ell^{-3/4})^{3/2} (1+\delta_x)^3 
	\ge  (1+ 3d \ell^{-3/4})(1+3\delta_x).
\end{split} \eeqq 
Thus, since $a>0$ and $d>0$, we find that 
\beq 
	\Re \left( \fff_1(x) \right) 
	\le - 6a (1+ 3d \ell^{-3/4}) \delta_x\ell^{3/2}  -  ( c   - 6 a d) (v_x-1) \ell^{3/4}  
	\le - 6a  \delta_x\ell^{3/2}  -  ( c   - 6 a d) (v_x-1) \ell^{3/4}  
\eeq
Since  $(1+x)^{1/2}\le 1+\frac12 x$ for all $x>0$, we see from \eqref{eq:pbydel} that $v_x \le (1+ d \ell^{-3/4})(1+\delta_x)$. Therefore, we find that
\beq 
	\Re \left( \fff_1(x) \right) 
	\le \left( - 6a  \delta_x\ell^{3/2} + | c   - 6 a d| \ell^{3/4}+ | c   - 6 a d| d \right) \delta_x + | c   - 6 a d| d.  
\eeq
Thus, since $a>0$, there is a constant $c_0\ge 1$ such that if $\ell \ge c_0$, then 
\beq \label{eq:alpha1es} \begin{split}
	\Re \left( \fff_1(x) \right) \le  -5a  \ell^{3/2} \delta_x + |c-6ad|d.  
\end{split} \eeq
Since $\Re(\ff(x))= \Re(\fff_1(x)) + 2bd$ from \eqref{eq:alpha2es}, \eqref{eq:alpha1es}, and Lemma \ref{lem:deltes} imply the following bound.

\begin{lemma} \label{cor:uppes0}
Uniformly for $(a, b, c, d)$ in a compact subset of $(0,\infty)\times \realR\times \realR\times [0, \infty)$, there are constants $C>0$ and $c_0>0$ such that if $\ell \ge c_0$, then 
\beq \label{eq:uppes1}
	| e^{ \ff(x) } | \le C e^{-\frac{a}{2}  x^2} \qquad \text{for $|x|\le  \sqrt{3} \ell^{3/4} $}
\eeq
and 
\beq \label{eq:uppes2}
	| e^{ \ff(x) } | \le C e^{- a \ell^{9/8}  \sqrt{|x|}}   \qquad \text{for $|x|\ge \frac65 \ell^{3/4}$.}
\eeq

\end{lemma}

\begin{cor} \label{cor:uppes}
Let $\ff(x)$ be the function defined in \eqref{eq:expfunt}. 
Uniformly for $(a, b, c, d)$ in a compact subset of $(0,\infty)\times \realR\times \realR\times [0, \infty)$, 
there are constants $c_0\ge 1$, $c_1>0$, and $c_2>0$ such that 
\beq \label{eq:uppesall}
	| e^{ \ff(x) } | \le c_1 e^{  - c_2 \sqrt{|x|}} \qquad \text{for all $x\in \realR$} 
\eeq
and  for all $\ell\ge c_0$. 
\end{cor}

\begin{proof}
The result follows from Lemma \ref{cor:uppes0} by noting $\ell^{9/8}  \ge 1$ and $x^2+1\ge \sqrt{|x|}$ for all $x\in \realR$. 
\end{proof}

\section{Asymptotic analysis} \label{sec:asymp}

We prove Proposition \ref{lm:main_contribution_largeL},  \ref{resultD1error}, and \ref{lm:error_estimate_largeL} in this section. 
The proofs are almost uniform for all three cases except that we need to add the restriction $p\ll \ell^{5/4}$ in the proof of 
Proposition~\ref{lm:error_estimate_largeL} for Case 1. 
The remaining situation for Case 1 is handled separately at the end of this section.

\subsection{Choice of contours}

It is convenient to introduce the notation 
\begin{equation}
\label{eq:rdefine}
    \ratio=   p \ell^{1/4}   . 
\end{equation}
Note that $\ratio\to \infty$ for Case 1, $\ratio$ is a constant for Case 2, and $\ratio\to 0$ for Case 3.

The contours for the integrals of \eqref{eq:PPdefns} are circles around the origin satisfying $0<|z_1|<\cdots <|z_m|<1$. 
We make the following specific choice of the radii. The choice is the same for all three Cases except in the last subsection which we change the analysis slightly. 
Let 
\beqq
	\rho_1>\cdots>\rho_m>0
\eeqq
be real numbers which we keep fixed. We choose the contours as  
\beq \label{eq:zellrho}
	z_i = e^{- \frac{\ell p}{2}-  \ratio \rho_i + \rmi \theta_i} , \qquad \theta_i\in (-\pi, \pi],
\eeq
for each $i=1, \cdots, m$. 
Throughout this section except for the last subsection \ref{sec:proof_estimate_remaining}, we assume that $z_i$ are given by the above equation. 
We write $\bfz=(z_1, \cdots, z_m)$.


\subsection{Bound of $\bC$} 
\label{sec:bound_of_C_bullet}

The function $\bC(\bfz)$ is given by the formula \eqref{eq:bulleCdeff}. 
For every $a>0$, polylogarithm functions satisfy 
\beq \label{eq:polylogbaes}
	| \Li_{a}(z) | = \left| \sum_{n=1}^\infty  \frac{z^n}{n^{a}} \right| \le \sum_{n=1}^\infty  |z|^n  \le 2|z| \qquad \text{for $|z|\le 1/2$.} 
\eeq
Thus, if $|z|\le 1/2$, then (see \eqref{eq:A12B})
\beq \label{eq:bound_A}
	|A_1(z)|=\left| -\frac{1}{\sqrt{2\pi}} \Li_{3/2}(z)\right| \le |z| \quad \text{and} 
	\quad |A_2(z)| =\left|  -\frac{1}{\sqrt{2\pi}} \Li_{5/2}(z) \right| \le |z|.
\eeq
Similarly, for $|z|, |z'|\le 1/2$, 
\beq \label{eq:bound_B}
	|B(z,z') | = \left| \frac{1}{4\pi}\sum_{k,k'= 1}^\infty \frac{z^k(z')^{k'}}{(k+k')\sqrt{kk'}} \right|
	\le  \frac{1}{4\pi}\sum_{k,k'= 1}^\infty |z|^k |z'|^{k'} \le |z||z'|.
\eeq 
We find the following bound. 

\begin{lemma}\label{cor:wCbound}
For $\bfz$ given in \eqref{eq:zellrho}, there is a constant $c>0$ such that 
\beqq
	| \bC(\bfz) -1 | \le c \ell p^{-1/2} e^{-\frac{\ell p}{2}}  e^{c \ell p^{-1/2}e^{-\frac{\ell p}{2}}} 
\eeqq
for all $\theta\in (-\pi, \pi]^m$ and $\ell , p>0$ satisfying $\ell p\ge 2$. 
Furthermore, $|\bC(\bfz)|\le 2$ and $\bC(\bfz)\to 1$  uniformly in $\theta\in (-\pi, \pi]^m$ eventually for all three Cases. 
\end{lemma}

\begin{proof}
From \eqref{eq:zellrho}, $|z_i|\le e^{-\frac{\ell p}{2}}$. 
If $\ell p\ge 2$, then $|z_i|\le e^{-1}\le 1/2$. 
From the formula \eqref{eq:bulleCdeff} of $\bC(\bfz)$, the bounds \eqref{eq:bound_A} and \eqref{eq:bound_B}, and the choice of the parameters \eqref{eq:tagahi00}, we find, using  the inequality $|e^w-1| \le |w|  e^{|w|}$ for all complex number $w$,  that there is a constant $c>0$ so that 
\beqq
	|\bC(\bfz) -1 | \le c(\ell p^{-1/2} + p^{-3/2}) \big( \sum_{i=1}^m |z_i| \big)  e^{c (\ell p^{-1/2} + p^{-3/2}) \sum_{i=1}^m |z_i|} .
\eeqq
Since $\ell p\ge 2$, we see $\ell p^{-1/2} + p^{-3/2}\le \frac32 \ell p^{-1/2}$. 
Using $|z_i | \le e^{- \frac{\ell p}{2}}$, we obtain the bound after replacing the constant $c$ by $\frac{2c}{3m}$. 

Note that $\ell p^{-1/2}e^{-\frac{\ell p}{2}} \le \frac{1}{(\ell p^4)^{1/2}} (\ell p)^{3/2} e^{- \frac{\ell p}{2}}$. 
In all Cases, $\ell p\ge \log \ell \to \infty$. Hence, the term $(\ell p)^{3/2} e^{- \frac{\ell p}{2}}\to 0$. 
For Case 1 and 2, The term $(\ell p^4)^{1/2}$ is bounded below, and thus, $\ell p^{-1/2}e^{-\frac{\ell p}{2}}\to 0$.  
For Case 3, we have $\ell^{-1}\log \ell \ll p$. Thus, $\ell p^{-1/4} \ll \frac{\ell^{5/4}}{(\log \ell)^{1/4}}$ and $\ell p \ge 4 \log \ell$ eventually. 
Thus, $\ell p^{-1/2}e^{-\frac{\ell p}{2}} \le \frac{\ell^{5/4}}{(\log \ell)^4} e^{-2 \log \ell}\to 0$. 
Hence, $\bC(\bfz)\to 1$ for all three Cases, which also implies that $|\bC(\bfz)|\le 2$ eventually. 
\end{proof}

\subsection{The functions $\uu_i(k)$}

For $|z|<1$, a complex number $u$ is in the discrete set $\rmL_{z}= \{ u: e^{-u^2/2}=z, \, \Re(u)<0\}$ if and only if 
it is of the form $u= - \sqrt{ -2\log z + 4\pi \ii k}$ for some $k\in \intZ$. 
With  \eqref{eq:zellrho} in mind, define the function
\beq \label{eq:uellndfn}
	\uu_i(k)= \uu_i(k; \theta_i)=  - \sqrt{ \ell p +2\ratio \rho_i  -2 \rmi\theta_i  + 4\pi \ii k}, \qquad k\in \intZ
\eeq
for $i=1, \cdots, m$, where the branch of the square root is chosen so that $\Re (\uu_i(k))<0$. 
We also define 
\beq 
	\uu_i(\bfk^{(i)})= ( \uu_i(k_1^{(i)}), \cdots, \uu_{i}(k_{n_i}^{(i)})) \quad \text{for $\bfk^{(i)}= (k_1^{(i)}, \cdots, k_{n_i}^{(i)})\in \intZ^{n_i}$.} 
\eeq 
Furthermore, we write  
\beq \label{eq:UUdeffw}
	\UU(\bfk)= (\uu_1(\bfk^{(1)}), \cdots, \uu_m(\bfk^{(m)})) \quad \text{for $\bfk= (\bfk^{(1)}, \cdots, \bfk^{(m)})\in \intZ^{\bfn}
	= \intZ^{n_1}\times \cdots \times \intZ^{n_m}$} 
\eeq 
for $\bfn=(n_1, \cdots, n_m)$.

Using these notations, the functions 
in \eqref{eq:bulleD_bfz} become 
\beq \label{eq:DhatDinkj} \begin{split}
	\bD_{\bfn}(\bfz) = \sum_{\bfk, \hat\bfk \in \intZ^{\bfn}} \bs_{\bfn}(\bfk, \hat \bfk) 
	\qquad\text{and} \qquad 
	\hbD_{\bfn}(\bfz) = \sum_{\bfk, \hat\bfk \in \intZ^{\bfn}} \hbs_{\bfn}(\bfk, \hat \bfk)
\end{split} \eeq
with
\beq \label{eq:sdefhre} \begin{split} 
	&\bs_{\bfn}(\bfk, \hat \bfk):= H_{\bfn}(\UU(\bfk), \UU(\hat \bfk)) R_{\bfn}(\UU(\bfk), \UU(\hat \bfk)) E_{\bfn} (\UU(\bfk), \UU(\hat \bfk)) , \\
	&\hbs_{\bfn}(\bfk, \hat \bfk):= H_{\bfn}(\UU(\bfk), \UU(\hat \bfk)) \hat R_{\bfn}(\UU(\bfk), \UU(\hat \bfk)) E_{\bfn}(\UU(\bfk), \UU(\hat \bfk)) , 	
\end{split} \eeq
where by $\bfk\in \intZ^{\bfn}$, we mean that $\bfk= (\bfk^{(1)}, \cdots , \bfk^{(m)}) \in \intZ^{n_1}\times \cdots \times \intZ^{n_m}$.

\subsection{Bound of $H_{\bfn}$} 
\label{sec:bound_H_bfn}

The function $H_{\bfn}(U, \hat U)$ in \eqref{eq:Hdefff} involves the function
\beqq
	h(w, z )
	=  -\frac{1}{\sqrt{2\pi}} \int_{-\infty}^{w} \Li_{1/2}(z e^{(w^2-y^2)/2})\rmd y 
	=  -\frac{1}{\sqrt{2\pi}} \int_{-\infty}^{\Re (w)} \Li_{1/2}(z e^{(w^2-(x+\ii \Im(w))^2)/2})\rmd x
\eeqq
for $\Re(w)<0$. From \eqref{eq:polylogbaes}, we see that for $\Re(w)<0$ and $|z|\le 1/2$, 
\beq \label{eq:bound_h}
	| h(w, z ) |
	\le \frac{2}{\sqrt{2\pi}} \int_{-\infty}^{\Re(w)} |z|  e^{(\Re(w)^2-x^2)/2} \rmd x
	\le  |z|.
\eeq

\begin{lemma}\label{cor:qHbound}
For $\bfz$ given in \eqref{eq:zellrho}, we have  
\beqq
	|H_{\bfn}(\UU(\bfk), \UU(\hat \bfk)) -1 | \le 8  | \bfn| e^{-\frac{\ell p}{2}} e^{ 8   | \bfn| e^{-\frac{\ell p}{2}} } 
	\le 4 |\bfn| e^{4 |\bfn| }
\eeqq
for all $\bfn \in \N^m$, $\bfk, \hat \bfk\in \intZ^{\bfn}$, $\theta\in (-\pi, \pi]^m$, and $\ell, p>0$ satisfying $\ell p\ge 2$. 
As a consequence, 
\beqq
	|H_{\bfn}(\UU(\bfk), \UU(\hat \bfk))|\le 5 |\bfn| e^{4 |\bfn| }. 
\eeqq 
\end{lemma}

\begin{proof}
Using the inequality $|e^w-1| \le |w|  e^{|w|}$ and the estimate \eqref{eq:bound_h},
\beqq
	|H_{\bfn}(U, \hat U) -1 | \le 8|\bfn|  ( \max_{1\le i\le m} |z_i| )  e^{ 8|\bfn|  ( \max_{1\le i\le m} |z_i| )}  \le 4 |\bfn| e^{4 |\bfn| }
\eeqq
for all $U, \hat U \in \rmL_{z_1}\times \cdots \times \rmL_{z_m}$ if $|z_1|, \cdots, |z_m|\le 1/2$. 
For $\bfz$ given in \eqref{eq:zellrho}, $|z_i | \le e^{- \frac{\ell p}{2}}\le 1/2$ for all $i$ if $\ell p\ge 2$. 
Inserting $U=\UU(\bfk)$ and $\hat U= \UU(\hat \bfk)$, we obtain the result. 
\end{proof}

\subsection{Bound and limits of $E_{\bfn}$} 
\label{sec:bound_E_bfn}

Recall from \eqref{eq:Edeffff} and \eqref{eq:tagahi00} that for $\bfn=(n_1, \cdots, n_m)$, 
\beqq
	E_{\bfn}(U, \hat U)
	=  \prod_{i=1}^m\prod_{j_i=1}^{n_i} E^{i,+}(u_{j_i}^{(i)})E^{i,-}(\hat u_{j_i}^{(i)})  
	\quad \text{where}\quad 
	E^{i, \pm} (s)= e^{-\frac{\Delta t_i}{3p^{3/2}} s^3  \pm \frac{\Delta x_i}{2 p \ell^{1/4}}  s^2 
	+ \frac{\ell \Delta t_i +  \ell^{1/4} \Delta h_i}{p^{1/2}} s}.
\eeqq
We compute the limit and bounds of $E_{\bfn}(\UU(\bfk), \UU(\hat \bfk))$ where $\UU(\bfk)$ is the function from \eqref{eq:UUdeffw}. 
For the limit, we only need the case when $\bfn=\bfone$, and thus we do not state the results when $\bfn\neq \bfone$.

Define
\beq
	\widetilde E^{i, \pm} (s) = E^{i, \pm} (s) e^{\frac{2\Delta t_i }{3} \ell^{3/2} +   (  \Delta h_i \mp  \frac{\Delta x_i}{2}) \ell^{3/4} }.
\eeq 
Then, 
\beq \label{eq:Ebtiwhtidgn}
	E_{\bfn}(\UU(\bfk), \UU(\hat \bfk)) 
	=
	e^{-\frac{4 }{3} \ell^{3/2} (\sum_{i=1}^m n_i \Delta t_i ) - 2 \ell^{3/4} (\sum_{i=1}^m n_i \Delta h_i) }
	\prod_{i=1}^m \prod_{j=1}^{n_j} 
	\widetilde E^{i, +}(\uu_i(k_j^{(i)})) \widetilde E^{i, -}(\uu_i(\hat k_j^{(i)})) 
\eeq
where $\uu_i(k)$ is the function from \eqref{eq:uellndfn}. 
When $\bfn=\bfone$, since $\sum_{i=1}^m \dDelta h_i= 0$ and $\sum_{i=1}^m \dDelta t_i=1$, this formula becomes   
\beq \label{eq:Ebtiwhtid}
	E_{\bfone}(\UU(\bfk), \UU(\hat \bfk)) 
	= e^{ - \frac{4 }{3} \ell^{3/2}} 
	\prod_{i=1}^m \widetilde E^{i, +}(\uu_i(k_1^{(i)})) \widetilde E^{i, -}(\uu_i(\hat k_1^{(i)})). 
\eeq

The function $\widetilde E^{i, \pm} (\uu_i(k))$ is expressible in terms of the function $\ff$ from \eqref{eq:expfunt}. 
Recall the function $\xi(x)$ in \eqref{eq:expfunt}, which contains the parameter $d$. Comparing with the formula \eqref{eq:uellndfn} of $\uu_i$, we find that 
\beq 
	\uu_i(k)=  - \sqrt{ \ell p } \big( 1-  \frac{1}{\ell^{3/4} } \xi\big(\frac{2\pi k - \theta_i}{  \ratio} \big) \big) 
	\qquad \text{with $d=  \rho_i $.} 
\eeq
A direct computation shows that 
\beq \label{eq:EiinG}
	\widetilde E^{i, \pm} (\uu_i(k))
	= e^{\ff\left( \frac{2\pi k - \theta_i}{  \ratio}  \right) }
\eeq
with the parameters 
\beq \label{eq:Ealphaparach}
	a= \frac{ \dDelta t_i}{3}, \qquad b= \pm \frac{\dDelta x_i}{2}, \qquad c=  \dDelta h_i, 
	\qquad \text{and} \qquad d=  \rho_i .
\eeq
Thus, the results from Subsection \ref{sec:asapre} are applicable. 

\bigskip

We first find a limit of $E_{\bfone}(\UU(\bfk), \UU(\hat \bfk))$. 
For $\bfy=(y_1,\cdots,y_m)\in \realR^m$, we use the notation $[\bfy]=([y_1],\cdots,[y_m])$.

\begin{lemma} [Limit of $E_{\bfn}$ when $\bfn=\bfone$] \label{resultElimit}
For Case 1, for every $\bfy, \hat \bfy \in \realR^m$,  
\beq \label{eq:Elimit1}
\begin{split}
	&e^{\frac{4}{3} \ell^{3/2}} E_{\bfone} \big(\UU ([\ratio \bfy]),  \UU( [\ratio \hat \bfy] ) \big) \\
	&\qquad \to \prod_{i=1}^m e^{ \dDelta t_i (\rho_i + 2\pi \ii y_i)^2 - (\dDelta h_i - \dDelta x_i)(\rho_i + 2\pi \ii y_i))+\dDelta t_i (\rho_i + 2\pi \ii \hat y_i))^2 -  (\dDelta h_i + \dDelta x_i)(\rho_i + 2\pi \ii \hat y_i) )} .
\end{split} \eeq
uniformly in $\theta\in (-\pi, \pi]^m$.
For Case 2, for every $\bfk, \bfk'\in \intZ^m$, 
\beq  \label{eq:Elimit2}
	e^{\frac{4}{3} \ell^{3/2}} E_{\bfone}(\UU(\bfk), \UU(\hat \bfk)) 
	\to 
	\prod_{i=1}^m e^{ \dDelta t_i  \xi_i(k_i)^2 -  (\dDelta h_i - \dDelta x_i)\xi_i(k_i)+ \dDelta t_i  \xi_i(\hat k_i)^2 -  (\dDelta h_i + \dDelta x_i)\xi_i(\hat k_i)} 
\eeq
uniformly in $\theta\in (-\pi, \pi]^m$, 
where 
\beq 
\label{eq:xindefn}
	\xi_i(k) =\rho_i +  \frac{1}{\ratio} (2\pi \rmi k - \rmi \theta_i ).
\eeq
For Case 3, if 
\beqq
	\theta_i= \ratio \varphi_i 
\eeqq 
for $i=1, \cdots, m$, 
then 
\beq \label{eq:Elimit3}
	e^{\frac{4}{3} \ell^{3/2}}E_{\bfone}(\UU(0), \UU(0)) 
	\to 
    \prod_{i=1}^m e^{2\dDelta t_i ( \rho_i - \ii \varphi_i)^2 - 2 \dDelta h_i( \rho_i - \ii \varphi_i)} 
\eeq
uniformly for $\varphi=(\varphi_1,\cdots,\varphi_m)$ in a compact subset of $\realR^m$. 
\end{lemma}

\begin{proof}
From \eqref{eq:Ebtiwhtid}, it is enough to compute the limits of $\widetilde E^{i, \pm} (\uu_i(k))= e^{\ff ( \frac{2\pi k - \theta_i}{  \ratio}  ) }$. 
We use Lemma  \ref{lem:ptwiselm}. From \eqref{eq:Ealphaparach}, $d=  \rho_i $.  
Recall that $\ratio \to \infty$ for Case 1, $\ratio$ is a constant for Case 2, and $\ratio\to 0$ for Case 3. 
Thus, for Case 1, 
\beqq
	d+ \ii \frac{2\pi [\ratio y] - \theta_i}{ \ratio}
	=  \rho_i +  \ii \frac{2\pi [\ratio y] - \theta_i}{  \ratio}
	\to  \rho_i + 2\pi \ii y , 
\eeqq
and hence Lemma \ref{lem:ptwiselm} yields the result \eqref{eq:Elimit1}. 
For Case 2,  $d+ \ii \frac{2\pi k - \theta_i}{ \ratio}=  \rho_i +  \ii \frac{2\pi k - \theta_i}{  \ratio}=  \xi_i(k) $ and we obtain \eqref{eq:Elimit2}. 
For Case 3, with $\theta_i = \ratio \varphi_i$, 
\beqq
	d - \ii \frac{ \theta_i}{ \ratio}
	=  \rho_i - \ii \varphi_i  . 
\eeqq
Thus, 
we obtain \eqref{eq:Elimit3}.  
\end{proof}

\medskip

We now find a uniform bound for $E_{\bfn}(\UU(\bfk), \UU(\hat \bfk))$. We start with the following result. 

\begin{lemma} \label{resultEunifbdgngn}
There are constants $c_0\ge 1$, $c_1>0$, and $\cstar>0$ such that 
\beq \label{eq:gnEbdfsall}
	 \left| E_{\bfn}(\UU(\bfk), \UU(\hat \bfk))  \right| 
	\prod_{i=1}^m e^{\frac{4 n_i \Delta t_i }{3} \ell^{3/2} + 2  n_i \Delta h_i \ell^{3/4}}
	\le c_1^{|\bfn|} 
	\prod_{i=1}^m \prod_{j=1}^{n_i} 
	e^{  - 2\cstar \sqrt{\frac{1}{\ratio } | k^{(i)}_j -\frac{\theta_i}{2\pi} |}
	- 2\cstar  \sqrt{\frac{1}{\ratio } | \hat k^{(i)}_j  -\frac{\theta_i}{2\pi} |}} 
\eeq
for all $\bfn\in \N^m$, $\bfk, \hat \bfk \in \intZ^{\bfn}$, $\theta\in (-\pi,\pi]^m$, and $\ell\ge c_0$. 
As a consequence,  
\beq \label{eq:gnEbdfssimpler0000}
	 \left| E_{\bfn}(\UU(\bfk), \UU(\hat \bfk))  \right| 
	\prod_{i=1}^m e^{\frac{4 n_i \Delta t_i }{3} \ell^{3/2} + 2 n_i \Delta h_i\ell^{3/4} }	
	\le c_1^{|\bfn|} 
	\prod_{i=1}^m \prod_{j=1}^{n_i} 
	e^{  - \cstar \sqrt{\frac{1}{\ratio } | k^{(i)}_j |}
	-  \cstar  \sqrt{\frac{1}{\ratio } | \hat k^{(i)}_j |}} .
\eeq 
\end{lemma}

\begin{proof}
Since $\widetilde E^{i, \pm} (\uu_i(k)) = e^{\ff ( \frac{2\pi k - \theta_i}{  \ratio} ) }$  with the parameters \eqref{eq:Ealphaparach}, Corollary \ref{cor:uppes} gives a bound: there are constants $c_0\ge 1$, $c_1>0$, and $c_2>0$ such that 
\beq \label{eq:uppesall_sp}
	  | \widetilde E^{i, \pm} (\uu_i(k)) | 
	   \le  c_1 e^{  - c_2 \sqrt{\left| \frac{2\pi k - \theta_i}{ \ratio} \right|}} 
\eeq
for all $k\in \intZ$, $\theta\in (-\pi, \pi]$, and $\ell\ge c_0$. 
Thus, from \eqref{eq:Ebtiwhtidgn}, we obtain the bound \eqref{eq:gnEbdfsall} where we replaced $c_1^{2}$ by $c_1$ and $c_2\sqrt{2\pi}$ by $2\cstar$. 
The bound \eqref{eq:gnEbdfssimpler0000} follows from \eqref{eq:gnEbdfsall} since
\beq
	\left| k - \frac{\theta}{2\pi} \right| \ge \frac{|k|}{2} \ge \frac{|k|}{4} 
\eeq
for every $k\in \intZ$ and $\theta\in (-\pi, \pi]$. 
\end{proof}

When $\bfn=\bfone$, due to \eqref{eq:Ebtiwhtid}, the above result implies the next bound. 

\begin{cor} [Bound of $E_{\bfn}$ for $\bfn=\bfone$] \label{resultEunifbd}
Suppose $\bfn=\bfone$. With the same constants $c_0\ge 1$, $c_1>0$, and $\cstar>0$ in Lemma \ref{resultEunifbdgngn}, 
\beq \label{eq:Ebdfs}
	e^{ \frac{4}{3} \ell^{3/2}} \left| E_{\bfone}(\UU(\bfk), \UU(\hat \bfk))  \right| 
	\le c_1^m \prod_{i=1}^m e^{  - 2\cstar \sqrt{\frac{1}{\ratio } | k_i -\frac{\theta_i}{2\pi} |}
	- 2\cstar  \sqrt{\frac{1}{\ratio } | \hat k_i -\frac{\theta_i}{2\pi} |}} 
\eeq
and
\beq \label{eq:Ebdfssimpler}
	e^{ \frac{4}{3} \ell^{3/2}} \left| E_{\bfone}(\UU(\bfk), \UU(\hat \bfk))  \right| 
	\le c_1^m \prod_{i=1}^m e^{  - \cstar \sqrt{\frac{| k_i |}{\ratio } }
	- \cstar  \sqrt{\frac{ | \hat k_i |}{\ratio }}}  
\eeq 
for all $\bfk, \hat \bfk \in \intZ^m$, $\theta\in (-\pi,\pi]^m$, and $\ell \ge c_0$. 
\end{cor}

For the case when $\bfn\neq \bfone$, we have the following estimate. 
We use the fact that $t_1, \cdots, t_{m-1}$ are distinct. 

\begin{cor} [Bound of $E_{\bfn}$ for $\bfn\neq \bfone$] \label{lem:uniform_bound_E}
Let $\cstar>0$ be the constant from Lemma \ref{resultEunifbdgngn}. 
There are positive constants $c_0$, $\delta$, and $c_2$ such that 
\beq 
	e^{ \frac{4}{3} \ell^{3/2}} \left| E_{\bfn}(\UU(\bfk), \UU(\hat \bfk))  \right| 
	\le 
	e^{  -\frac{4\delta}{3} \ell^{3/2}-   c_2 |\bfn| \ell^{3/2} } 
	\prod_{i=1}^m \prod_{j=1}^{n_i} 
	e^{  - \cstar \sqrt{\frac{1}{\ratio } | k^{(i)}_j |}
	-  \cstar \sqrt{\frac{1}{\ratio } | \hat k^{(i)}_j |}} 
\eeq 
for all $\bfn\in \N^m\setminus \{\bfone\}$, $\bfk, \hat \bfk \in \intZ^{\bfn}$, $\theta\in (-\pi,\pi]^m$, and $\ell \ge c_0$. 
\end{cor}

\begin{proof}
Since $\dDelta t_i$ are positive constants (recall that  $t_1, \cdots, t_{m-1}$ are distinct) that add up to $1$, we find that 
\beqq
	\sum_{i=1}^m n_i \dDelta t_i  
    =1+  \sum_{i=1}^m (n_i-1) \dDelta t_i   \ge  1+ \min_{1\le i\le m} \{ \dDelta t_i \}   \qquad \text{for every $\bfn \in \N^m\setminus \{\bfone\}$.}
\eeqq
Let $c$ and $\delta$ be any positive constants satisfying $1+ \min_{1\le i\le m} \{ \dDelta t_i \} = \frac{1+\delta}{1-c}$. Then, 
\beq \label{eq:delsumtno}
	(1-c)\sum_{i=1}^m n_i \dDelta t_i \ge 1 + \delta \qquad \text{for all $\bfn \in \N^m\setminus \{\bfone\}$.}
\eeq
This inequality implies that 
\beqq
	\sum_{i=1}^m n_i \dDelta t_i \ge 1 + \delta + c \sum_{i=1}^m n_i \dDelta t_i
	\ge 1+ \delta + c |\bfn| \min_{1\le i\le m} \{ \dDelta t_i \} 
\eeqq
for all $\bfn \in \N^m\setminus \{\bfone\}$. 
Thus,
\beqq\begin{split}
	&\log \left( c_1^{|\bfn|} \prod_{i=1}^m e^{-\frac{4 n_i \Delta t_i }{3} \ell^{3/2} - 2  n_i \Delta h_i \ell^{3/4} } \right) \\
	&\le -\frac{4 (1+\delta)}{3} \ell^{3/2} -   |\bfn| \ell^{3/2} \left( \frac{4c}3  \min_{1\le i\le m} \{ \dDelta t_i \} 
	- 2  \ell^{-3/4} \max_{1\le i\le m} |\Delta h_i|  - \ell^{-3/2} \log |c_1| 
	\right) .
\end{split} \eeqq
The last parenthesis term is larger than or equal to a positive constant $c_2$ if $\ell$ is large enough. 
Thus, we obtain the result from Lemma \ref{resultEunifbdgngn} after adjusting the constant $c_0$.  
\end{proof}

\subsection{Bounds and limits of $R_{\bfn}$ and $\hat R_{\bfn}$} \label{sec:Rbounds}

From \eqref{eq:Rdetfom}, 
\beq \label{eq:HRREte}
	| R_{\bfn}(\UU(\bfk), \UU(\hat \bfk))|
	= \prod_{i=1}^m \prod_{j_i=1}^{n_i} \frac{1}{ | \uu_i (k_{j_i}^{(i)}) \uu_i (\hat k_{j_i}^{(i)}) |} 
	\prod_{i=0}^{m} \left|  \Cd( U^{(i)}, -\hat U^{(i+1)} ; \hat U^{(i)}, - U^{(i+1)})  \right| 
\eeq
with $U^{(i)}= \uu_i(\bfk^{(i)})$ and $\hat U^{(i)}= \uu_i(\hat \bfk^{(i)})$ and the convention that $U^{(0)}=\hat U^{(0)}=U^{(m+1)}=\hat U^{(m+1)}=\emptyset$. 
Recall from \eqref{eq:Caudedef} that 
\beq \label{eq:Cdemt}
	\Cd(X; Y)= \det \left( \frac{1}{x_i + y_j} \right)_{i, j=1}^{a} 
	= \frac{\prod_{1\le i<j\le a} (x_j-x_i) (y_j-y_i)}{\prod_{i, j=1}^a (x_i + y_j)}
\eeq 
for $X=(x_1, \cdots, x_a)$ and $Y=(y_1, \cdots, y_a)$. 

\medskip

From the definition \eqref{eq:uellndfn}, we have a trivial bound
\beq \label{eq:ukiklb}
	|\uu_i(k)| \ge \sqrt{ \ell p}
\eeq
for all $i=1, \cdots, m$ and $k\in \intZ$.

To estimate 
\eqref{eq:HRREte}, we need the following lemmas. 

\begin{lemma}
For every $i, i'=1, \cdots, m$ and $k, k'\in \intZ$, 
\beq \label{eq:Rest3}
	| \uu_i(k)+\uu_{i'}(k')| \ge \sqrt{ \ell p}. 
\eeq
\end{lemma}

\begin{proof}
From the definition \eqref{eq:uellndfn}, $\Re( \uu_i(k)^2)>0$ and $\Re(\uu_i(k))<0$. Thus, $\arg(-\uu_i(k))\in (-\pi/4, \pi/4)$.
Using polar forms  $- \uu_i(k)=c e^{\ii\varphi}$ 
and $-\uu_{i'} (k') = c'e^{\ii \varphi'}$ for some $c, c'>0$ and $\varphi, \varphi'\in (-\pi/4, \pi/4)$, we find that
\beqq
	| \uu_i(k)+\uu_{i'}(k')| 
	= | c+ c' e^{\ii(\varphi'-\varphi)}| \ge |c+c'\cos(\varphi-\varphi')| 
	\ge c = | \uu_i(k) | \ge \sqrt{ \ell p}
\eeqq
for all $i, i'$ and $k, k'$. The last inequality is due to \eqref{eq:ukiklb}. 
\end{proof}

\begin{lemma}
We have 
\beq \label{eq:uikupe}
	| \uu_i(k)| \le 5 \sqrt{ \ell p}  + 5 \sqrt{|k|} 
\eeq
for all $i=1, \dots, m$, $k\in \intZ$, and $\ell, p>0$ satisfying $\ell^3\ge 16\rho_1^4$  and $\ell p \ge 1$. 
\end{lemma}

\begin{proof}
If $\ell^3\ge 16\rho_1^4$, then $2\ratio \rho_i\le 2\ratio \rho_1 \le \ell p$ (recall \eqref{eq:rdefine}). Thus, (see \eqref{eq:uellndfn}) 
\beqq
	|\uu_i(k)|^4 =  ( \ell p+2\ratio \rho_i)^2 + (4\pi k-2\theta_i)^2 
	\le 4 (\ell p)^2 + (4\pi |k|+ 2\pi)^2
	\le 4 (\ell p)^2 + 32\pi^2 |k|^2+ 8\pi^2. 
\eeqq
for all $\theta_i\in (-\pi, \pi]$. Hence, for $\ell p\ge 1$, 
\beqq
	|\uu_i(k)|
	\le \left(  (4+8\pi^2) (\ell p)^2 + 32\pi^2 |k|^2 \right)^{1/4}
	\le (4+8\pi^2)^{1/4} \sqrt{ \ell p} + (32\pi^2)^{1/4}  \sqrt{|k|}. 
\eeqq
Since $(4+8\pi^2)^{1/4} \approx 3.01$ and $(32\pi^2)^{1/4} \approx 4.21$, we obtain the result. 
\end{proof}

\begin{lemma} \label{resultsumex}
\begin{enumerate}[(a)]
\item 
For every $r>0$,  $a\ge 0$ and $\epsilon>0$, 
\beq \label{eq:ss12} \begin{split}
	\frac1{r^a} \sum_{k=1}^\infty |k|^{a} e^{-\epsilon \sqrt{\frac{k}{r}}} 
	\le r \int_{\frac1{r}}^{\infty} y^a e^{-\epsilon \sqrt{\frac{y}{2}}} \dd y. 
\end{split} \eeq

\item 
Recall that $\ratio=  \ell^{1/4} p $. 
For every $\epsilon>0$ and $a\ge 0$, 
there is a positive constant $C=C(a, \epsilon)$ such that 
\beq \label{eq:estimate_sum_exp} \begin{split}
	\sum_{k=-\infty}^\infty   |\uu_{i}(k)|^a e^{-\epsilon \sqrt{\frac{|k|}{\ratio}}} 
	\le C (\ell p)^{\frac{a}{2}+1} 
\end{split} \eeq
for all $i=1, \dots, m$ and $\ell, p>0$ satisfying $\ell^3\ge 16\rho_1^4$ and $\ell p\ge 1$. 
\end{enumerate}
\end{lemma}

\begin{proof}
(a) For $k\ge 1$, we have $k\ge \frac{k+1}{2}\ge \frac{x}{2}$ for all $x\in [k, k+1]$. 
Thus,  
\beqq \begin{split}
	\frac1{r^a} \sum_{k=1}^\infty k^{a} e^{-\epsilon \sqrt{\frac{k}{r}}} 
	\le  \frac{1}{r^a} \sum_{k=1}^\infty \int_{k}^{k+1} x^{a} e^{-\epsilon \sqrt{\frac{x}{2 r}}} \dd x
	= \frac{1}{r^a} \int_1^\infty x^{a} e^{-\epsilon \sqrt{\frac{x}{2r}}} \dd x
	= r \int_{\frac1{r}}^{\infty} y^a e^{-\epsilon \sqrt{\frac{y}{2}}} \dd y. 
\end{split} \eeqq

(b) 
The result (a) implies that 
\beqq
	\frac1{\ratio^a} \sum_{k=-\infty}^\infty |k|^{a} e^{-\epsilon \sqrt{\frac{|k|}{\ratio}}} 
	\le \delta_{a=0}+ 2 \ratio B_a
	\quad \text{where $B_a=  \int_0^\infty y^{a} e^{- \epsilon\sqrt{\frac{y}{2}}} \dd y$.} 
\eeqq
Hence, if $\ell^3\ge 16\rho_1^4$ and $\ell p\ge 1$, then \eqref{eq:uikupe} implies that 
\beqq
	\sum_{k=-\infty}^\infty   |\uu_{i}(k)|^a e^{-\epsilon \sqrt{\frac{|k|}{\ratio}}} 
	\le 5^a \sum_{k=-\infty}^\infty   (2^a (\ell p)^{a/2}  + 2^a |k|^{a/2} ) e^{-\epsilon \sqrt{\frac{|k|}{\ratio}}} 
	\le 	10^a \left( (\ell p)^{a/2}(1+ 2 \ratio B_0)  + 2 \ratio^{a/2+1} B_{a/2} \right) .
\eeqq
Since  $\ell^3\ge 16\rho_1^4$ implies that $\ratio \le \frac{\ell p}{2\rho_1}$,  
the above is bounded by a constant times $(\ell p)^{a/2+1}$ if $\ell^3\ge 16\rho_1^4$ and $\ell p\ge 1$.  
We thus obtain the result. 
\end{proof}

\begin{lemma} 
\label{lem:longdetra1}
For every $\epsilon>0$, there is a positive constant $C_0$ such that 
\beq \label{eq:longdetra1}
	\frac{\left| \Cd( \uu_i(\bfk), - \uu_{i'}(\hat \bfk') ; \uu_{i}(\hat \bfk) , -\uu_{i'}(\bfk') ) \right| }
	{\prod_{j=1}^{n} |u_i(\hat k_j)|\prod_{j=1}^{n'} |u_{i'}(k'_j)| }
	\prod_{j=1}^n e^{ -\epsilon \sqrt{ \frac1{\ratio} | k_{j}|} - \epsilon  \sqrt{\frac1{\ratio} | \hat k_{j}|} } \prod_{j=1}^{n'} e^{ - \epsilon \sqrt{ \frac1{\ratio} |k_{j}'|} -  \epsilon  \sqrt{ \frac1{\ratio} |\hat k_{j}'|} }
	\le  \left( \frac{C_0 (\ell p)^2}{\ratio^2} \right)^{\frac{n+n'}{2}} 
\eeq
for all  two distinct integers $i$ and $i'$ from $\{0, \cdots, m+1\}$,  
$n, n'\in\N$, $\bfk, \hat \bfk\in \intZ^{n}$, $\bfk', \hat \bfk'\in \intZ^{n'}$, 
and $\ell, p>0$ satisfying $\ell^3\ge 4\rho_1^4$ and $\ell p\ge 1$. 
\end{lemma}

\begin{proof}
Since $\Cd$ is a Cauchy determinant \eqref{eq:Cdemt}, the left-hand side of \eqref{eq:longdetra1} is zero if two components of any of $\bfk, \bfk', \hat \bfk$, or $\hat \bfk'$ are equal. 
Thus, it is enough to consider the case that $\bfk, \bfk', \hat \bfk$, or $\hat \bfk'$ all have distinct components. Let $\epsilon>0$ be an arbitrary constant. 

For vectors $X=(x_1, \cdots, x_a)$ and $Y=(y_1, \cdots, y_a)$, and scalars $f_1, \cdots, f_a$, 
the Hadamard's inequality implies that 
\beqq
	\left| \Cd(X; Y) \prod_{\qq=1}^a f_\qq \right| 
	= \left| \det\left( \frac{f_\qq}{x_{\pp}+y_\qq} \right)_{\pp,\qq=1}^a \right|
	\le \prod_{\pp=1}^a \sqrt{\sum_{\qq=1}^a \frac{f_\qq^2}{(x_{\pp}+y_\qq)^2} }.
\eeqq
Thus,  
\beq	\label{eq:Kbdlng} \begin{split}
	& \frac{\left| \Cd( \uu_i(\bfk), - \uu_{i'}(\hat \bfk') ; \uu_{i}(\hat \bfk) , -\uu_{i'}(\bfk') ) \right| }
	{\prod_{j=1}^{n} |u_i(\hat k_j)|\prod_{j=1}^{n'} |u_{i'}(k_j')| }
	\prod_{j=1}^n e^{- \epsilon  \sqrt{\frac1{\ratio} | \hat k_{j}|}} \prod_{j'=1}^{n'} e^{ - \epsilon  \sum_{j=1}^{n'} \sqrt{\frac1{\ratio} |k_{j}'|} } \\	  
	&\le 
	 \prod_{\pp=1}^{n} \sqrt{ \sum_{\qq=1}^{n} \frac{ e^{- 2\epsilon \sqrt{\frac1{\ratio} | \hat k_{\qq}|}} }{ |u_{i}(k_{\pp}) + u_{i}(\hat k_\qq)|^2 |u_i(\hat k_\qq)|^2} 
	 + \sum_{\qq=1}^{n'} \frac{ e^{- 2\epsilon \sqrt{\frac1{\ratio} |  k_{\qq}'|}} }{ |u_{i}(k_{\pp}) - u_{i'}( k_\qq')|^2 |u_{i'}( k_\qq')|^2}  } \\
	&\quad \times 
	 \prod_{\pp=1}^{n'} \sqrt{ \sum_{\qq=1}^{n} \frac{ e^{- 2\epsilon \sqrt{\frac1{\ratio} | \hat k_{\qq} |}} }{ |u_{i'}(\hat k_{\pp}') - u_{i}(\hat k_\qq)|^2|u_i(\hat k_\qq)|^2} 
	 + \sum_{\qq=1}^{n'} \frac{ e^{- 2\epsilon \sqrt{ \frac1{\ratio} |  k_{\qq}' |}} }{ |u_{i'}(\hat k_{\pp}') + u_{i'}( k_\qq')|^2 |u_{i'}( k_\qq')|^2}  } .
\end{split} \eeq 
Consider the first sum. 
From \eqref{eq:ukiklb} and \eqref{eq:Rest3}, $|u_i(k)|\ge \sqrt{\ell p}$ and $|u_i(k)+u_{i'}(k')| \ge \sqrt{\ell p}$. 
Since we assume that the components of $\hat \bfk$ are distinct, the case $a=0$ of \eqref{eq:estimate_sum_exp} implies that there is a constant $C_1>0$ so that 
\beq \label{eq:uplub1}
	\sum_{\qq=1}^{n} \frac{ e^{- 2\epsilon \sqrt{\frac1{\ratio} | \hat k_{\qq}|}} }{ |u_{i}(k_{\pp}) + u_{i}(\hat k_\qq)|^2 |u_i(\hat k_\qq)|^2} \le 
	\frac{1}{(\ell p)^2} \sum_{j =1}^{n} e^{-  2\epsilon \sqrt{\frac1{\ratio} | \hat k_{j}|}}  	
	\le 
	\frac{1}{(\ell p)^2} \sum_{k =-\infty}^{\infty} e^{- 2\epsilon \sqrt{\frac{| k |}{\ratio} }}   \le \frac{C_1}{\ell p}
\eeq
$\ell^3\ge 16\rho_1^4$ and $\ell p\ge 1$. 
The same bound holds for the fourth sum. 
For the second sum, note that $\frac1{|a-b|^{2}}\le  \frac{2(1+|a|^2)(1+|b|^2)}{|a^2-b^2|^2}$ for all complex $a, b$. 
Since
\beq
\label{eq:Jan17_02}
	|u_{i}(k_{\pp})^2 - u_{i'}( k_\qq')^2|= | 2\ratio \rho_{i} + \ii (-2\theta_{i}+4\pi k_{\pp}) - 2\ratio \rho_{i'} - \ii (-2\theta_{i'}+4\pi k_{\qq}')  | 
	\ge 2\ratio | \rho_i - \rho_{i'}|, 
\eeq
the $a=0$ and $a=2$ cases of \eqref{eq:estimate_sum_exp} show that there is a constant $C_2>0$ such that 
\beq 
\label{eq:main_term_long_det}\begin{split}
	 &\sum_{\qq=1}^{n} \frac{ e^{- 2\epsilon \sqrt{\frac1{\ratio} |  k_{\qq}' |}} }{ |u_{i}(k_{\pp}) - u_{i'}( k'_\qq)|^2|u_{i'}( k'_\qq)|^2}
	 \le 
	  \frac{ (1+|u_{i}(k_{\pp})|^2)}{2\ratio^2 (\rho_i- \rho_{i'})^2 \ell p} 
	  \sum_{j=1}^{n} (1+|u_{i'}(k'_{j})|^2) e^{-  2\epsilon \sqrt{\frac1{\ratio} |  k_{j}'|}}  \\
	  &\qquad\qquad \qquad
	  \le 
	  \frac{ (1+|u_{i}(k_{\pp})|^2)}{2\ratio^2 (\rho_i- \rho_{i'})^2 \ell p} 
	  \sum_{k=-\infty}^{\infty} (1+|u_{i'}(k)|^2) e^{- 2\epsilon\sqrt{ \frac{|k|}{\ratio} }}  
	  \le   \frac{C_2 \ell p  }{\ratio^2}  (1+|u_{i}(k_{\pp})|^2)
\end{split} \eeq
if $\ell^3\ge 16\rho_1^4$  and $\ell p\ge 1$. 
The third sum is similar. 
Hence, the left-hand side of \eqref{eq:longdetra1} is bounded by 
\beq \label{eq:bnbnddtem2}
	 \prod_{j=1}^n e^{ -\epsilon \sqrt{ \frac1{\ratio} | k_{j}|}  } \prod_{j=1}^{n'} e^{ - \epsilon \sqrt{ \frac1{\ratio} |\hat k_{j}'|} }
	\left[ \prod_{\pp=1}^{n} 
	\sqrt{ \frac{C_1}{\ell p} +  \frac{C_2 \ell p  }{\ratio^2}  (1+|u_{i}(k_{\pp})|^2) } \right]
	 \left[ \prod_{\pp=1}^{n'} 
	 \sqrt{ \frac{C_1}{\ell p} +  \frac{C_2 \ell p }{\ratio^2}  (1+|u_{i'}(\hat k_{\pp}')|^2)} \right] .
\eeq
From \eqref{eq:uikupe}, since $e^{- 2\epsilon \sqrt{ \frac{|  k|}{\ratio}}} \le 1$ and $e^{- 2\epsilon \sqrt{ \frac{|  k|}{\ratio}}} \frac{|k|}{\ratio} \le \max \{ x e^{- 2\epsilon \sqrt{x}} : x\ge 0\} <\infty$, 
with additional constants $C_3$ and $C_4$, 
\beq \label{eq:abt}
	e^{- 2 \epsilon \sqrt{\frac1{\ratio} |  k_{\pp}|}} 
	\left( \frac{C_1}{\ell p} +   \frac{C_2 \ell p  }{\ratio^2}  (1+|u_{i}(k_{\pp})|^2) \right) 
	\le \frac{C_1}{\ell p} +   \frac{C_2 \ell p }{\ratio^2}
	+ \frac{C_3 (\ell p)^2}{\ratio^2} + \frac{C_4 \ell p}{\ratio}.
\eeq 
Since $\ell^3\ge 16\rho_1^4$  implies that $\frac{\ell p}{\ratio} \ge 2\rho_1$, 
there is a positive constant $C_0$ so that 
the right-side of \eqref{eq:abt} is bounded by $ \frac{C_0 (\ell p)^2}{\ratio^2}$ if $\ell^3\ge 16\rho_1^4$  and $\ell p\ge 1$. 
Hence, \eqref{eq:bnbnddtem2} is bounded by $(\frac{C_0 (\ell p)^2}{\ratio^2} )^{(n+n')/2}$. 
\end{proof}

\begin{cor}[Bound of $R_{\bfn}$ and $\hat R_{\bfn}$] \label{resultRRhatbd}
For every $\epsilon>0$, there is a positive constant $C_0$ such that 
\beq \label{eq:bound_R}
	| R_{\bfn}(\UU(\bfk), \UU(\hat \bfk))|
	\le \left( \frac{C_0 (\ell p)^2}{\ratio^2} \right)^{|\bfn|}
	\prod_{i=1}^m \prod_{j=1}^{n_i} 
	 e^{2\epsilon \sqrt{\frac1{\ratio} | k_{j}^{(i)}|}+ 2\epsilon \sqrt{ \frac1{\ratio} |\hat k_{j}^{(i)}|} }
\eeq
and
\beq \label{eq:bound_hatR}
	| \hat R_{\bfn}(\UU(\bfk), \UU(\hat \bfk))|
	\le |\bfn| (\ell p)^{1/2} \left( \frac{C_0 (\ell p)^2}{\ratio^2} \right)^{|\bfn|}
	\prod_{i=1}^m \prod_{j=1}^{n_i} 
	 e^{ 2\epsilon \sqrt{ \frac1{\ratio} | k_{j}^{(i)}|}+ 2\epsilon \sqrt{ \frac1{\ratio} |\hat k_{j}^{(i)}|} }
\eeq 
for all $\bfn\in \N^m$, $\bfk, \hat \bfk \in \intZ^{\bfn}$, and $\ell, p>0$ satisfying  $\ell^3\ge 16\rho_1^4$  and $\ell p\ge 1$. 

\end{cor}

\begin{proof}
The bound \eqref{eq:bound_R} follows by inserting the estimate \eqref{eq:longdetra1} in the formula \eqref{eq:HRREte}. 
For the bound \eqref{eq:bound_hatR}, we need to modify the argument a little bit. Recall that $ \hat R_{\bfn}(\UU(\bfk), \UU(\hat \bfk))$ is equal to $R_{\bfn}(\UU(\bfk), \UU(\hat \bfk))$ times the sum $\sum_{j=1}^{n_m}(u_m(k^{(m)}_{j})+u_m(\hat k^{(m)}_{j}))$ . We may assume that $k_j^{(m)}$ are distinct for $1\le j\le n_m$, and $\hat k_j^{(m)}$ are also distinct since otherwise the left-hand side is zero. 
Using the lower bound $|u_i(k)|\le 5\sqrt{\ell p}+5\sqrt{k}$ in \eqref{eq:uikupe}, 
\beq
\label{eq:bound_extra_term} \begin{split}
	&\left|  \sum_{j=1}^{n_m}(u_m(k^{(m)}_{j})+u_m(\hat k^{(m)}_{j}))  \right| \\
	&\le \left(10n_m\sqrt{\ell p} + 5\sum_{j=1}^{n_m}\sqrt{k_j^{(m)}}e^{-\epsilon\sqrt{\frac{1}{\ratio} k_j^{(m)}}}
	+ 5\sum_{j=1}^{n_m}\sqrt{\hat k_j^{(m)}}e^{-\epsilon\sqrt{\frac{1}{\ratio} \hat k_j^{(m)}}}\right) 
	\prod_{j=1}^{n_m} 
	 e^{ \epsilon \sqrt{ \frac1{\ratio} | k_{j}^{(m)}|}+ \epsilon \sqrt{ \frac1{\ratio} |\hat k_{j}^{(m)}|} }\\
\end{split} \eeq
Note that the maximum of the function $\sqrt{x}e^{-\epsilon\sqrt{x/\ratio}}$ over $x\in[0,\infty]$ is $C\sqrt{\ratio} $. Here $C$ is a positive constant. Also note that $2\rho_1\ratio<\ell p$ by our assumption $\ell^3\ge 16\rho_1^4$. Hence the left hand side of~\eqref{eq:bound_extra_term} is bounded by a constant times $|\bfn| |\ell p|^{1/2} \prod_{i=1}^m \prod_{j=1}^{n_i} 
	 e^{ \epsilon \sqrt{ \frac1{\ratio} | k_{j}^{(i)}|}+ \epsilon \sqrt{ \frac1{\ratio} |\hat k_{j}^{(i)}|} }
$. Combining with the estimate \eqref{eq:bound_R} and adjusting the $\epsilon$ value accordingly we obtain \eqref{eq:bound_hatR}. 
\end{proof}

The exponential bounds of Corollary \ref{resultRRhatbd} are enough for $\bfn\neq\bfone$. 
However, for $\bfn=\bfone$, we need a stronger estimate. In the next lemma, we obtain a polynomial bound in this case.  
Note that when $\bfn=\bfone$, the product formula of the Cauchy determinant implies that 
\beq \label{eq:RRmmq} \begin{split}
       &R_{\bfone}(U, \hat U) = (-1)^m 
        \prod_{i=1}^m \frac{1}{(u_i+\ut_i)^2 u_i \ut_i } 
	\prod_{i=2}^m \frac{(u_i + \ut_{i-1})(\ut_i+u_{i-1})}{(u_i-u_{i-1})(\ut_i-\ut_{i-1})}
\end{split} \eeq
and $\hat R_{\bfone}(U, \hat U) = (u_m+\ut_m)  R_{\bfone}(U, \hat U)$. 
We insert $U= \UU(\bfk)$ and $\hat U= \UU(\hat \bfk)$ where $\bfk=(k_1, \cdots, k_m)\in \N^m$ and $\hat \bfk=(\hat k_1, \cdots, \hat k_m)\in \N^m$.

\begin{lemma}[Bound of $R_{\bfn}$ and $\hat R_{\bfn}$ for $\bfn=\bfone$] \label{lem:Rest}
There is a polynomial $P$ of $2m$ variables such that 
\beqq
	(\ell p)^{2}  \ratio^{2m-2} | R_{\bfone}(\UU(\bfk), \UU(\hat \bfk)) | 
	\le  | P \big( \frac{\bfk}{\ratio}, \frac{\hat{\bfk}}{\ratio} \big) |
	\quad \text{and} \quad 
	(\ell p)^{3/2} \ratio^{2m-2} | \hat R_{\bfone}(\UU(\bfk), \UU(\hat \bfk)) | 
	\le  | P \big( \frac{\bfk}{\ratio}, \frac{\hat{\bfk}}{\ratio} \big) |
\eeqq
for all $\bfk, \hat \bfk\in \intZ^m$, $\theta\in (-\pi,\pi]^m$, and $\ell, p$ satisfying  
 $\ell^3\ge 16\rho_1^4$  and $\ell p\ge 1$. 
\end{lemma}

\begin{proof} 
Recall the trivial bound $|\uu_i(k)|\ge \sqrt{\ell p}$ from  \eqref{eq:ukiklb} 
for all $i$ and $k$ and the bound from \eqref{eq:Rest3} that 
$| \uu_i(k)+\uu_{i'}(k')| \ge \sqrt{ \ell p}$ for all $i, i'$ and $k, k'$. 
The bound \eqref{eq:uikupe} implies that 
\beqq
	\big| \uu_i(k) \big| 
	 \le  
	 5 \sqrt{\ell p} \left( 1 + \sqrt{ \frac{|k|}{\ell p}  }  \right) 
     \le 5 \sqrt{\ell p} \left( 1 + \sqrt{ \frac{ |k|}{2\rho_1\ratio}  }  \right) 
\eeqq
if  $\ell^3\ge 16\rho_1^4$  and $\ell p \ge 1$. 
On the other hand, since $|\uu_i(k)^2- \uu_{i'}(k')^2| = | 2\ratio \rho_i  +\ii (- 2\theta_i + 2\pi k)  - 2\ratio \rho_{i'}  - \ii (- 2\theta_{i'} + 2\pi k')|\ge 2\ratio |\rho_i-\rho_{i'}|$, we have 
\beq \label{eq:Rest4}
	\left| \frac1{ \uu_i(k)- \uu_{i'}(k')} \right| 
	= \left| \frac{\uu_i(k)+ \uu_{i'}(k')}{\uu_i(k)^2- \uu_{i'}(k')^2} \right|
	\le \frac{|\uu_i(k)| + | \uu_{i'}(k')| } { 2\ratio  | \rho_i -\rho_{i'} |}
\eeq
for all $k, k'\in \N$, and $i\neq i'$. 
Inserting these estimates into \eqref{eq:RRmmq}, we obtain the desired inequalities. 
\end{proof} 

The proof shows that we also have the bound given by $P \big( \frac{\bfk}{\ell p} , \frac{\bfk}{\ell p} \big)$. 
For a later convenience, we replaced it by a less precise bound $P \big( \frac{\bfk}{\ratio}, \frac{\hat{\bfk}}{\ratio} \big)$. 

\bigskip

We also need pointwise limits of $\hat R_{\bfn}$ when $\bfn=\bfone$. 

\begin{lemma}[Limit of $\hat R_{\bfn}$ for $\bfn=\bfone$] \label{resultRlimit}
For Case 1, for every $\bfy, \hat \bfy\in \realR^m$, 
\beq \label{eq:Rlmt1} \begin{split}
	(-1)^{m-1} 2(\ell p)^{3/2}\ratio^{2m-2} \hat R_{\bfone} \big(\UU ([\ratio \bfy]),  \UU( [\ratio \hat \bfy] ) \big)
	\to 
	\prod_{i=2}^m 
	 \frac{1}{ (\rho_i + 2\pi \ii y_i - \rho_{i-1} - 2\pi \ii y_{i-1})(\rho_i + 2\pi \ii \hat y_i - \rho_{i-1}-  2\pi \ii \hat y_{i-1})}   
\end{split} \eeq
uniformly for $\theta\in (-\pi, \pi]^m$. For Case 2, for every $\bfk, \hat \bfk\in \intZ^n$,  
\beq \label{eq:Rlmt2} \begin{split}
	(-1)^{m-1} 2(\ell p)^{3/2} \ratio^{2m-2} \hat R_{\bfone}(\UU(\bfk), \UU(\hat \bfk'))
	\to 
	 \prod_{i=2}^m 
	 \frac{1}{ (\xi_i(k_i)-\xi_{i-1}(k_{i-1}))(\xi_i(\hat k_i)-\xi_{i-1}(\hat k_{i-1}))}  
\end{split} \eeq
uniformly for $\theta\in (-\pi, \pi]^m$, where $\xi_i(k) =\rho_i +  \frac{1}{\ratio} ( 2\pi \ii k - \rmi \theta_i)$ as in \eqref{eq:xindefn}. 
For Case 3, if 
\beqq
	\theta_i= \ratio \varphi_i, \qquad i=1, \cdots, m, 
\eeqq
then 
\beq \label{eq:Rlmt3} \begin{split}
	(-1)^{m-1} 2(\ell p)^{3/2} \ratio^{2m-2} \hat R_{\bfone}(\UU(\bfzero), \UU(\bfzero))
	\to 
	\prod_{i=2}^m 
	 \frac{1}{ (\rho_i-\ii\varphi_i - \rho_{i-1}+\ii\varphi_{i-1})^2}  
\end{split} \eeq
uniformly for $\varphi$ is a compact subset of $\realR^m$. 
\end{lemma}

\begin{proof}
From the definition \eqref{eq:uellndfn} of $u_i(k)$, 
\beqq 
	u_i(k) 
	=-\sqrt{\ell p +2\ratio \rho_i- 2 \ii\theta_i  +4\pi\ii k}
	=-\sqrt{\ell p +2\ratio \xi_i(k)}
	=-\sqrt{\ell p} \left( 1+\frac{2 \ratio}{\ell p}  \xi_i(k) \right)^{1/2}
\eeqq
using $\xi(k) =\rho_i +  \frac{1}{\ratio} ( 2\pi \ii k - \rmi \theta_i )$. 
Hence, for every $k$, 
\beqq 
	u_i (k)
	=-\sqrt{\ell p} \left( 1+ O\left( \frac{1}{\ell^{3/4}} \right)  \right) 
\eeqq
uniformly in $\theta_i\in (-\pi, \pi]$. Also for every $i\neq i'$ and $k, k'$, 
\beqq 
	\frac1{u_i (k) - u_{i'} (k')  } 
	= 	\frac{u_i (k) + u_{i'}(k') }{u_i (k)^2 - u_{i'} (k')^2 } 
	= \frac{-2\sqrt{\ell p} \left( 1+ O\left( \frac{1}{\ell^{3/4}} \right)  \right) }{2\ratio (\xi_i(k) - \xi_{i'}(k')) } 
\eeqq
uniformly in $\theta_i, \theta_{i'}\in (-\pi, \pi]$. 
Inserting them into $\hat R_{\bfone}(\UU(\bfk), \UU(\hat \bfk'))$ using the formula \eqref{eq:RRmmq}, 
we find that for every $\bfk, \bfk'\in \intZ$, 
\beq \label{eq:Rlmttm} \begin{split}
	(-1)^{m-1}  2(\ell p)^{3/2}  \ratio^{2m-2} \hat R_{\bfone}(\UU(\bfk), \UU(\hat \bfk'))
	= 
	\left[ \prod_{i=2}^m 
	 \frac{1}{ (\xi_i(k_i)-\xi_{i-1}(k_{i-1}))(\xi_i(k_i')-\xi_{i-1}(k_{i-1}'))} \right] 
	 \left( 1+ O\left( \frac{1}{\ell^{3/4}} \right)  \right)
\end{split} \eeq
uniformly for $\theta\in (-\pi, \pi]^m$. 
The result \eqref{eq:Rlmttm} implies \eqref{eq:Rlmt2} for Case 2. 
For Case 1, we have $\ratio \to \infty$, and thus, 
\beqq
	\xi_i([\ratio y]) = \rho_i +  \frac{- \rmi \theta_i  + 2\pi \rmi [\ratio y]}{\ratio} 
	\to \rho_i +2\pi \ii y
\eeqq
for every $y\in\realR$. Hence, \eqref{eq:Rlmttm} implies \eqref{eq:Rlmt1}.  
If $\theta_i=\ratio \varphi_i$, then 
\beqq
	\xi_i(0) =\rho_i +  \frac{- \rmi \theta_i}{\ratio}  = \rho_i -\ii \varphi_i. 
\eeqq
Thus, \eqref{eq:Rlmt3} follows from \eqref{eq:Rlmttm} after inserting $\bfk=\hat \bfk=\bfzero$. 
\end{proof}

\subsection{Proof of Proposition~\ref{lm:main_contribution_largeL} }  \label{sec:prooflimit}

We analyze 
\beq \label{eq:ChatD} \begin{split}
	\hat \PP_{m,1} &= (-1)^{m-1} \int_{(-\pi, \pi]^m}  \bC(\bfz) \hat D^{\bullet}_{\bfone}(\bfz) \bT_{\bfone}(\bfz) 
	\prod_{i=1}^m \frac{\rmd \theta_i}{2\pi}
\end{split}\eeq
where  $z_i= e^{- \frac{\ell p}{2}-  \ratio \rho_i + \rmi \theta_i}$, $\theta_i\in (-\pi, \pi]$, as given in recall \eqref{eq:zellrho}. 
From \eqref{eq:bfTdefff} when $\bfn=\bfone$, 
\beq \label{eq:tZonelim}
	\bT_{\bfone}(\bfz) =  \prod_{i=2}^m \left(1-\frac{z_{i-1}}{z_i} \right)  
	= \prod_{i=2}^m \left(1-\frac{e^{- \ratio  \rho_{i-1} + \rmi \theta_{i-1} }}{e^{- \ratio  \rho_{i} + \rmi \theta_{i}} } \right) .
\eeq 
Since $\rho_1>\cdots>\rho_m>0$, we find 
\beq \label{eq:theprodababd}
	| \bT_{\bfone}(\bfz)| \le 2^m .
\eeq

Recall from \eqref{eq:DhatDinkj} that
\beq \label{eq:hDdeffff} \begin{split}
	\hbD_{\bfone}(\bfz) = \sum_{\bfk, \hat\bfk \in \intZ^{m}} \hbs_{\bfone}(\bfk, \hat \bfk)
\end{split} \eeq
where 
$\hbs_{\bfone}(\bfk, \bfk')= H_{\bfone}(\UU(\bfk), \UU(\hat \bfk)) \hat R_{\bfone}(\UU(\bfk), \UU(\hat \bfk)) E_{\bfone}(\UU(\bfk), \UU(\hat \bfk))$. 
Lemma \ref{cor:qHbound}, Corollary \ref{resultEunifbd}, and Lemma \ref{lem:Rest} imply that 
there are constants $c_0\ge 1, c_*>0$ and a polynomial $P$ of $2m$ variables such that  
\beq \label{eq:HREbd}
\begin{split}
	& (\ell p)^{3/2}  \ratio^{2m-2} e^{ \frac{4}{3}\ell^{3/2}}
	\left| \hbs_{\bfone} (\bfk, \hat \bfk)  \right|  
	\le 
	| P \big( \frac{\bfk}{\ratio}, \frac{\hat{\bfk}}{\ratio} \big) |
	\prod_{i=1}^m e^{  - \cstar \sqrt{\frac{| k_i |}{\ratio } }
	- \cstar \sqrt{\frac{ | \hat k_i |}{\ratio }}} 
\end{split} \eeq
for all $\bfk, \hat \bfk\in \intZ^m$, $\theta\in (-\pi, \pi]^m$, and $\ell, p>0$ satisfying $\ell \ge c_0$ and $\ell p\ge 2$.

\subsubsection{Case 2}

For Case $2$, $\ratio=  \ell^{1/4} p $ is a constant. 
Thus, the right-hand side of \eqref{eq:HREbd} gives a uniform upper bound, independent of  $\ell$ and $p$, that is summable.  
Therefore, by the dominated convergence theorem, Lemma \ref{cor:qHbound} and equations \eqref{eq:Elimit2} and \eqref{eq:Rlmt2} imply that 
\beqq \begin{split}
	(-1)^{m-1} \frac{4}{\ratio^2} (\ell p)^{3/2} e^{ \frac{4}{3} \ell^{3/2}} \hat D^{\bullet}_1(\bfz) 
	&\to 
	\frac2{\ratio^{2m}} \sum_{\bfk, \hat \bfk\in \intZ^m}
	\prod_{i=2}^m  \frac{1}{ (\xi_i(k_i)-\xi_{i-1}(k_{i-1}))(\xi_i(\hat k_i)-\xi_{i-1}(\hat k_{i-1}))}    
	 \\
	&\qquad  \times  \prod_{i=1}^m e^{ \dDelta t_i  \xi_i(k_i))^2 - (\dDelta h_i - \dDelta x_i)\xi_i(k_i)}
\prod_{i=1}^me^{ \dDelta t_i  \xi_i(\hat k_i)^2 - (\dDelta h_i + \dDelta x_i)\xi_i(\hat k_i)}
\end{split}\eeqq
uniformly for $\theta\in (-\pi, \pi]^m$. 
Furthermore, the left-hand side is uniformly bounded in $\ell, p, \theta$. 
The limit factorizes to the product of two series, and we find from Definition \ref{defnS} that it is equal to
\beq \label{eq:Dlimmm2} \begin{split}
	S_{\ratio}(\bft,\bfh-\bfx; \bfw) S_{\ratio}(\bft,\bfh+\bfx; \bfw)  \quad \text{where} \quad w_i= e^{-\ratio \rho_i + \ii \theta_i} . 
\end{split}\eeq 

Consider the limit of \eqref{eq:ChatD}.
Lemma \ref{cor:wCbound} shows that $\bC(\bfz)\to 1$ uniformly in $\theta$. Thus, the above limit for  $\hat D^{\bullet}_1(\bfz)$  implies that 
\beqq 
\begin{split}
	\frac{4}{\ratio^2} (\ell p)^{3/2} e^{ \frac{4}{3} \ell^{3/2}}  
	\hat \PP_{m,1}
	\to \int_{(-\pi, \pi]^m}  
	S_{\ratio}(\bft,\bfh-\bfx; \bfw) S_{\ratio}(\bft,\bfh+\bfx; \bfw) 
	\prod_{i=2}^m \left(1-\frac{e^{- (r\rho_{i-1}-\rmi \theta_{i-1})}}{e^{- (r\rho_{i}-\rmi \theta_{i})} } \right) 
	\prod_{i=1}^m  \frac{\dd \theta_i}{2\pi}
\end{split}\eeqq 
where $w_i= e^{-\ratio \rho_i + \ii \theta_i}$. 
Changing the variables $\theta_i$ to $w_i$, this proves Proposition \ref{lm:main_contribution_largeL} for Case 2.

\subsubsection{Case 1}

For Case 1, we write the series \eqref{eq:hDdeffff} as an integral of a piecewise constant function, 
\beq \label{eq:haDasHRE2} \begin{split}
	\hat D^{\bullet}_{\bfone}(\bfz)
	&= \int_{\realR^m} \int_{\realR^m}
	\hbs([\bfy] , [\hat \bfy]) \dd \bfy \dd \hat \bfy 
	=  \ratio^{2m} \int_{\realR^m} \int_{\realR^m}
	\hbs([ \ratio \bfy ], [ \ratio \hat \bfy ] ) \dd \bfy \dd \hat \bfy .
\end{split} \eeq
Consider the bound \eqref{eq:HREbd} and insert $[ \ratio y_i]$ for  $k_i$ and
$[ \ratio y_i']$ for  $k_i'$. 
Since $\ratio \to \infty$ for Case 1, we may assume that $\ratio \ge 1$. Then 
\beqq
	\frac{|y|}{2}\le  \frac{ | [\ratio y] | }{\ratio}  \le 2|y| \qquad \text{for $|y|\ge 2$.} 
\eeqq
Thus, the estimate \eqref{eq:HREbd} implies an $\ratio$-independent upper bound,  
\beqq \label{eq:case1HREbd} \begin{split}
	& (\ell p)^{3/2}  \ratio^{2m-2} e^{ \frac{4}{3} \ell^{3/2}}
	\left| \hbs ([ \ratio \bfy ], [ \ratio \hat \bfy ] )  \right|  
	\le | \widetilde P(\bfy, \bfy') |
	\prod_{i=1}^m e^{  -  c_* \sqrt{ \frac{| y_i |}{2} } - c_* \sqrt{\frac{| y_i'|}{2} } }
\end{split} \eeqq
where $\widetilde{P}(\bfy, \bfy')$ is a polynomial of $\bfy, \bfy'\in \realR^m$ that does not depend on $\ratio$ and $\theta$. 
Therefore, the dominated convergence theorem, Lemma \ref{cor:qHbound}, and equations \eqref{eq:Rlmt1} and \eqref{eq:Elimit1} imply 
\beqq  \begin{split}
	&(-1)^{m-1} \frac{4}{\ratio^2} (\ell p)^{3/2} e^{ \frac{4}{3} \ell^{3/2}} 
	\hat D^{\bullet}_{\bfone}(\bfz) \\
	&\to \int_{\realR^m} \int_{\realR^m}
	\prod_{i=2}^m 
	 \frac{1}{ (\rho_i + 2\pi \ii y_i - \rho_{i-1} - 2\pi  \ii y_{i-1})(\rho_i + 2\pi  \ii y_i' - \rho_{i-1}-  2\pi \ii y_{i-1}')} \\
	&\qquad \qquad \times  \prod_{i=1}^m e^{ \dDelta t_i (\rho_i + 2\pi \ii y_i)^2 - (\dDelta h_i - \dDelta x_i)(\rho_i +2\pi  \ii y_i )}
	\prod_{i=1}^m e^{ \dDelta t_i (\rho_i + 2\pi \ii y_i')^2 - (\dDelta h_i + \dDelta x_i)(\rho_i +2\pi  \ii y_i')}\dd \bfy \dd \bfy' .
\end{split}\eeqq
Note that the $\bfy$-integrals and the $\bfy'$-integrals factorize. 
Changing the variables $\rho_i+ 2\pi  \ii y_i= \xi_i$, the $\bfy$-integral is equal to $(-1)^{m-1}S_\infty(\bft,\bfh-\bfx)/\sqrt{2}$ 
of Definition \ref{defnS}. Similarly, the $\bfy'$-integral is equal to $(-1)^{m-1}S_\infty(\bft,\bfh+\bfx)/\sqrt{2}$. 
Note that the order of the contours comes from condition $\rho_1>\cdots>\rho_m$. Thus, 
\beq \label{eq:D1c1001} 
\begin{split}
	&(-1)^{m-1} \frac{4}{\ratio^2} (\ell p)^{3/2} e^{ \frac{4}{3} \ell^{3/2}}  \hat D^{\bullet}(\theta) 
	\to S_\infty(\bft,\bfh-\bfx) S_\infty(\bft,\bfh+\bfx) 
\end{split}\eeq 
uniformly for $\theta\in (-\pi, \pi]^m$, and the limit does not depend on $\theta$.

From Lemma \ref{cor:wCbound}, $\bC(\bfz)\to 1$ uniformly in $\theta$. 
On the other hand, since $\ratio\to\infty$ for Case 1, 
\beqq
	\bT_{\bfone}(\bfz)
	= \prod_{i=2}^m \left(1-\frac{z_{i-1}}{z_i} \right)  
	= \prod_{i=2}^m \left(1-\frac{e^{- \ratio  \rho_{i-1} + \rmi \theta_{i-1} }}{e^{- \ratio  \rho_{i} + \rmi \theta_{i}} } \right) \to 1
\eeqq
uniformly in $\theta$ as well. 
Thus, 
\beqq  
\begin{split}
	\frac{4}{\ratio^2} (\ell p)^{3/2} e^{ \frac{4}{3} \ell^{3/2}}  \hat \PP_{m,1}\to  
	S_\infty(\bft,\bfh-\bfx) S_\infty(\bft,\bfh+\bfx) .
\end{split}\eeqq 
This proves Proposition \ref{lm:main_contribution_largeL} for Case 1.

\subsubsection{Case 3}

For Case 3, $\ratio \to 0$. 
We change the variables $\theta_i= \ratio \varphi_i$ so that \eqref{eq:ChatD} becomes 
\beq  \begin{split}
	\hat \PP_{m,1} =  (-1)^{m-1} \frac{\ratio^m}{(2\pi)^m} \int_{\realR^m}   \bC(\bfz) \hat D^{\bullet}_{\bfone}(\bfz) 
	\bT_{\bfone}(\bfz) 
	\prod_{i=1}^m  1_{(-\frac{\pi}{\ratio}, \frac{\pi}{\ratio}]} (\varphi_i)\dd \varphi_i, 
	\qquad
	z_i= e^{- \frac{\ell p}{2}-  \ratio \rho_i +  \ratio \rmi \varphi_i}.
\end{split}
\eeq 
Write 
\beq \label{eq:IIII3}
	 \frac{2}{\ratio}  (\ell p)^{3/2} e^{ \frac{4}{3} \ell^{3/2} } \hat \PP_{m,1} 
	= \sum_{\bfk, \hat\bfk \in \intZ^{m}} \int_{\realR^m} Q_{\ratio}(\varphi; \bfk, \hat\bfk) \prod_{i=1}^m \frac{\dd \varphi_i}{2\pi}
\eeq 
where
\beq \label{eq:I3Qr}
	Q_{\ratio}(\varphi; \bfk, \hat \bfk) 
	=
	2\ratio^{m-1} (-1)^{m-1} (\ell p)^{3/2} e^{ \frac{4}{3} \ell^{3/2}}
	\bC(\bfz)  \hbs_{\bfone}(\bfk, \hat \bfk)
	\bT_{\bfone}(\bfz) 
	\prod_{i=1}^m  1_{(-\frac{\pi}{\ratio}, \frac{\pi}{\ratio}]} (\varphi_i). 
\eeq

By Lemma \ref{cor:wCbound}, $\bC(\bfz)\to 1$ uniformly as $\ell\to\infty$ and $\ell p\gg \log\ell$. Thus, we may assume that $|\bC(\bfz)|\le 2$.  
For the term $\bT_{\bfone}(\bfz)$, the estimate \eqref{eq:theprodababd} is not enough for Case 3. We need a better estimate. 
For every $\varphi\in \realR^m$, 
\beq 
\label{eq:bT_limit_0}
\begin{split}
	\frac{\bT_{\bfone}(\bfz)}{\ratio^{m-1}}    
	= \frac1{\ratio^{m-1}} 
	\prod_{i=2}^m \left(1-\frac{e^{- \ratio \rho_{i-1} + \rmi \ratio \varphi_{i-1}}}{e^{- \ratio \rho_{i} + \rmi \ratio \varphi_{i}} } \right) 
	\to 
	 (-1)^{m-1} \prod_{i=2}^m  (\rho_i-\rmi\varphi_i -\rho_{i-1}+\rmi \varphi_{i-1}) . 
\end{split}\eeq
Since $|1-e^{w}|\le |w|$ for complex numbers $w$ satisfying $\Re(w)\le 0$, we also see that 
\beq \label{eq:Tbqqq} \begin{split}
	\frac{ |\bT_{\bfone}(\bfz) | }{\ratio^{m-1}} 
	\le
	\prod_{i=2}^m | \rho_{i-1}-\rho_i - \ii (\varphi_{i-1}-\varphi_i)| 
	\qquad \text{for all $\varphi \in \realR^m$.}
\end{split}\eeq 
Thus, 
\beq \label{eq:Tbig}
	\frac{|\bT_{\bfone}(\bfz) |}{\ratio^{m-1}} 
	\prod_{i=1}^m  1_{(-\frac{\pi}{\ratio}, \frac{\pi}{\ratio}]} (\varphi_i)
	\le \big( \rho_1-\rho_m + \frac{2\pi}{\ratio} \big)^{m-1}
\eeq
since $\rho_1>\cdots>\rho_m$. 
Using the estimate \eqref{eq:HREbd} for $\hbs_{\bfone}(\bfk, \hat \bfk)$, we find that 
\beq \label{eq:I3Qr0}
	|Q_{\ratio}(\varphi; \bfk, \hat \bfk) |
	\le 
	4\big( \rho_1-\rho_m + \frac{2\pi}{\ratio} \big)^{m-1}
	| P \big( \frac{\bfk}{\ratio}, \frac{\hat{\bfk}}{\ratio} \big) |
	\prod_{i=1}^m e^{  - \cstar \sqrt{\frac{| k_i |}{\ratio } } - \cstar \sqrt{\frac{ | \hat k_i |}{\ratio }}} 
	\prod_{i=1}^m  1_{(-\frac{\pi}{\ratio}, \frac{\pi}{\ratio}]} (\varphi_i)
\eeq
for all $\varphi\in \realR^m$. 

For the sum over $(\bfk, \hat\bfk) \neq (\bfzero, \bfzero)$ in \eqref{eq:IIII3}, we find, after integrating over $\varphi_i$s, that 
\beqq \begin{split}
	&\sum_{(\bfk, \hat\bfk) \in \intZ^{2m} \setminus \{(\bfzero, \bfzero)\}} \int_{\realR^m} | Q_{\ratio}(\varphi; \bfk, \hat\bfk)| \prod_{i=1}^m \frac{\dd \varphi_i}{2\pi} \\
	&\le 4 \big( \rho_1-\rho_m + \frac{2\pi}{\ratio} \big)^{m-1}
	\left( \frac{2\pi}{\ratio} \right)^m 
	\sum_{\bfk, \hat\bfk \in \intZ^{m}\setminus \{(\bfzero, \bfzero)\}} | P \big( \frac{\bfk}{\ratio}, \frac{\hat{\bfk}}{\ratio} \big) |
	\prod_{i=1}^m e^{  - \cstar \sqrt{\frac{| k_i |}{\ratio } } - \cstar \sqrt{\frac{ | \hat k_i |}{\ratio }}} . 
\end{split} \eeqq
Recall that $\ratio \to 0$ for Case 3. 
Lemma \ref{resultsumex} (a) implies that 
for any non-negative integer $\ell$, there is a constant $C'_\ell>0$ such that 
\beqq\begin{split}
	\sum_{k\in \intZ\setminus\{0\}} \big( \frac{|k|}{\ratio} \big)^{\ell} e^{-\cstar \sqrt{\frac{|k|}{\ratio}}} 
	\le 2\ratio \int_{\frac1{\ratio}}^{\infty} y^{\ell} e^{-\cstar  \sqrt{\frac{y}{2}}} \dd y
	\le \frac{C_\ell'}{\ratio^{\ell-1/2}}  e^{-  \frac{\cstar}{ \sqrt{2\ratio }}}     
\end{split} \eeqq
and 
\beqq \begin{split}
	\sum_{k\in \intZ} \big( \frac{|k|}{\ratio} \big)^{\ell} e^{- \cstar \sqrt{\frac{|k|}{\ratio}}} 
	\le 1+ 
    \frac{C_\ell'}{\ratio^{\ell-1/2}}  e^{-  \frac{\cstar}{ \sqrt{2\ratio }}} 
    \le 2
\end{split} \eeqq
for all small enough $\ratio>0$.
Therefore, there are a positive constant $C$ and a non-negative integer $n$ so that 
\beqq \begin{split}
	&\sum_{ (\bfk, \hat\bfk) \in \intZ^{2m}\setminus \{(\bfzero, \bfzero)\}} \int_{\realR^m} | Q_{\ratio}(\varphi; \bfk, \hat\bfk)| \prod_{i=1}^m \frac{\dd \varphi_i}{2\pi} 
	\le 
	 \frac{C}{\ratio^n}  e^{- \frac{\cstar}{ \sqrt{2\ratio}}} .
\end{split} \eeqq
Thus, the series tends to $0$ as $\ratio \to 0$. 

We now consider the term for $\bfk=\hat \bfk=\bfzero$ in \eqref{eq:IIII3},  $\int_{\realR^m} Q_{\ratio}(\varphi; \bfzero, \bfzero) \prod_{i=1}^m \frac{\dd \varphi_i}{2\pi}$. 
In the derivation of \eqref{eq:HREbd}, we used \eqref{eq:Ebdfssimpler}. 
We now use the bound \eqref{eq:Ebdfs} instead to find   
\beq  \label{eq:REFHema} \begin{split}
	&\ratio^{2m-2} 
	 (\ell p)^{3/2} e^{ \frac{4}{3} \ell^{3/2} }
	\left| \hbs_{\bfone}(\bfk, \hat \bfk)  \right|  
	\le 
	| P \big( \frac{\bfk}{\ratio}, \frac{\hat{\bfk}}{\ratio} \big) |
	\prod_{i=1}^m e^{  - 2\cstar \sqrt{\frac{1}{\ratio } | k_i -\frac{\ratio \varphi_i}{2\pi} |}
	- 2\cstar  \sqrt{\frac{1}{\ratio } | \hat k_i -\frac{\ratio \varphi_i}{2\pi} |}} .
\end{split} \eeq
Thus, when $\bfk=\hat \bfk=\bfzero$, there is a constant $C>0$ such that 
\beq	
	\ratio^{2m-2}  (\ell p)^{3/2} e^{ \frac{4}{3} \ell^{3/2} }
	\left| \hbs_{\bfone}(\bfzero, \bfzero)  \right|  
	\le C \prod_{i=1}^m e^{  - \frac{4\cstar}{\sqrt{2\pi}} \sqrt{|\varphi_i|} } .
\eeq	
Using this estimate in \eqref{eq:I3Qr}, and also using \eqref{eq:Tbqqq} and the fact that $|\bC(\bfz)|\le 2$, 
\beqq
	|Q_{\ratio}(\varphi; \bfzero, \bfzero) |\le  4C \prod_{i=2}^m  | \rho_{i-1}-\rho_i - \ii (\varphi_{i-1}-\varphi_i)| 
	\prod_{i=1}^m e^{  - \frac{4\cstar}{\sqrt{2\pi}} \sqrt{|\varphi_i|} } .
\eeqq
Since the upper bound is absolutely integrable and does not depend on $\ell , p$, we can apply the dominated converge theorem to evaluate the integral of $Q_{\ratio}(\varphi; \bfzero, \bfzero)$.  
Recall $\hbs_{\bfone}(\bfzero, \bfzero)= H_{\bfone}(\UU(\bfzero), \UU(\bfzero)) \hat R_{\bfone}(\UU(\bfzero), \UU(\bfzero)) E_{\bfone}(\UU(\bfzero), \UU(\bfzero))$ in~\eqref{eq:sdefhre}. 
Lemma~\ref{cor:qHbound} implies that $H_{\bfone}(\UU(\bfzero), \UU(\bfzero))\to 1$. 
Thus, ~\eqref{eq:Elimit3},~\eqref{eq:Rlmt3}, and~\eqref{eq:bT_limit_0} imply, also using $\bC(\bfz)\to 1$, that 
\beqq  \begin{split}
	\int_{\realR^m} Q_{\ratio}(\varphi; \bfzero, \bfzero) \prod_{i=1}^m \frac{\dd \varphi_i}{2\pi} 
	\to
	& \frac{(-1)^{m-1}}{(2\pi)^m} \int_{\realR^m}   
	\prod_{i=2}^m  \frac{1}{ \rho_i-\ii\varphi_i - \rho_{i-1}+\ii\varphi_{i-1}}  
	\prod_{i=1}^m e^{2\dDelta t_i ( \rho_i - \ii \varphi_i)^2 - 2 \dDelta h_i( \rho_i - \ii \varphi_i) )}
	\prod_{i=1}^m   \dd \varphi_i . 
\end{split}\eeqq 
The limit is $S_\infty(2\bft, 2\bfh)/\sqrt{2}$ in Definition \ref{defnS}.

Combining all together, we conclude that
$\frac{2\sqrt{2}}{\ratio}  (\ell p)^{3/2} e^{ \frac{4}{3} \ell^{3/2} } \hat \PP_{m,1} \to S_\infty(2\bft, 2\bfh)$. 
Thus, we proved Proposition \ref{lm:main_contribution_largeL} for Case 3.

\subsection{Proof of Proposition \ref{resultD1error}}

The formula of $\PP_{m,1}$ and $\hat \PP_{m,1}$ are similar: 
\beqq \begin{split}
	\PP_{m,1} = \frac{(-1)^{m-1}}{(2\pi \ii)^m}  \oint \cdots \oint  A_1(z_m) \bC(\bfz) \bD_{\bfone}(\bfz) \bT_{\bfone}(\bfz) \prod_{i=1}^m \frac{\rmd z_i}{z_i} 
\end{split}\eeqq
and 
\beqq
	\hat \PP_{m,1} = \frac{(-1)^{m-1}}{(2\pi \ii)^m} \oint \cdots \oint  \bC(\bfz) \hat D^{\bullet}_{\bfone}(\bfz) \bT_{\bfone}(\bfz) \prod_{i=1}^m \frac{\rmd z_i}{z_i}.
\eeqq
In the previous section on the analysis of $\hat \PP_{m,1}$, all upper bounds were obtained from absolute value estimates. 
In $\PP_{m,1}$, there is an additional decay factor due to (see  \eqref{eq:bound_A}) 
\beqq
	| A_1(z_m) |\le |z_m| \le e^{- \frac{\ell p}{2}} 
\eeqq
and the fact that $\ell p\to\infty$ for all three Cases. 
Furthermore, the term $D_{\bfone}(\bfz)$ involves $R^{\bullet}_{\bfone}(\UU(\bfk), \UU(\hat \bfk))$
while $\hat D_{\bfone}(\bfz)$ contains $\hat R^{\bullet}_{\bfone}(\UU(\bfk), \UU(\hat \bfk) )$. 
By Lemma \ref{lem:Rest}, we find that an estimate of $R^{\bullet}_{\bfone}(\UU(\bfk), \UU(\hat \bfk))$ is the $\frac1{(\ell p)^{1/2}}$ times the estimate of $\hat R^{\bullet}_{\bfone}(\UU(\bfk), \UU(\hat \bfk)) $. 
Thus, in  all estimates obtained in the last sections for $\hat D_{\bfone}(\bfz)$, we can multiply $\frac1{(\ell p)^{1/2}}$  to obtain an estimate for $D_{\bfone}(\bfz)$. 
Due to these two factors, since $|\hat \PP_{m,1}|$ is uniformly bounded in all three cases, we find that $|\PP_{m,1}|$ is of order $\frac{e^{- \frac{\ell p}{2}}}{(\ell p)^{1/2}}  $ is all three cases. 
This proves Proposition \ref{resultD1error}.

\subsection{Proof of Proposition \ref{lm:error_estimate_largeL} when $p\ll \ell^{5/4}$} \label{sec:prooferrors}

We prove Proposition \ref{lm:error_estimate_largeL} for Case 2 and 3 as well as Case 1 under the extra assumption that $p\ll \ell^{5/4}$ in this section
and prove remaining part of Case 1 in the next section. 
The assumption $p\ll \ell^{5/4}$ will be used only when we simplify~\eqref{eq:PPPane} at the very end of  the analysis. 

Recall \eqref{eq:DhatDinkj} and \eqref{eq:sdefhre}. 
Lemma \ref{cor:qHbound}, Corollary \ref{lem:uniform_bound_E}, and Corollary \ref{resultRRhatbd} imply a bound for $\bs_{\bfn}(\bfk, \hat\bfk)$ and  $\hbs_{\bfn}(\bfk, \hat\bfk)$. 
Let $\cstar>0$ be the constant from Lemma \ref{resultEunifbdgngn} that appears in Corollary \ref{lem:uniform_bound_E}.
When applying Corollary  \ref{resultRRhatbd}, we use the constant $\epsilon= \frac{\cstar}{2}$. 
Thus, we find that there are positive constants $c_0, c_2, \cstar, \delta$ and $C_0$ such that 
\beq \label{eq:Stt}
	e^{ \frac{4}{3} \ell^{3/2}}  |\bs_{\bfn}(\bfk, \hat\bfk)| 
	\le 5 |\bfn| e^{4 |\bfn| } 
	\left( \frac{C_0 (\ell p)^2}{\ratio^2} \right)^{|\bfn|}
	e^{  -\frac{4\delta}{3} \ell^{3/2} -   c_2 |\bfn| \ell^{3/2} } 
	\prod_{i=1}^m \prod_{j=1}^{n_i} 
	e^{  - \frac{\cstar}{2} \sqrt{\frac{1}{\ratio } | k^{(i)}_j |}
	-  \frac{\cstar}{2}  \sqrt{\frac{1}{\ratio } | \hat k^{(i)}_j |}} 
\eeq
for all $\bfn \in \N^m\setminus \{\bfone\}$, $\bfk, \hat \bfk\in \intZ^{\bfn}$, $\theta\in (-\pi, \pi]^m$, and $L,T>0$ satisfying $\ell \ge c_0$ and $\ell p\ge 2$. 
We also have a similar estimate for $\hbs_{\bfn}(\bfk, \hat \bfk)$ where we need to multiply $|\bfn|\ell p$ due to the difference between \eqref{eq:bound_hatR} and \eqref{eq:bound_R}. 

Consider the series \eqref{eq:DhatDinkj} which are sums over $\bfk, \hat \bfk\in \intZ^{\bfn}$.
Since $\bs_{\bfn}(\bfk, \hat\bfk) =\hbs_\bfn(\bfk, \hat\bfk) =0$ if two components of any one of $k_1, \cdots, k_m , \hat k_1, \cdots, \hat k_m$ are equal (due to the Cauchy determinants in $R_{\bfn}(\bfk, \hat \bfk)$), it is enough to take sums over indices of distinct components. 
Thus, noting $\sum_{i=1}^m n_i=|\bfn|$,  
\beqq \begin{split}
	e^{ \frac{4}{3} \ell^{3/2}} \bD_{\bfn}(\bfz) 
	\le 5|\bfn| e^{4 |\bfn| }  \left( \frac{C_0 (\ell p)^2}{\ratio^2} \right)^{|\bfn|}
	e^{  -\frac{4\delta}{3} \ell^{3/2}  -   c_2 |\bfn| \ell^{3/2}} 
	\left( \sum_{k=-\infty}^\infty e^{  - \frac{\cstar}{2} \sqrt{\frac{| k |}{\ratio } }} \right)^{2|\bfn|} . 
\end{split} \eeqq
The sum can be estimated using the $a=0$ case of \eqref{eq:estimate_sum_exp}, and we find that
\beq \label{eq:bDuhb} \begin{split}
	e^{ \frac{4}{3} \ell^{3/2}} \bD_{\bfn}(\bfz) 
	\le 5 |\bfn| 
	 \left( \frac{C_0 (\ell p)^4}{\ratio^2} \right)^{|\bfn|}
	e^{  -\frac{4\delta}{3} \ell^{3/2}  -   c_2 |\bfn| \ell^{3/2} } 
\end{split} \eeq
where the constant $C_0$ is modified from the last equation. 
We also have a similar estimate for $\hbD_{\bfn}(\bfz)$ where we need to multiply $|\bfn| \ell p$. 

From \eqref{eq:PPdefns}, 
\beqq \begin{split}
	|e^{ \frac{4}{3} \ell^{3/2}} \PP_{m,2} |& \le  \sum_{\bfn \in \N^m\setminus\{\bfone\} }  \frac1{(\bfn !)^2} \int_{(-\pi, \pi]^m} | A_1(z_m)\bC(\bfz) \bD_{\bfn}(\bfz)  \bT_{\bfn} (\bfz) |  
	\prod_{i=1}^m \frac{\rmd \theta_i }{2\pi}, \\
	| e^{ \frac{4}{3} \ell^{3/2}} \hat \PP_{m,2} | & \le \sum_{\bfn \in \N^m\setminus\{\bfone\} }  \frac1{(\bfn !)^2} 
	\int_{(-\pi, \pi]^m}|  \bC(\bfz) \hat D^{\bullet}_{\bfn}(\bfz)  \bT_{\bfn}(\bfz)  | \prod_{i=1}^m \frac{\rmd \theta_i}{2\pi},
\end{split}\eeqq
where $z_i= e^{-\frac{\ell p}{2} -\ratio \rho_i + \ii \theta_i}$. 
By \eqref{eq:bound_A}, $|A_1(z_m)|\le |z_m|\le 1$. 
By Lemma \ref{cor:wCbound}, $|\bC(\bfz)|\le 2$ for all three Cases eventually. 
Using the formula of $z_i$, since $\rho_1>\cdots>\rho_m$, we see that \eqref{eq:bfTdefff} satisfies 
\beq
\label{eq:bound_bT}
	| \bT_{\bfn}(\bfz) |
	=  
	\left| \prod_{i=2}^m \left(1-\frac{z_{i-1}}{z_i} \right)^{n_i}
	\left(1-\frac{z_{i}}{z_{i-1}} \right)^{n_{i-1}-1} \right|
	\le \prod_{i=2}^m 2^{n_i} (1+e^{\ratio (\rho_{i-1}-\rho_i)})^{n_{i-1} }
	\le 2^{2|\bfn|} e^{c' \ratio |\bfn|}
\eeq
where $c'=\max\{ \rho_{i-1}-\rho_i : 2\le i\le m\}>0$. Note that this estimate contains an exponential function and is very loose but it is sufficient when we assume that $p\ll \ell^{5/4}$. 

Thus, with a new positive constant $C_0$, 
\beq \label{eq:PPPane} \begin{split}
	|e^{ \frac{4}{3} \ell^{3/2}} \PP_{m,2} |& \le 
	e^{  -\frac{4\delta}{3} \ell^{3/2}} \sum_{\bfn \in \N^m\setminus\{\bfone\} }  \frac{|\bfn|}{(\bfn !)^2}  
	 \left( \frac{C_0 (\ell p)^4}{\ratio^2} \right)^{|\bfn|} e^{c' \ratio |\bfn|}
	e^{ -   c_2  |\bfn| \ell^{3/2}} .
\end{split}\eeq

Since we assume that $p\ll \ell^{5/4}$, we have $\ratio=p\ell^{1/4} \ll \ell^{3/2}$ and $(\ell p)^4/\ratio^2= p^2\ell^{7/2}\ll \ell^6$. 
Recall that $\ell\to\infty$ for all three Cases. 
Thus the sum on the right hand side of~\eqref{eq:PPPane} is convergent and uniformly bounded for all three cases. Note that $\frac{\ell}{p^{1/2}}\ll \ell^{3/2}$ since $\ell p\to\infty$.
This proves first result of Proposition \ref{lm:error_estimate_largeL}. 
An estimate of $\hat \PP_{m,2}$ is similar; the summand in \eqref{eq:PPPane} is multiplied by $|\bfn|\ell p$. 
This change does not affect the proof much and we obtain the second result of Proposition \ref{lm:error_estimate_largeL}.


\subsection{Proof of Proposition \ref{lm:error_estimate_largeL} when $p\gg \ell$}\label{sec:proof_estimate_remaining}

Case 1 is when $\ell^{-1/4} \ll p$ and $\log p\ll \ell^{3/2}$. 
We prove Proposition \ref{lm:error_estimate_largeL} for Case 1 when $p\ll \ell^{5/4}$ does not hold. 
The proof given here applies to the situation when $p$ and $\ell$ satisfy $p\gg \ell$ and $\log p\ll \ell^{3/2}$. 
Note that we have $\ell$ and $\ratio$ both tend to infinity in this case.

The main reason that we added the assumption  $p\ll \ell^{5/4}$ in the last section is the factor $e^{c'\ratio |\bfn|}$ in~\eqref{eq:PPPane} which comes from the estimate \eqref{eq:bound_bT} of $| \bT_{\bfn}(\bfz) |\le 2^{2|\bfn|} e^{c' \ratio |\bfn|}$. 
In order to improve this estimate, we modify the integral contours. 
In \eqref{eq:zellrho}, the contours were chosen as 
\beqq
	z_i = e^{- \frac{\ell p}{2}-  \ratio \rho_i + \rmi \theta_i} , \qquad \theta_i\in (-\pi, \pi],
\eeqq
where $\rho_1>\cdots>\rho_m>0$ were fixed numbers. In this section, we choose these numbers to be dependent on $\ratio$:  
\begin{equation}
	\rho_i=\rho_1-\frac{i-1}{\ratio} 
\end{equation}
for $1\le i\le m$, where $\rho_1$ is a a fixed positive number. 
With this change, the estimate~\eqref{eq:bound_bT} is changed to 
\beq \label{eq:bound_bT2}
	| \bT_{\bfn}(\bfz) |
	\le 2^{2|\bfn|} e^{c'|\bfn|}.
\eeq
The difference is that the exponent is changed from $c' \ratio |\bfn|$ to $c' |\bfn|$, which gives a much tighter bound. 
However, we need to check how other quantities in the estimate \eqref{eq:bound_bT} change due to the contour changes.

The estimates in Sections~\ref{sec:bound_of_C_bullet} and~\ref{sec:bound_H_bfn} are still valid without any change. 
For the estimates in Section~\ref{sec:bound_E_bfn}, note that $d= \rho_i = \rho_1 -\frac{i-1}{ \ratio}$  which depends on $\ratio$ but is close to the constant $ \rho_1 $. 
Since Corollary~\ref{cor:uppes} holds uniformly on $d$, Lemma~\ref{resultEunifbdgngn} and Corollary~\ref{lem:uniform_bound_E} still hold.
However, the estimates in Section \ref{sec:Rbounds} need some changes. 

Lemma~\ref{lem:longdetra1} is changed to the following estimate. 

\begin{lemma} 
\label{lem:longdetra2}
For every $\epsilon>0$, there is a positive constant $C_0$ such that 
\beq \label{eq:longdetra2}
	\frac{\left| \Cd( \uu_i(\bfk), - \uu_{i'}(\hat \bfk') ; \uu_{i}(\hat \bfk) , -\uu_{i'}(\bfk') ) \right| }
	{\prod_{j=1}^{n} |u_i(\hat k_j)|\prod_{j=1}^{n'} |u_{i'}(k'_j)| }
	\prod_{j=1}^n e^{ -\epsilon \sqrt{ \frac1{\ratio} | k_{j}|} - \epsilon  \sqrt{\frac1{\ratio} | \hat k_{j}|} } \prod_{j=1}^{n'} e^{ - \epsilon \sqrt{ \frac1{\ratio} |k_{j}'|} -  \epsilon  \sqrt{ \frac1{\ratio} |\hat k_{j}'|} }
	\le  C_0^{\frac{n+n'}{2}} 
\eeq
for all  two distinct integers $i$ and $i'$ from $\{0, \cdots, m+1\}$,  
$n, n'\in\N$, $\bfk, \hat \bfk\in \intZ^{n}$, $\bfk', \hat \bfk'\in \intZ^{n'}$, 
and $\ell, p>0$ satisfying $\ell^3\ge 4\rho_1^4$ and $\ell p\ge 1$. 
\end{lemma}

\begin{proof}
Recall that it is enough to consider the case when $\bfk, \bfk', \hat \bfk$, or $\hat \bfk'$ all have distinct components. 
In the proof of Lemma~\ref{lem:longdetra1}, the estimates~\eqref{eq:Kbdlng} and~\eqref{eq:uplub1} still hold. However, we need to change the estimate~\ref{eq:main_term_long_det}.  
Since the components of $k'_\qq$ are all distinct, we have  
\begin{equation} \label{eq:Jan17_01}
	\sum_{\qq=1}^{n} \frac{ e^{- 2\epsilon \sqrt{\frac1{\ratio} |  k_{\qq}' |}} }{ |u_{i}(k_{\pp}) - u_{i'}( k'_\qq)|^2|u_{i'}( k'_\qq)|^2}
	\le 
	\sum_{\qq=1}^{n} \frac{1}{|u_{i}(k_{\pp}) - u_{i'}( k'_\qq)|^2|u_{i'}( k'_\qq)|^2}\le \sum_{k'\in\intZ}\frac{1}{|u_{i}(k_{\pp}) - u_{i'}( k')|^2|u_{i'}( k')|^2}
\end{equation}
We split the last sum into two parts. 
The first part contains all $k'$ satisfying $|u_i(k_\pp)-u_{i'}(k')|\ge |u_{i'}(k')|$. This part is bounded by, recalling the definition of $u_{i'}(k')$ in~\eqref{eq:uellndfn},
\begin{equation}
	\sum_{k'}\frac{1}{|u_{i'}(k')|^4}
	= 
	\sum_{k'}\frac{1}{(\ell p+2\ratio\rho_{i'})^2 +(4\pi k'-2\theta_{i'})^2}\le \sum_{k'\in\intZ}\frac{1}{1+(4\pi k'-2\theta_{i'})^2}
\end{equation}
which is uniformly bounded by a constant. 
The second part of the sum 
contains all $k'$ satisfying $|u_i(k_\pp)-u_{i'}(k')|< |u_{i'}(k')|$. 
Noting the fact that $|u_{i}(k_{\pp}) + u_{i'}( k')|\le |u_{i}(k_{\pp}) -u_{i'}( k')|+2|u_{i'}(k')|\le 3|u_{i'}(k')|$, this part is bounded by 
\begin{equation}
	\sum_{k'}\frac{|u_{i}(k_{\pp}) + u_{i'}( k')|^2}{|(u_{i}(k_{\pp}))^2 - (u_{i'}( k'))^2|^2|u_{i'}( k')|^2}
	\le \sum_{k'}\frac{9}{|(u_{i}(k_{\pp}))^2 - (u_{i'}( k'))^2|^2}
\end{equation}
which is uniformly bounded by a constant since $i\ne i'$ and
\begin{equation}
\begin{split}
	|(u_{i}(k_{\pp}))^2 - (u_{i'}( k'))^2|^2&=4\ratio^2(\rho_i-\rho_{i'})^2+(4\pi (k_\pp-k')+2(\theta_{i'}-\theta_i))^2\\
	&=4(i-i')^2+(4\pi (k_\pp-k')+2(\theta_{i'}-\theta_i))^2
\end{split}
\end{equation}
by our choices of $\rho_i$ and $\rho_{i'}$.
Combing the above two parts, we obtain 
\begin{equation}
	\sum_{\qq=1}^{n} \frac{ e^{- 2\epsilon \sqrt{\frac1{\ratio} |  k_{\qq}' |}} }{ |u_{i}(k_{\pp}) - u_{i'}( k'_\qq)|^2|u_{i'}( k'_\qq)|^2}\le C_2, 
\end{equation}
which implies that the bound~\eqref{eq:bnbnddtem2} changes to 
\beq \label{eq:bnbnddtem3}
	 \prod_{j=1}^n e^{ -\epsilon \sqrt{ \frac1{\ratio} | k_{j}|}  } \prod_{j=1}^{n'} e^{ - \epsilon \sqrt{ \frac1{\ratio} |\hat k_{j}'|} }
	\left[ \prod_{\pp=1}^{n} 
	\sqrt{ \frac{C_1}{\ell p} +  C_2} \right]
	 \left[ \prod_{\pp=1}^{n'} 
	 \sqrt{ \frac{C_1}{\ell p} +  C_2} \right] . 
\eeq
We thus obtain \eqref{eq:longdetra2}.
\end{proof}

Using the above bound instead of Lemma~\ref{lem:longdetra2}, the same proof shows that Corollary~\ref{resultRRhatbd} changes to the following. 

\begin{cor}
\label{resultRRhatbd1}
For every $\epsilon>0$, there is a positive constant $C_0$ such that 
\beq 
	| R_{\bfn}(\UU(\bfk), \UU(\hat \bfk))|
	\le C_0^{|\bfn|}
	\prod_{i=1}^m \prod_{j=1}^{n_i} 
	 e^{2\epsilon \sqrt{\frac1{\ratio} | k_{j}^{(i)}|}+ 2\epsilon \sqrt{ \frac1{\ratio} |\hat k_{j}^{(i)}|} }
\eeq
and
\beq
	| \hat R_{\bfn}(\UU(\bfk), \UU(\hat \bfk))|
	\le |\bfn| (\ell p)^{1/2} C_0^{|\bfn|}
	\prod_{i=1}^m \prod_{j=1}^{n_i} 
	 e^{ 2\epsilon \sqrt{ \frac1{\ratio} | k_{j}^{(i)}|}+ 2\epsilon \sqrt{ \frac1{\ratio} |\hat k_{j}^{(i)}|} }
\eeq 
for all $\bfn\in \N^m$, $\bfk, \hat \bfk \in \intZ^{\bfn}$, and $\ell, p>0$ satisfying  $\ell^3\ge 16\rho_1^4$  and $\ell p\ge 1$. 
\end{cor}

We are ready to prove Proposition~\ref{lm:error_estimate_largeL} for Case 1 assuming $p\gg \ell$ and $\log p\ll \ell^{3/2}$. 
We follow the same analysis of subsection~\ref{sec:prooferrors} except that we replace Corollary~\ref{resultRRhatbd} by Corollary~\ref{resultRRhatbd1} and the inequality~\eqref{eq:bound_bT} by \eqref{eq:bound_bT2}. Then the inequality~\eqref{eq:Stt} is replaced by
\beq \label{eq:Stt2}
	e^{ \frac{4}{3} \ell^{3/2}}  |\bs_{\bfn}(\bfk, \hat\bfk)| 
	\le 5 |\bfn| e^{4 |\bfn| } 
	C_0^{|\bfn|}
	e^{  -\frac{4\delta}{3} \ell^{3/2} -   c_2 |\bfn| \ell^{3/2} } 
	\prod_{i=1}^m \prod_{j=1}^{n_i} 
	e^{  - \frac{\cstar}{2} \sqrt{\frac{1}{\ratio } | k^{(i)}_j |}
	-  \frac{\cstar}{2}  \sqrt{\frac{1}{\ratio } | \hat k^{(i)}_j |}} 
\eeq
and the inequality~\eqref{eq:bDuhb} is changed to 
\beq \label{eq:bDuhb2} \begin{split}
	e^{ \frac{4}{3} \ell^{3/2}} \bD_{\bfn}(\bfz) 
	\le 5 |\bfn| 
	 \left( C_0 (\ell p)^2 \right)^{|\bfn|}
	e^{  -\frac{4\delta}{3} \ell^{3/2}  -   c_2 |\bfn| \ell^{3/2} } . 
\end{split} \eeq
For the bounds of $|\hbs_{\bfn}(\bfk, \hat\bfk)| $ and $\hbD_{\bfn}(\bfz)$, we only need to multiply the bounds of $|\bs_{\bfn}(\bfk, \hat\bfk)| $ and $\bD_{\bfn}(\bfz) $ by a factor $|\bfn|(\ell p)^{1/2}$ due to Corollary~\ref{resultRRhatbd1}.
Finally, using~\eqref{eq:bound_bT2}, the inequality~\eqref{eq:PPPane} changes to 
\beq \label{eq:PPPane3} \begin{split}
	|e^{ \frac{4}{3} \ell^{3/2}} \PP_{m,2} |& \le  
	e^{  -\frac{4\delta}{3} \ell^{3/2}} \sum_{\bfn \in \N^m\setminus\{\bfone\} }  \frac{|\bfn|}{(\bfn !)^2}  
	 \left( C_0 (\ell p)^2 \right)^{|\bfn|} e^{c'  |\bfn|}
	e^{ -   c_2  |\bfn| \ell^{3/2}} .
\end{split}\eeq 
The sum is uniformly bounded provided $p\ll e^{c_2\ell^3/2}$, which holds since $\log p\ll \ell^{3/2}$. 
This proves the first part of Proposition \ref{lm:error_estimate_largeL}. The proof of the second part on $\hat\PP_{m,2}$ is similar.

\section{Proof of Proposition \ref{propSrpr}} \label{sec:pflastprop}

We first prove Proposition \ref{propSrpr} (a). We have the following lemma. 

\begin{lemma} \label{resultprobasic}
Let $\bfa \in \realR^m$ satisfy $0<a_1<\cdots<a_{m-1}<a_m$ and let $\bfb \in \realR^m$.
Then, for every $r>0$, 
\beq \label{eq:prin1}
\begin{split}
	&\frac{(-1)^{m-1}\sqrt{2}}{(2\pi \ii)^m} \int\cdots \int 
	\prod_{i=2}^m \frac{1-e^{r(\xi_i-\xi_{i-1})}}{\xi_i-\xi_{i-1}} \prod_{i=1}^m e^{ \dDelta a_i \xi_i^2 - \dDelta b_i \xi_i} \dd \xi_i  \\
	&= \prob\left( \sqrt{2}\bm(a_1)-b_1 \in [0,r), \cdots, \sqrt{2}\bm(a_{m-1}) -b_{m-1}\in [0,r)  \mid \sqrt{2}\bm(a_m)=b_m \right) \phi_{a_m}\left(\frac{b_m}{\sqrt{2}}\right)
\end{split} \end{equation}
where the contours are distinct vertical lines oriented upwards, $\bm$ is a standard Brownian motion, and $\phi_t(x)=\frac{1}{\sqrt{2\pi t}}e^{-\frac{x^2}{2t}}$.
\end{lemma}

\begin{proof}
From Gaussian integrals, 
\beqq
 	\frac{1}{(2\pi \ii)^m } \int\cdots \int  \prod_{i=1}^m e^{ \dDelta a_i \xi_i^2 - \dDelta y_i \xi_i} \dd \xi_i
	=  \prod_{i=1}^m \frac{e^{-\frac{ (\dDelta y_i)^2}{4 \dDelta a_i}} }{\sqrt{2\pi  \dDelta a_i}} = \prod_{i=1}^m \frac{1}{\sqrt{2}}\phi_{\dDelta a_i}\left(\frac{\dDelta y_i}{\sqrt2}\right)
\eeqq
for every $\bfy\in \realR^m$. The right hand side is the joint density of $(\bm(a_1),\cdots,\bm(a_m))$ at $(y_1/\sqrt{2},\cdots,y_m/\sqrt{2})$.
The equation \eqref{eq:prin1} follows by integrating $y_i$ from $b_i$ to $b_i+r$ for $i=1, \cdots, m-1$ and taking $y_m=b_m$.  
\end{proof}

If the contours are ordered as $\Re(\xi_1)>\cdots >\Re(\xi_m)$, then 
the left hand-side of \eqref{eq:prin1} converges, as  $r\to +\infty$, to 
\beqq  
\begin{split}
	&\frac{(-1)^{m-1}\sqrt{2}}{(2\pi \ii)^m} \int\cdots \int 
	\prod_{i=2}^m \frac{1}{\xi_i-\xi_{i-1}} \prod_{i=1}^m e^{ \dDelta a_i \xi_i^2 - \dDelta b_i \xi_i} \dd \xi_i	\\
\end{split} \eeqq
Thus, Proposition \ref{propSrpr} (a) follows. This computation is due to \cite[Lemma 3.4]{Liu-Wang22}. 

\bigskip 

We now prove Proposition \ref{propSrpr} (b).

\begin{proof}[Proof of Proposition \ref{propSrpr} (b)]
Denote the left-side of \eqref{eq:Srpb} by 
\beq 
	A:= \frac{1}{(2\pi \ii)^m}\oint\cdots \oint S_{\ratio}(\bfa, \bfc-\bfb; \bfw) S_{\ratio}(\bfa, \bfc+\bfb; \bfw) 
	\prod_{i=2}^{m} \left( 1-\frac{w_{i-1}}{w_i} \right)
	\prod_{i=1}^m \frac{\rmd w_i}{ w_i} 
\eeq 
where the contours are circles satisfying $0<|w_1|<\cdots < |w_m|<1$ and (recall \eqref{eq:Srdefnt}) 
\beqq
	S_{\ratio}(\bfa,\bfb; \bfw) = \frac{(-1)^{m-1}\sqrt{2}}{\ratio^m} \sum_{\xi_1,\cdots,\xi_m} 
	\prod_{i=2}^m \frac{1}{\xi_i-\xi_{i-1}} \prod_{i=1}^m e^{ \dDelta a_i \xi_i^2 - \dDelta b_i \xi_i} .
\eeqq
Since the sum is over the points $\xi_i$ satisfying $e^{-\ratio\xi_i}=w_i$, we see that 
$\prod_{i=2}^{m} \left( 1-\frac{w_{i-1}}{w_i} \right) = \prod_{i=2}^{m} \left( 1- e^{\ratio(\xi_i-\xi_{i-1})} \right). $
Thus, we can write $A$ as
\beq \label{eq:ASrSr}
	A= \frac1{(2\pi \ii)^m}\oint\cdots \oint S_{\ratio}(\bfa, \bfc-\bfb; \bfw) T_{\ratio}(\bfa, \bfc+\bfb; \bfw) 
	\prod_{i=1}^m \frac{\rmd w_i}{ w_i}
\eeq
where
\beq
	T_{\ratio}(\bfa, \bfb; \bfw) 
	= \frac{(-1)^{m-1}\sqrt{2}}{\ratio^m} \sum_{\xi_1,\cdots,\xi_m} 
	\prod_{i=2}^m \frac{1- e^{ \ratio(\xi_i-\xi_{i-1})}}{\xi_i-\xi_{i-1}} \prod_{i=1}^m e^{ \dDelta a_i \xi_i^2 - \dDelta b_i \xi_i}.
\eeq

Let $0<|w|<1$. 
Note that if $f(\xi)$ is a function that is analytic in a vertical strip 
$ p -2\delta < \Re(\xi)<  p +2\delta$ for some $\delta>0$, where $p=-\frac{\log |w|}{\ratio}$, and decays fast as $\Im(\xi)\to \pm \infty$ in the strip, 
then by the Cauchy residue theorem, 
\beqq
	\sum_{\xi: e^{-\ratio \xi}=w} f(\xi) = \frac1{2\pi \ii} \int_{p +\delta -\ii \infty}^{p +\delta +\ii \infty} \frac{-\ratio w f(\xi)}{e^{-\ratio\xi}-w} \dd \xi
	- \frac1{2\pi \ii}  \int_{p -\delta -\ii \infty}^{p -\delta +\ii \infty}  \frac{-\ratio w f(\xi)}{e^{-\ratio\xi}-w} \dd \xi . 
\eeqq
Thus for such $f$, we find, using the geometric series and moving the contours, that
\beqq
	\sum_{\xi: e^{- \ratio \xi}=w} f(\xi) 
	= \frac{\ratio}{2\pi \ii} \sum_{k=-\infty}^\infty \frac1{w^k} \int_{p+\ii\realR}  f(\xi) e^{-k \ratio \xi} \dd \xi 
\eeqq
with the contour oriented upwards. 
Extending the above formula in a natural way, we find that 
\beq
	S_{\ratio}( \bfa,\bfb; \bfw) = \frac{(-1)^{m-1}\sqrt{2}}{(2\pi \ii)^m} 
	\sum_{\bfn \in \intZ^m}  \frac1{w_1^{n_1}\cdots w_m^{n_m}}
	\int\cdots \int 
	\prod_{i=2}^m \frac{(-1)^{m-1}}{\xi_i-\xi_{i-1}} \prod_{i=1}^m e^{ \dDelta a_i \xi_i^2 - \dDelta b_i \xi_i }e^{- n_i \ratio \xi_i} \dd \xi_i 
\eeq
and
\beq
	T_{\ratio}( \bfa,\bfb; \bfw) = \frac{(-1)^{m-1}\sqrt{2}}{(2\pi \ii)^m} 
	\sum_{\bfn \in \intZ^m}  \frac1{w_1^{n_1}\cdots w_m^{n_m}}
	\int\cdots \int 
	\prod_{i=2}^m \frac{1- e^{r(\xi_i-\xi_{i-1})}}{\xi_i-\xi_{i-1}} \prod_{i=1}^m e^{ \dDelta a_i \xi_i^2 - \dDelta b_i \xi_i} e^{ - n_i \ratio \xi_i} \dd \xi_i 
\eeq
where the contours are vertical lines, oriented upwards, satisfying $\Re(\xi_1)>\cdots >\Re(\xi_m)$. 
The ordering of the contours follows from $|w_1|<\cdots<|w_m|$.

Change the summation index $\bfn$ to $\bfk$ by setting $n_i=k_1+\cdots + k_i$ for $i=1, \cdots, m$ so that  $k_i=\dDelta n_i$ (where $n_0:=0$.) 
Using Lemma \ref{resultprobasic} with $b_i$ replaced by $b_i+\ratio k_i$ and $c_i$ replaced by $c_i+ \ratio k_i$, we find that
\beq
	S_{\ratio}(\bfa,\bfb; \bfw) = \sum_{\bfk \in \intZ^m}  
	\frac{\prob\left( \bigcap_{i=1}^{m-1} \{\sqrt{2}\bm'_1(a_i) - b_i \ge  \ratio k_i \} \mid \sqrt{2}\bm'_1(a_m) =b_m + \ratio k_m  \right) \phi_{a_m}(\frac{b_m+ \ratio k_m}{\sqrt{2}}) }{w_1^{\dDelta k_1}\cdots w_m^{\dDelta k_m}}
\eeq
and
\beq
	T_{\ratio}( \bfa,\bfb; \bfw) = \sum_{\bfk \in \intZ^m}  
	\frac{\prob\left( \bigcap_{i=1}^{m-1} \{ \sqrt{2}\bm'_2(a_i)-b_i \in [\ratio k_i, \ratio (k_i+1))\} \mid  \sqrt{2}\bm'_2(a_m) =c_m + \ratio k_m   \right) \phi_{a_m}(\frac{c_m + \ratio k_m}{\sqrt{2}})}{w_1^{\dDelta k_1}\cdots w_m^{\dDelta k_m}}
\eeq
where $\bm'_1$ and $\bm'_2$ are independent Brownian motions. 

Inserting the above formulas into \eqref{eq:ASrSr} and computing the integrals, we obtain 
\beq \begin{split} 
	A= &\sum_{\bfk \in \intZ^m}  
	 \prob\left( \bigcap_{i=1}^{m-1} \{\sqrt{2}\bm'_1(a_i)- c_i+b_i \ge  - \ratio k_i \} \, \bigg|\,  \sqrt{2}\bm'_1(a_m)- c_m+b_m =  - \ratio k_m \right) \phi_{a_m}(\frac{c_m-b_m -\ratio k_m}{\sqrt{2}})\\
	& \times \prob\left( \bigcap_{i=1}^{m-1} \{ \sqrt{2}\bm'_2(a_i)- c_i - b_i \in [ \ratio k_i, \ratio (k_i+1))\} \, \bigg|\,   \sqrt{2}\bm'_2(a_m) - c_m- b_m = \ratio k_m  \right) \phi_{a_m}(\frac{c_m+b_m + \ratio k_m}{\sqrt{2}}).
\end{split} \eeq
Using the independence of $B_1'$ and $B_2'$, and noting $\phi_t((x-y)/\sqrt{2})\phi_t((x+y)/\sqrt{2})=\phi_t(x)\phi_t(y)$, 
\beq \begin{split} 
	A= &\phi_{a_m}(c_m) \sum_{\bfk \in \intZ^m}  
	 \prob\left( \bigcap_{i=1}^{m-1} E_{i,k_i} \, \bigg|\,  G_{k_m} \right)  \phi_{a_m} (b_m+ \ratio k_m) 
\end{split} \eeq
where
\beqq
    E_{i,k}= \{\sqrt{2}\bm'_1(a_i)- c_i+b_i \ge  - \ratio k \} \cap \{ \sqrt{2}\bm'_2(a_i)- c_i - b_i \in [ \ratio k, \ratio (k+1))\}
\eeqq
and
\beqq
    G_{k_m}= \{\sqrt{2}\bm'_1(a_m)- c_m+b_m = - \ratio k_m\} \cap \{\sqrt{2}\bm'_2(a_m) - c_m- b_m = \ratio k_m\} .
\eeqq
For each $i$, $E_{i,k}$, $k\in \intZ$, are mutually disjoint events. Thus, taking the sums over $k_1, \cdots, k_{m-1}$, 
\beq \begin{split} 
	A= &\phi_{a_m}(c_m) \sum_{k_m\in \intZ}  
	 \prob\left( \bigcap_{i=1}^{m-1} E_{i} \, \bigg|\,  G_{k_m} \right)  \phi_{a_m} (b_m+ \ratio k_m) , \qquad E_i= \bigcup_{k\in \intZ} E_{i,k} . 
\end{split} \eeq
Lemma \ref{lm:sum_inequalities_identity} below implies that 
\beqq
    E_i= \left\{ \frac{\bm'_1(a_i) + \bm'_2(a_i)}{\sqrt{2}} - c_i \ge \dist_\ratio \left( \big\{ \frac{\bm'_2(a_i) - \bm'_1(a_i)}{\sqrt{2}} - b_i \big\}_\ratio, \{0\}_\ratio \right) 
    \right\} 
\eeqq
Setting $\bm_1 = (\bm'_2 -\bm'_1)/\sqrt{2}$ and $\bm_2  = (\bm'_1 + \bm'_2)/\sqrt{2}$, which are two independent Brownian motions, and noting the fact that $\dist_r(\{x-y\}_\ratio, \{0\}_\ratio)=\dist_r(\{x\}_\ratio, \{y\}_\ratio)$, the above equation becomes 
\begin{equation}
\begin{split}
    A&= \phi_{a_m}(c_m) \\
    &\times \sum_{k_m\in\intZ } \prob\left( \bigcap_{i=1}^{m-1}
    \{\bm_2(a_i) -\dist_\ratio(\{\bm_1(a_i)\}_\ratio,\{b_i\}_\ratio)\ge c_i \} \mid
    \bm_2(a_m)=c_m, \bm_1(a_m)=b_m +\ratio k_m
    \right)\phi_{a_m}(b_m+\ratio k_m).
\end{split}
\end{equation}

Now, note that for a Brownian motion $\bm(t)$, 
\begin{equation}
\prob\left(\bm(t)=x+\ratio n \mid \{\bm(t)\}_\ratio=\{x\}_\ratio\right) =\frac{\phi_t(x+\ratio n)}{\sum_{k\in\intZ} \phi_t(x+\ratio k)} = \frac{\phi_t(x+\ratio n)}{\phi_t^{(\ratio)}(\{x\}_\ratio)}
\end{equation}
where $\phi_t^{(\ratio)}(\{x\}_\ratio)$ is defined in \eqref{eq:def_phi_ratio}.
Thus, 
\begin{equation}
    A= \phi_{a_m}(c_m)  \prob\left( \bigcap_{i=1}^{m-1}
    \{\bm_2(a_i) -\dist_\ratio(\{\bm_1(a_i)\}_\ratio,\{b_i\}_\ratio)\ge c_i \} \mid
    \bm_2(a_m)=c_m, \{\bm_1(a_m)\}_\ratio=\{b_m\}_\ratio
    \right)\phi_t^{(\ratio)}(\{b_m\}_\ratio).
\end{equation}
This completes the proof of Proposition \ref{propSrpr} (b).
\end{proof}

\begin{lemma}
\label{lm:sum_inequalities_identity}
Recall that $\{x\}_\ratio$ denotes the equivalence class of the real number $x$ in the quotient space $\Sr=\realR/ \ratio \intZ$. 
Recall the distance $\dist_\ratio(\{x\}_\ratio,\{y\}_\ratio) = \min_{k\in\intZ} |x-y-\ratio k|$ between equivalence classes is defined in \eqref{eq:distratiodef}. 
For every two real-valued random variables $X$ and $Y$, 
\begin{equation}
\label{eq:lm_sum_events}
    \bigcup_{k=-\infty}^\infty \left\{X-Y \ge -\ratio k, \,\,  X+Y \in [\ratio k,\ratio (k+1))\right\} = \left\{X \ge \dist_\ratio(\{Y\}_\ratio,\{0\}_\ratio)\right\}.
\end{equation}  
\end{lemma}

\begin{proof}
Since both sides of the equation \eqref{eq:lm_sum_events} are unchanged if we replace $Y$ with $Y \pm \ratio$, we may assume, without loss of generality, 
that $-\ratio/2\le Y< \ratio/2$. Under this assumption, $\dist_\ratio(\{Y\}_\ratio,\{0\}_\ratio)=|Y|$.

Suppose that there exists a $k\in\intZ$ such that $X-Y\ge -\ratio k$ and $X+Y \in [\ratio k, \ratio(k+1))$. If $k=0$, then $X\ge Y$ and $X\ge -Y$. Thus $X\ge |Y|=\dist_\ratio(\{Y\}_\ratio,\{0\}_\ratio)$. If $|k|\ge 1$, we have $X\ge \max\{-\ratio k+Y, \ratio k -Y\} \ge \ratio/2 \ge   |Y|=\dist_\ratio(\{Y\}_\ratio,\{0\}_\ratio)$. 
Therefore, we find that the left hand side of \eqref{eq:lm_sum_events} is a sub-event of the right hand side.

Now we suppose that $X\ge \dist_\ratio(\{Y\}_\ratio,\{0\}_\ratio)=|Y|$. 
There is an integer $k$ such that $X+Y\in [\ratio k, \ratio(k+1))$. 
Note that $k\ge 0$ since $X+Y\ge 0$. Therefore $X-Y \ge 0 \ge -\ratio k$. This implies the right hand side of \eqref{eq:lm_sum_events} is a sub-event of the left hand side. 
Hence, the proof is complete. 
\end{proof}

\appendix

\section{Extension and continuity of the distribution functions $\tasepF_m$}\label{sec:appendix}

The limit result \eqref{eq:taseplimit} was proved in \cite{Baik-Liu19} for most but not all parameters. 
In this section, we first show that the convergence holds for all parameters. 
We then show that the limit functions are a consistent collection of multivariate cumulative distribution functions. 
We further show that they are continuous in all variables.

Let $\tasep(n,t)$ be the height function for the TASEP on the discrete ring of size $2a$ as in Section \ref{sec:pKPZfp}. 
For $T>0$, let 
\beqq
	\widetilde{\tasep}_T(\gamma, \tau):= \frac{\tasep(\gamma T^{2/3}, 2\tau T) - \tau T}{-T^{1/3}}, 
	\qquad (\gamma, \tau) \in \realR \times \realR_+
\eeqq
where the ring size is set as $(2a)^{3/2}=T$.\footnote{To be precise, we set $a= [T^{2/3}]/2$ since $a$ is a half-integer.} 
Let 
\beqq
	\realR_{+, \le}^m = \{ \tau= (\tau_1, \cdots, \tau_m)\in (0, \infty)^m : 0<\tau_1\le\cdots\le \tau_m\}. 
\eeqq
For $\tau \in \realR_{+, \le}^m$, define
\beqq
	\Omega_+^m(\tau) = \{ \hite=(\hite_1, \cdots, \hite_m)\in \realR^m : \text{$\hite_i< \hite_{i+1}$ if $\tau_i=\tau_{i+1}$}  \}. 
\eeqq
It was shown in \cite{Baik-Liu19} that for every $\gamma\in \realR^m$ and $\tau\in \realR^m_{+, \le}$, the limit 
\begin{equation}
\label{eq:taseplimit_app}
	\lim_{T\to \infty} \prob\bigg( 
            \bigcap_{i=1}^m 
            \big\{ \widetilde{\tasep}_T(\gamma_i, \tau_i) \le \hite_i  \big\}   \bigg)
	= \tasepF_m(\hite; \gamma, \tau)  \text{ converges if $\hite\in \Omega^m_+(\tau)$,} 
\end{equation}
as mentioned in \eqref{eq:proplusminus}.

When $m=1$, it was already shown in \cite{Baik-Liu18} that the one-point distribution $\tasepF_1$ is a distribution function and is continuous. 
We do not need the explicit form of $\tasepF_m$ for the first two results below. 

\medskip

We first show that the limit \eqref{eq:taseplimit_app} convergences for every $\hite\in \realR^m$. 
For $\tau \in \realR_{+, \le}^m$, define the set 
\beqq
	\Omega^m(\tau) := \overline{\Omega_+^m(\tau) } = \{ \hite=(\hite_1, \cdots, \hite_m)\in \realR^m : \text{$\hite_i\le  \hite_{i+1}$ if $\tau_i=\tau_{i+1}$}  \}. 
\eeqq

\begin{lemma} \label{prop:tasepF_extension}
Let $\gamma\in \realR^m$ and $\tau\in \realR^m_{+, \le}$. 
For every $\hat \hite\in \Omega^m(\tau)$, the limit 
\begin{equation} \label{eq:tasepF_aux02}
   	\lim_{\Omega^m(\tau)\ni  \hite \to  \hat \hite  } \tasepF_m(\hite; \gamma, \tau) 
\end{equation}
exists. Furthermore, if we denote the limit as $\tasepF_m(\hat \hite; \gamma, \tau)$, then 
\begin{equation}
\label{eq:tasepF_extension_prob}
	\lim_{T \to \infty} \prob\bigg( 
            \bigcap_{i=1}^m 
            \big\{  \widetilde{\tasep}_T(\gamma_i, \tau_i) \le \hat \hite_i  \big\}  \bigg)
	= \tasepF_m(\hat\hite, \gamma, \tau ).
\end{equation}
\end{lemma}

\begin{proof}
For $\epsilon>0$, let  
\beqq
	\hite'_i 
	=\hat\hite_i - (2-\frac{i}{m+1} ) \epsilon
	\quad \text{and} \quad 
	\hite''_i 
	=\hat\hite_i + (1+\frac{i}{m+1}) \epsilon
\eeqq
for $i=1, \cdots, m$. 
Then, $\hite', \hite'' \in \Omega^m_+(\tau)$.  
Furthermore, $(\hite'_1, \cdots, \hite_k', \hite_{k+1}'', \cdots, \hite_m)\in \Omega^m_+(\tau)$ for every $k$.

Let $\beta\in \Omega^m_+(\tau)$ be an arbitrary number satisfying $\sum_{i=1}^m|\hite_i-\hat\hite_i|<\epsilon$. 
Note that $\hite'_i \le \hite_i\le \hite''_i$. Thus, 
\beqq
	\tasepF_m(\hite'; \gamma, \tau) \le \tasepF_m(\hite ; \gamma, \tau) \le \tasepF_m(\hite'' ; \gamma, \tau)  
\eeqq
where we used the fact that being a limit of a distribution function, $\tasepF_m(\hite; \gamma, \tau) $ 
is a weakly increasing function of $\hite\in \Omega_+^m(\tau)$. 
From the monotonicity property again, as $\epsilon\downarrow 0$, $\tasepF_m(\hite'; \gamma, \tau)$  increases weakly and $\tasepF_m(\hite'' ; \gamma, \tau)$  decreases weakly. 
Therefore, the limit~\eqref{eq:tasepF_aux02} converges if we show that 
\begin{equation}
\label{eq:tasepF_bounds_diff}
	\lim_{\epsilon \downarrow 0} \tasepF_m(\hite'' ; \gamma, \tau)  - \tasepF_m(\hite' ; \gamma, \tau)  = 0. 
\end{equation}

For every $j$, from \eqref{eq:taseplimit_app},  
\beqq
\begin{split}
	&\tasepF_m(\hite'_1,\cdots,\hite'_{j-1},\hite''_{j},\hite''_{j+1},\cdots,\hite''_m; \gamma, \tau )
	 -\tasepF_m(\hite'_1,\cdots,\hite'_{j-1},\hite'_{j},\hite''_{j+1},\cdots,\hite''_m; \gamma, \tau )\\
	&= \lim_{T\to \infty} \prob \bigg( 
            \big\{
            \hite'_{j}<
            \widetilde{\tasep}_T(\gamma_{j} , \tau_j) \le  \hite''_{j}   \big\}
            \bigcap_{i=1}^{j-1}
            \big\{ 
             \widetilde{\tasep}_T(\gamma_{j} , \tau_j) \le \hite'_i 
           \big\} \bigcap_{i=j+1}^m
           \big\{ 
             \widetilde{\tasep}_T(\gamma_{j} , \tau_j)\le \hite''_{i} 
           \big\}
           \bigg)\\
    	&\le \lim_{T \to \infty} \prob\big( 
              \hite'_{j}<
              \widetilde{\tasep}_T(\gamma_{j} , \tau_j) \le \hite''_{j}  \big) 
           =\tasepF_1(\hite''_{j}; \gamma_{j},\tau_{j}) -\tasepF_1(\hite'_{j}; \gamma_{j},\tau_{j}).
\end{split}
\eeqq 
Summing over $j$, we obtain
\beqq
	\tasepF_m(\hite''_1,\cdots,\hite''_m; \gamma, \tau) 
	-\tasepF_m(\hite'_1,\cdots,\hite'_m; \gamma, \tau) \le \sum_{j=1}^{m} \left(\tasepF_1(\hite''_{j}; \gamma_{j},\tau_{j}) -\tasepF_1(\hite'_{j} \gamma_{j},\tau_{j})\right).
\eeqq 
Since the one-point distribution $\tasepF_1$ is continuous (see \cite{Baik-Liu18}), the right side converges to zero as $\epsilon\to 0$. The left-hand side is also nonnegative due to the monotonicity property of $\tasepF_m$. Thus we obtain \eqref{eq:tasepF_bounds_diff}, which implies the convergence of \eqref{eq:tasepF_aux02}. 

With the same notations as above, from the monotonicity of probabilities which holds for the parameters without any restrictions, 
\beqq
	\prob\bigg( 
            \bigcap_{i=1}^m 
            \big\{  \widetilde{\tasep}_T(\gamma_i, \tau_i) \le  \hite_i'  \big\}  \bigg) 
            \le \prob\bigg( 
            \bigcap_{i=1}^m 
            \big\{  \widetilde{\tasep}_T(\gamma_i, \tau_i) \le \hat \hite_i  \big\}  \bigg)
	\le 
	\prob\bigg( 
            \bigcap_{i=1}^m 
            \big\{  \widetilde{\tasep}_T(\gamma_i, \tau_i) \le  \hite_i''  \big\}  \bigg)
\eeqq
As $T\to \infty$, the lower bound tends to $\tasepF_m(\hite'; \gamma, \tau)$ and the upper bound tends to $\tasepF_m(\hite''; \gamma, \tau)$. 
If we let $\epsilon \downarrow 0$, then both of them converge to $\tasepF_m(\hat \hite; \gamma, \tau)$. This shows \eqref{eq:tasepF_extension_prob}. 
\end{proof}

\begin{cor}
For every $\gamma\in \realR^m$,  $\tau\in \realR_+^m$, and $\hite\in \realR^m$, the limit
\begin{equation}
\label{eq:tasepF_general_def}
	\tasepF_m(\hite; \gamma, \tau) :=
	\lim_{T \to \infty} \prob\bigg( \bigcap_{i=1}^m \big\{  \widetilde{\tasep}_T(\gamma_i, \tau_i) \le \hite_i  \big\}  \bigg)
\end{equation}
converges. The function $\tasepF_m(\hite; \gamma, \tau)$ is invariant under the permutations of the triples $(\hite_i,\gamma_i,\tau_i)$, $i=1, \cdots, m$. 
\end{cor}

\begin{proof}
The probability $\prob( \cap_{i=1}^m \big\{  \widetilde{\tasep}_T(\gamma_i, \tau_i) \le \hite_i  \}  )$ is defined for every $(\gamma, \tau, \hite)\in \realR^m\times \realR^m_+\times \realR^m$, and is invariant under the permutations of the triples $(\hite_i,\gamma_i,\tau_i)$. 
For a permutation $\sigma \in S_m$, let $\gamma^\sigma, \tau^\sigma, \hite^\sigma$ be the parameters obtained from $\gamma, \tau, \hite$ applying the permutation $\sigma$ to the index. 
Let $\sigma$ be a permutation so that $\tau^\sigma\in \realR^m_{+, \le}$ and $\hite^\sigma\in \Omega^m_0(\tau^\sigma)$. 
There is at least one such permutation. 
By the last lemma, 
$\prob( \cap_{i=1}^m \big\{  \widetilde{\tasep}_T(\gamma_i, \tau_i) \le \hite_i  \}  ) = \prob( \cap_{i=1}^m \big\{  \widetilde{\tasep}_T(\gamma_i^\sigma, \tau_i^\sigma) \le \hite_i^\sigma  \}  )$ converges. 
If there are more than one permutation with the same property, then it is easy to check that they results in the same limit. 
The invariance under permutations follow easily. 
\end{proof}

Therefore, the convergence \eqref{eq:taseplimit} holds for all parameters, and we use the same notation $\tasepF_m(\hite; \gamma, \tau)$ for the limit.  
We now show that $\tasepF_m(\hite; \gamma, \tau)$ are a collection of consistent multivariate cumulative distribution functions. 
For the restricted parameters, this fact was proved in  \cite[Section 7]{Baik-Liu19}. Here, we prove it for all parameters.

\begin{prop}\label{prop:tasepF_consistence} 
\begin{enumerate}[(a)]
\item For every $m$ and $(\gamma, \tau)\in \realR^m\times \realR^m_+$, $\hite\mapsto \tasepF_m(\hite; \gamma, \tau)$ is a multivariate cumulative distribution function.  
\item 
Let $(\gamma, \tau, \hite)\in \realR^m\times \realR_+^m \times \realR^m$. 
For each $j=1, \cdots, m$, let $\gamma^{(j)}, \tau^{(j)}, \hite^{(j)}$ be the points in $\realR^{m-1}$ obtained from $\gamma, \tau, \hite$ by removing $\gamma_j, \tau_j, \hite_j$, respectively. Then, 
\beqq
	\lim_{\hite_j\to\infty} \tasepF_m(\hite; \gamma, \tau)
	= \tasepF_{m-1}(\hite^{(j)}; \gamma^{(j)}, \tau^{(j)})  .
\eeqq
\end{enumerate}
\end{prop}

\begin{proof}
(a) The equation \eqref{eq:tasepF_general_def} implies the monotone non-decreasing property. 
It also implies that $\tasepF_m(\hite; \gamma, \tau)\le \tasepF_1(\hite_j; \gamma_j,\tau_j)$ and $1-\tasepF_m(\hite; \gamma, \tau) \le \sum_{j=1}^{m} (1- \tasepF_1(\hite_j; \gamma_j,\tau_j))$. We thus find the correct limit properties as $\hite$ becomes small or large.

(b) From \eqref{eq:tasepF_general_def} again, 
\beqq \begin{split}
	0\le 
	 \tasepF_{m-1}(\hite^{(j)}; \gamma^{(j)}, \tau^{(j)}) -  \tasepF_m(\hite; \gamma, \tau)
	&= \lim_{T\to \infty} \prob\bigg( 
            \big\{ \widetilde{\tasep}_T(\gamma_j , \tau_j ) > \hite_j \big\}
            \bigcap_{\substack{1\le i\le m\\ i\ne j}} 
            \big\{  \widetilde{\tasep}_T(\gamma_j , \tau_j ) \le \hite_i   \big\}  \bigg) \\
         &\le  \lim_{T \to \infty} \prob \big( 
             \widetilde{\tasep}_T(\gamma_j , \tau_j ) > \hite_j \big)
            = 1-\tasepF_1(\hite_j; \gamma_j,\tau_j) . 
\end{split} \eeqq
The upper bound tends to $0$ as $\hite_j\to +\infty$ since $\tasepF_1$ is a distribution function \cite{Baik-Liu18}. 
\end{proof}

\medskip

The final result of this Section is the continuity of $\tasepF_m(\hite; \gamma, \tau)$. 
When $\hite\in \Omega^m_+(\tau)$, the limit $\tasepF_m(\hite; \gamma, \tau)$ for \eqref{eq:taseplimit_app} is given by the formula 
\begin{equation}
\label{eq:def_tasepF}
\begin{split}
	&\tasepF_m(\hite; \gamma, \tau)
	=\frac{1}{(2\pi \ii)^m } \oint \cdots \oint C(\bfz)D(\bfz) \prod_{i=1}^m \frac{\dd z_i}{z_i}
\end{split}
\end{equation}
where the integrand is same as that of the formula \eqref{eq:jointprointe} (with $p=1$) but the radii of the contour circles satisfy the reverse inequalities $0<|z_m|<\cdots<|z_1|<1$.  
From the formula of $C(\bfz)$ and $D(\bfz) $ in Section \ref{sec:formdf} (with $p=1$) and Lemma~\ref{prop:tasepF_extension}, 
$\tasepF_m(\hite; \gamma, \tau)$ is jointly continuous for $\gamma\in\realR^m$, $\tau \in \realR_{+, \le}^m$, and $\hite\in\Omega^m(\tau) $. 
Due to the invariance under permutations of the triples of the parameters, it is continuous on the set 
\beqq
	U_+^m:= 
	\realR^m\times \realR^m_+\times \realR^m 
	\setminus \{(\gamma, \tau, \hite): \hite_i=\hite_j \text { and } \tau_i=\tau_j \text{ for some } 1\le i<j\le m\}. 
\eeqq 
The next result shows that it is continuous in all of $\realR^m\times \realR^m_+\times \realR^m$. 

\begin{prop} \label{prop:tasepF_continuity}
The function $\tasepF_m(\hite; \gamma, \tau)$ is jointly continuous in $(\gamma, \tau, \hite)\in \realR^m\times \realR^m_+\times \realR^m$. 
\end{prop}

\begin{proof}
Let $(\gamma, \tau, \hite)$ be a point in $\realR^m\times \realR^m_+\times \realR^m$. 
For $(\gamma', \tau', \hite')\in \realR^m\times \realR^m_+\times \realR^m$ and $1\le j\le m$, let 
\beqq
	\gamma'_{(j)} = (\cdots, \gamma'_{j-1}, \tau'_j, \gamma_{j+1}, \cdots), 
	\qquad 
	\tau'_{(j)}= (\cdots, \tau'_{j-1}, \tau'_j, \tau_{j+1}, \cdots), 
	\qquad 
	 \hite'_{(j)} = (\cdots, \hite'_{j-1}, \hite'_j, \hite_{j+1}, \cdots). 
\eeqq
Fom \eqref{eq:tasepF_general_def}, for every $j$, 
\beqq
\begin{split}
	&  | \tasepF_m( \hite'_{(j)};  \gamma'_{(j)},  \tau'_{(j)}) - \tasepF_m( \hite'_{(j-1)};  \gamma'_{(j-1)},  \tau'_{(j-1)})  | \\
	&\le \lim_{T \to \infty} \prob \big( \widetilde{\tasep}_T(\gamma_j' , \tau_j' ) \le \hite'_j  ,  \,  \widetilde{\tasep}_T(\gamma_j , \tau_j ) > \hite_j \big) 
           + \prob \big(\widetilde{\tasep}_T(\gamma_j , \tau_j ) \le \hite_j  , \,  \widetilde{\tasep}_T(\gamma_j' , \tau_j') > \hite'_j  \big) \\
	&= \lim_{T \to \infty} \prob \big(    \widetilde{\tasep}_T(\gamma_j' , \tau_j' ) \le \hite'_j  \big) 
           + \prob \big(  \widetilde{\tasep}_T(\gamma_j , \tau_j ) \le \hite_j  \big)
           - 2   \prob \big( \widetilde{\tasep}_T(\gamma_j' , \tau_j' )\le \hite'_j, \,  \widetilde{\tasep}_T(\gamma_j , \tau_j ) \le \hite_j  \big) . 
\end{split}
\eeqq
Thus, 
\beqq
\begin{split}
  	| \tasepF_m( \hite'_{(j)};  \gamma'_{(j)},  \tau'_{(j)}) - \tasepF_m( \hite'_{(j-1)};  \gamma'_{(j-1)},  \tau'_{(j-1)})  | 
	&=\tasepF_1(\hite'_j; \gamma'_j,\tau'_j) + \tasepF_1(\hite_j; \gamma_j,\tau_j) 
	-2\tasepF_2(\hite'_j,\hite_j; \gamma'_j, \gamma_j,\tau'_j, \tau_j) \\
	&\le \tasepF_1(\hite'_j; \gamma'_j,\tau'_j) + \tasepF_1(\hite_j; \gamma_j,\tau_j) 
	-2\tasepF_2(\hite'_j,\hite_j - \epsilon; \gamma'_j, \gamma_j,\tau'_j, \tau_j)
\end{split}
\eeqq
for every $\epsilon>0$. 
If $(\gamma'_j, \tau'_j, \hite'_j)$ is close enough to $(\gamma_j, \tau_j, \hite_j)$, then $(\gamma_j', \gamma_j, \tau_j', \tau_j, \hite_j', \hite_j-\epsilon)\in U^2_+$. 
Thus, the continuity of $\tasepF_2$ on $U^2_+$ implies that 
\beqq
\begin{split}
	&\limsup_{(\gamma', \tau', \hite')\to (\gamma, \tau, \hite)} | \tasepF_m( \hite'_{(j)};  \gamma'_{(j)},  \tau'_{(j)}) - \tasepF_m( \hite'_{(j-1)};  \gamma'_{(j-1)},  \tau'_{(j-1)})  |\\
	&\quad \le 2\tasepF_1(\hite_j; \gamma_j,\tau_j) 
	-2\tasepF_2(\hite_j,\hite_j - \epsilon; \gamma_j, \gamma_j,\tau_j, \tau_j)  
	= 2\tasepF_1(\hite_j; \gamma_j,\tau_j) - 2\tasepF_1(\hite_j-\epsilon; \gamma_j,\tau_j) 
\end{split}
\eeqq 
where we used the fact that $\tasepF_2(a,b; \gamma,\gamma, \tau, \tau)=\tasepF_1(a; \gamma,\tau)$ if $a<b$, which follows from the definition of~\eqref{eq:tasepF_general_def}. 
Since the inequality holds for every $\epsilon>0$ and the one-point distribution function $\tasepF_1$ is continuous, we find that
\beqq
\begin{split}
	&\limsup_{(\gamma', \tau', \hite')\to (\gamma, \tau, \hite)} | \tasepF_m( \hite'_{(j)};  \gamma'_{(j)},  \tau'_{(j)}) - \tasepF_m( \hite'_{(j-1)};  \gamma'_{(j-1)},  \tau'_{(j-1)})  |
	= 0 . 
\end{split}
\eeqq 
Summing over $j$, we conclude that 
\beqq
	\limsup_{(\gamma', \tau', \hite')\to (\gamma, \tau, \hite)} \left|\tasepF_m(\hite'; \gamma', \tau') - \tasepF_m(\hite; \gamma, \tau)  \right|
	= 0, 
\eeqq 
proving the desired continuity. 
\end{proof}

\section{Formula of $D_{\bfn}(z)$} \label{sec:Dnformula}

We state the formula of $D(\bfz)$ given in \cite[Lemma 2.10]{Baik-Liu19} and show that it can be written as the form \eqref{eq:def_Dbfz} in Subsection \ref{sec:formdf}. 
It is enough to check it  when $p=1$ since the general $p$ case follows from the property~\eqref{eq:scalingprop}.

For complex vectors $W=(w_1, \cdots, w_n)$ and $W'=(w_1',\cdots, w_{n'}')$, we denote
\beqq
	\Delta(W)= \prod_{1\le i<j\le n} (w_j-w_i) \quad \text{and} \quad \Delta(W;W')= \prod_{i=1}^n \prod_{j=1}^{n'} (w_i-w_j') .
\eeqq
We also use the notation that for a function $g$ of a single variable and a vector $W=(w_1, \cdots, w_n)$, 
\beqq
	g(W) := \prod_{i=1}^n g(w_i) .
\eeqq

For $0<|z|<1$, the sets $\rmL_z$ and $\rmR_z$ are the discrete sets in the complex plane defined as
\beqq
	\rmL_z=\{w: e^{-w^2/2}=z, \, \Re(w)<0\} \quad \text{and} \quad \rmR_z=\{w: e^{-w^2/2}=z, \, \Re(w)>0\}. 
\eeqq
The series formula of $D(\bfz)$ given in \cite[Lemma 2.10]{Baik-Liu19} is 
\beqq
	D(\bfz) = \sum_{\bfn \in \{0, 1, \cdots\}^m } \frac{1}{(\bfn !)^2} D_{\bfn}(\bfz)
\eeqq
where for $\bfn=(n_1, \cdots, n_m)$ and $0<|z_1|<\cdots<|z_m|<1$, 
\beq \label{eq:D_bfzorig} \begin{split}
	D_{\bfn}(\bfz) 
	= \left(1-\frac{z_{i-1}}{z_i} \right)^{n_i}
	&\left(1-\frac{z_{i}}{z_{i-1}} \right)^{n_{i-1}}
	\sum_{\substack{U^{(i)}\in \rmL_{z_i}^{n_i}\\ V^{(i)}\in \rmR_{z_i}^{n_i} \\ i=1,\cdots,m}} 
	\prod_{i=1}^m \frac{\Delta(U^{(i)})^2 \Delta(V^{(i)})^2}{\Delta(U^{(i)};V^{(i)})^2} \hat f_i (U^{(i)}) \hat f_i (V^{(i)})\\
	&\qquad \times \prod_{i=2}^m \frac{\Delta(U^{(i)};V^{(i-1)})\Delta(V^{(i)};U^{(i-1)})}{\Delta(U^{(i)};U^{(i-1)})\Delta(V^{(i)};V^{(i-1)})}
	\frac{e^{-h(V^{(i)},z_{i-1})-h(V^{(i-1)},z_{i})}}{e^{h(U^{(i)},z_{i-1})+h(U^{(i-1)},z_{i})}}
\end{split} \eeq
with 
\beqq \begin{split}
	&h(w,z) =\begin{dcases}
-\frac{1}{\sqrt{2\pi}} \int_{-\infty}^w \Li_{1/2}(ze^{(w^2-y^2)/2})\rmd y, \qquad &\Re(w)<0,\\
-\frac{1}{\sqrt{2\pi}} \int_{-\infty}^{-w} \Li_{1/2}(ze^{(w^2-y^2)/2})\rmd y, \qquad &\Re(w)>0, 
\end{dcases}
\end{split}\eeqq 
\beqq \label{eq:ffuntion}
\begin{split}
	&f_i(w)
=\begin{dcases}
e^{-\frac{1}{3}(\tau_i-\tau_{i-1})w^3 +\frac{1}{2}(\gamma_i-\gamma_{i-1})w^2 + (\hite_i-\hite_{i-1})w},
	\qquad & \Re(w)<0,\\
e^{\frac{1}{3}(\tau_i-\tau_{i-1})w^3 -\frac{1}{2}(\gamma_i-\gamma_{i-1})w^2 - (\hite_i-\hite_{i-1})w},
	\qquad & \Re(w)>0,
\end{dcases}
\end{split}\eeqq 
and 
\beqq
	\hat f_i(w) = \frac{1}{w} f_i(w) e^{2h(w,z_i)}.
\eeqq

Note that $w\in \rmR_z$ if and only if $-w\in \rmL_z$. Thus, setting $\hat U^{(i)}= - V^{(i)}$, the sum in \eqref{eq:D_bfzorig} can be written as 
\beqq \begin{split}
	&
	\sum_{\substack{U^{(i)}, \hat U^{(i)} \in \rmL_{z_i}^{n_i} \\ i=1,\cdots,m}} 
	\prod_{i=1}^m \frac{\Delta(U^{(i)})^2 \Delta(- \hat U^{(i)})^2}{\Delta(U^{(i)};-\hat U^{(i)})^2} \hat f_i (U^{(i)}) \hat f_i (-\hat U^{(i)})\\
	&\qquad\qquad \times  \prod_{i=2}^m \frac{\Delta(U^{(i)}; -\hat U^{(i-1)})\Delta( -\hat U^{(i)};U^{(i-1)})}{\Delta(U^{(i)};U^{(i-1)})\Delta( -\hat U^{(i)}; -\hat U^{(i-1)})}
	\frac{e^{-h(-\hat U^{(i)},z_{i-1})-h(-\hat U^{(i-1)},z_{i})}}{e^{h(U^{(i)},z_{i-1})+h(U^{(i-1)},z_{i})}} .
\end{split} \eeqq
Since $h(-w,z)=h(w,z)$, after inserting the formula $\hat f_i(w) = \frac{1}{w} f_i(w) e^{2h(w,z_i)}$ and using the notation $E^{i, \pm}$ of \eqref{eq:Edeffff} instead of $f_i$, we can express the above sum as 
\beqq \begin{split}
	&
	\sum_{\substack{U^{(i)}, \hat U^{(i)} \in \rmL_{z_i}^{n_i} \\ i=1,\cdots,m}} 
	\prod_{i=1}^m 
	\frac{e^{2h(\hat U^{(i)},z_{i}) +2h(\hat U^{(i)},z_{i})}}{e^{h(U^{(i-1)},z_{i})+h(U^{(i+1)},z_{i})+ h(\hat U^{(i-1)},z_{i})+h(\hat U^{(i+1)},z_{i})}} 
	\prod_{i=1}^m \prod_{j=1}^{n_i} E^{i, +}(u^{(i)}_j) E^{i, -}( \hat u^{(i)}_j) \\
	&\qquad \qquad \times 
	\prod_{i=1}^m \prod_{j=1}^{n_i} \frac{-1}{u^{(i)}_j \hat u^{(i)}_j}  
	\prod_{i=1}^m \frac{\Delta(U^{(i)})^2 \Delta(- \hat U^{(i)})^2}{\Delta(U^{(i)};-\hat U^{(i)})^2} 
	\prod_{i=2}^m \frac{\Delta(U^{(i)}; -\hat U^{(i-1)})\Delta( -\hat U^{(i)};U^{(i-1)})}{\Delta(U^{(i)};U^{(i-1)})\Delta( -\hat U^{(i)}; -\hat U^{(i-1)})}
\end{split} \eeqq
where we set $U^{(0)}= \hat U^{(0)}= U^{(m+1)} = \hat U^{(m+1)}=\emptyset$ so that $e^{h(U^{(0)}; z_1)}=1$, and so on. 
The product involving the function $h$ is $H_{\bfn}(U, \hat U)$ of \eqref{eq:Hdefff} 
and the next product involving $E^{i, \pm}$ is $E_{\bfn}(U, \hat U)$ of \eqref{eq:Edeffff}. 
Finally, again with the convention $U^{(0)}= \hat U^{(0)}= U^{(m+1)} = \hat U^{(m+1)}=\emptyset$, we have 
\beqq \begin{split}
	&\prod_{i=1}^m \frac{\Delta(U^{(i)})^2 \Delta(- \hat U^{(i)})^2}{\Delta(U^{(i)};-\hat U^{(i)})^2} 
	\prod_{i=2}^m \frac{\Delta(U^{(i)}; -\hat U^{(i-1)})\Delta( -\hat U^{(i)};U^{(i-1)})}{\Delta(U^{(i)};U^{(i-1)})\Delta( -\hat U^{(i)}; -\hat U^{(i-1)})} \\
	&= 
	\prod_{i=1}^{m+1} \frac{\Delta(U^{(i-1)}) \Delta(- \hat U^{(i-1)}) \Delta(U^{(i)}; -\hat U^{(i-1)})\Delta( -\hat U^{(i)};U^{(i-1)}) \Delta(U^{(i)}) \Delta(- \hat U^{(i)})}{ \Delta(U^{(i-1)};-\hat U^{(i-1)}) \Delta(U^{(i)};U^{(i-1)})\Delta( -\hat U^{(i)}; -\hat U^{(i-1)}) \Delta(U^{(i)};-\hat U^{(i)}) } \\
	&= 
 (-1)^{n_1+\cdots+n_m}\prod_{i=1}^{m+1} \Cd( U^{(i-1)}, -\hat U^{(i)} ; \hat U^{(i-1)}, - U^{(i)}) 
\end{split} \eeqq
in terms of the Cauchy determinant \eqref{eq:Caudedef}. 
This is a factor of $R_{\bfn}(U, \hat U)$ of \eqref{eq:Rdetfom}, and we thus find that $D_{\bfn}(\bfz)$ is equal to the form \eqref{eq:def_Dbfz} 
in Subsection \ref{sec:formdf}.


\def\cydot{\leavevmode\raise.4ex\hbox{.}}


\begin{thebibliography}{10}

\bibitem{Baik-Liu18}
J.~Baik and Z.~Liu.
\newblock Fluctuations of {TASEP} on a ring in relaxation time scale.
\newblock {\em Comm. Pure Appl. Math.}, 71(4):747--813, 2018.

\bibitem{Baik-Liu19}
J.~Baik and Z.~Liu.
\newblock Multipoint distribution of periodic {TASEP}.
\newblock {\em J. Amer. Math. Soc.}, 32(3):609--674, 2019.

\bibitem{Baik-Liu21}
J.~Baik and Z.~Liu.
\newblock Periodic {TASEP} with general initial conditions.
\newblock {\em Probab. Theory Related Fields}, 179(3-4):1047--1144, 2021.

\bibitem{Baik-Liu-Silva22}
J.~Baik, Z.~Liu, and G.~L.~F. Silva.
\newblock Limiting one-point distribution of periodic {TASEP}.
\newblock {\em Ann. Inst. Henri Poincar\'{e} Probab. Stat.}, 58(1):248--302,
  2022.

\bibitem{Baik-Prokhorov-Silva23}
J.~Baik, A.~Prokhorov, and G.~L.~F. Silva.
\newblock Differential equations for the {KPZ} and periodic {KPZ} fixed points.
\newblock {\em Comm. Math. Phys.}, 401(2):1753--1806, 2023.

\bibitem{Basu-Ganguly21}
R.~Basu and S.~Ganguly.
\newblock Time correlation exponents in last passage percolation.
\newblock In {\em In and out of equilibrium 3. {C}elebrating {V}ladas
  {S}idoravicius}, volume~77 of {\em Progr. Probab.}, pages 101--123.
  Birkh\"{a}user/Springer, Cham, 2021.

\bibitem{Ferrari-Occelli18}
P.~L. Ferrari and A.~Occelli.
\newblock {Universality of the GOE Tracy-Widom distribution for TASEP with
  arbitrary particle density}.
\newblock {\em Electronic Journal of Probability}, 23:1--24, 2018.

\bibitem{Ganguly-Hedge22}
S.~Ganguly and M.~Hegde.
\newblock Sharp upper tail estimates and limit shapes for the {KPZ} equation
  via the tangent method, 2022.
\newblock arXiv:2208.08922.

\bibitem{Ganguly-Hedge-Zhang23}
S.~Ganguly, M.~Hegde, and L.~Zhang.
\newblock Brownian bridge limit of path measures in the upper tail of {KPZ}
  models, 2023.
\newblock arXiv:2311.12009.

\bibitem{Lamarre-Lin-Tsai23}
P.~Y. Gaudreau~Lamarre, Y.~Lin, and L.-C. Tsai.
\newblock K{PZ} equation with a small noise, deep upper tail and limit shape.
\newblock {\em Probab. Theory Related Fields}, 185(3-4):885--920, 2023.

\bibitem{Johansson20}
K.~Johansson.
\newblock Long and short time asymptotics of the two-time distribution in local
  random growth.
\newblock {\em Math. Phys. Anal. Geom.}, 23(4):Paper No. 43, 34, 2020.

\bibitem{Li-Saenz23}
J.-H. Li and A.~Saenz.
\newblock Contour integral formulas for {PushASEP} on the ring, 2023.
\newblock arXiv:2308.05372.

\bibitem{Liao22}
Y.~Liao.
\newblock Multi-point distribution of discrete time periodic {TASEP}.
\newblock {\em Probab. Theory Related Fields}, 182(3-4):1053--1131, 2022.

\bibitem{Liu22c}
Z.~Liu.
\newblock One-point distribution of the geodesic in directed last passage
  percolation.
\newblock {\em Probab. Theory Related Fields}, 2022.

\bibitem{Liu22b}
Z.~Liu.
\newblock When the geodesic becomes rigid in the directed landscape.
\newblock {\em Electronic Communications in Probability}, 27:1--13,
  2022.

\bibitem{Liu-Wang22}
Z.~Liu and Y.~Wang.
\newblock {A conditional scaling limit of the KPZ fixed point with height tending to infinity at one location}.
\newblock {\em Electronic Journal of Probability}, 29:1--27, 2024.

\bibitem{Matetski-Quastel-Remenik21}
K.~Matetski, J.~Quastel, and D.~Remenik.
\newblock The {KPZ} fixed point.
\newblock {\em Acta Math.}, 227(1):115--203, 2021.

\bibitem{Nissim-Zhang22}
R.~Nissim and R.~Zhang.
\newblock Asymptotics of the one-point distribution of the {KPZ} fixed point
  conditioned on a large height at an earlier point, 2022.
\newblock arXiv:2210.04999.

\bibitem{Quastel-Remenik22}
J.~Quastel and D.~Remenik.
\newblock K{P} governs random growth off a 1-dimensional substrate.
\newblock {\em Forum Math. Pi}, 10:Paper No. e10, 26, 2022.

\end{thebibliography}
\end{document}